\documentclass[11pt]{article}
\usepackage{amssymb, authblk}
\usepackage{amsmath,bbm}
\usepackage{fullpage}
\usepackage{amsthm, hyperref}
\usepackage{graphicx}
\usepackage{verbatim}
\usepackage{algorithm}
\usepackage{algpseudocode}
\usepackage[dvipsnames]{xcolor}
\usepackage{tikz}
\usetikzlibrary{arrows.meta}
\usetikzlibrary{decorations.pathreplacing}

\usepackage{natbib}

\algnewcommand\algorithmicinput{\textbf{INPUT:}}
\algnewcommand\INPUT{\item[\algorithmicinput]}
\algnewcommand\algorithmicoutput{\textbf{OUTPUT:}}
\algnewcommand\OUTPUT{\item[\algorithmicoutput]}

\usepackage{hyperref}[]
\hypersetup{
    colorlinks=true,
    linkcolor=blue,
    filecolor=magenta,      
    urlcolor=cyan,
      citecolor=blue,
    }

\usepackage{cleveref}

\newtheorem{theorem}{Theorem}

\newtheorem{lemma}[theorem]{Lemma}
\newtheorem{proposition}[theorem]{Proposition}
\newtheorem{corollary}[theorem]{Corollary}

\newtheorem{definition}{Definition}

\newtheorem{assumption}{Assumption}

\allowdisplaybreaks

\DeclareMathOperator*{\argmin}{arg\,min}

\DeclareFontFamily{U}{mathx}{\hyphenchar\font45}
\DeclareFontShape{U}{mathx}{m}{n}{
	<5> <6> <7> <8> <9> <10>
	<10.95> <12> <14.4> <17.28> <20.74> <24.88>
	mathx10
}{}
\DeclareSymbolFont{mathx}{U}{mathx}{m}{n}
\DeclareFontSubstitution{U}{mathx}{m}{n}
\DeclareMathAccent{\widecheck}{0}{mathx}{"71}
\DeclareMathAccent{\wideparen}{0}{mathx}{"75}

\title{  Lattice partition recovery with dyadic CART}
\author[1]{Oscar Hernan Madrid Padilla}
\author[2]{Yi Yu}
\author[3]{Alessandro Rinaldo}
\affil[1]{Department of Statistics, University California, Los Angeles}
\affil[1]{Department of Statistics, University of Warwick}
\affil[3]{Department of Statistics \& Data Science, Carnegie Mellon University}
\date{\today}

\begin{document}

\maketitle

\begin{abstract}

We study piece-wise constant signals corrupted by additive Gaussian noise over a $d$-dimensional lattice. Data of this form naturally arise in a host of applications, and the tasks of signal detection or testing, de-noising and  estimation have been studied extensively in the statistical and signal processing literature. In this paper we consider instead the problem of partition recovery, i.e.~of estimating the partition of the lattice induced by the constancy regions of the unknown signal, using the computationally-efficient dyadic classification and regression tree (DCART)  methodology proposed by \citep{donoho1997cart}.   We prove that, under appropriate regularity conditions on the shape of the partition elements, a DCART-based procedure consistently estimates the underlying partition at a rate of order $\sigma^2 k^* \log (N)/\kappa^2$, where $k^*$ is the minimal number of rectangular sub-graphs obtained using recursive dyadic partitions supporting the signal partition, $\sigma^2$ is the noise variance, $\kappa$ is the minimal magnitude of the signal difference among contiguous elements of the partition and $N$ is the size of the lattice. Furthermore,  under stronger assumptions, our method attains a sharper estimation error of order $\sigma^2\log(N)/\kappa^2$, independent of $k^*$, which we show to be minimax rate optimal.  Our theoretical guarantees further extend to the partition estimator based on the optimal regression tree estimator (ORT) of \cite{chatterjee2019adaptive} and to the one obtained through an NP-hard exhaustive search method.  We corroborate our theoretical findings and the effectiveness of  DCART for partition recovery in simulations.

 
\end{abstract}

\section{Introduction}\label{sec-intro}

Suppose we observe a noisy realization of a structured, piece-wise constant signal supported over a $d$-dimensional square lattice (or grid graph) $L_{d, n} = \{1, \ldots, n\}^d$.  Data that can be modeled in this manner arise in several application areas, including in satellite imagery \citep[e.g.][]{stroud2017bayesian, whiteside2020semi}, computer vision \citep[e.g.][]{bian2017gms, wirges2018object}, medical imaging \citep[e.g.][]{roullier2011multi, lang2014adaptive}, and neuroscience \citep[e.g.][]{fedorenko2013broad,tansey2018false}.  Our goal is to estimate the constancy regions of the underlying  signal. Specifically, we assume that the data $y \in \mathbb{R}^{L_{d, n}}$ are such that, for each coordinate $i \in L_{d, n}$,
\begin{equation}\label{eq-model}
	y_i = \theta^*_i + \epsilon_i,
\end{equation}
where $(\epsilon_i,  i \in L_{d, n})$ are i.i.d.~$\mathcal{N}(0, \sigma^2)$ noise variables and the unknown signal $\theta^* \in \mathbb{R}^{L_{d, n}}$ is assumed to be piece-wise constant over an unknown rectangular partition of~$L_{d, n}$. We define a subset 
$R \subset L_{d, n}$ to be a \emph{rectangle} if $R = \prod_{i=1}^d [a_i, b_i]$, where $[a, b] = \{j \in \mathbb{Z}: \, a \leq j \leq b\}$, $a, b \in \mathbb{Z}$.  A \emph{rectangular partition} of $L_{d,n}$, $\mathcal{P}$, is a collection of disjoint rectangles $\{R_l\} \subset L_{d, n}$, satisfying $\cup_{R \in \mathcal{P}} R = L_{d,n}$.  To each vector in $L_{d, n}$, there corresponds a (possibly trivial) rectangular partition.
\begin{definition}\label{def:0}
	A \emph{rectangular partition associated with a vector $\theta \in \mathbb{R}^{L_{d, n}}$} is a rectangular partition $\{R_l\}_{l \in [1, k]}$ of $L_{d, n}$, such that $\theta$ takes on constant values over each $R_l$. For a vector $\theta \in \mathbb{R}^{L_{d, n}}$, we let  $k(\theta)$ be the smallest positive integer such that there exists a rectangular partition with $k(\theta)$ elements and associated with $\theta$. 
\end{definition}

In this paper, we are interested in recovering a rectangular partition  associated with the signal $\theta^*$ in~\eqref{eq-model}. 
A  complication immediately arises when $d \geq 2$: the rectangular partition associated with a given $\theta$ is not necessarily unique. This fact is illustrated in \Cref{fig-1}, where the left-most plot depicts the lattice supported  vector $\theta$, consisting  of a rectangle of elevated value (in grey) against a background (in white).  For such $\theta$, we show three possible rectangular partitions, each of which comprised of five rectangles (the second, third and fourth plots). In fact, the partition recovery problem is well defined, as long as we consider coarser partitions comprised by unions of adjacent rectangles instead of individual rectangles: see \Cref{def1} below for details. We remark that this issue does not occur in the univariate ($d = 1$) case, for which the partition recovery task has been  thoroughly studied in the change point literature; see \cref{sec:literture} below. Thus, we assume that $d \geq 2$.

\begin{figure}
	\includegraphics[width = \textwidth, height = 0.15\textheight]{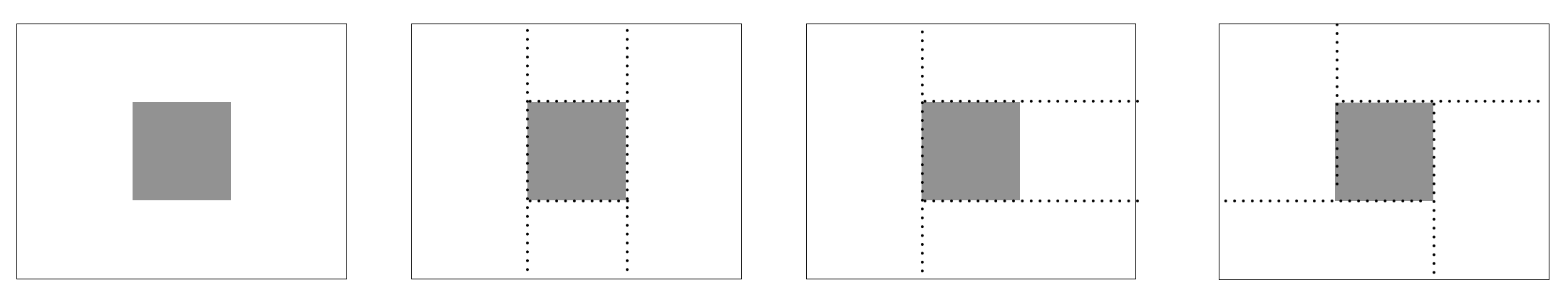}	
	\caption{Rectangular partitions associated to a vector are not necessarily unique.}\label{fig-1}
\end{figure}

For the purpose of estimating the rectangular partition associated with $\theta^*$ (or, more precisely, its unique coarsening as formalized in \Cref{def1}), we resort to the dyadic classification and regression tree (DCART) algorithm of \cite{donoho1997cart}. This is a polynomial-time  decision-tree-based algorithm developed for de-noising purposes for signals over lattices, and is a variant of the classification and regression trees (CART)~\cite{breiman1984classification}. See \Cref{sec-prob-setup} below for a description of DCART. The optimal regression trees (ORT) estimator,  recently proposed in \cite{chatterjee2019adaptive}, further builds upon DCART and delivers sharp theoretical guarantees for signal estimation while retaining good computational properties -- though we should mention that in our experiments we have found DCART to be significantly faster. Both DCART and its more sophisticated version ORT can be seen as approximations to the NP-hard estimator
\begin{equation}\label{eq-NP}
	\theta_1 = \argmin_{\theta \in \mathbb{R}^{L_{d, n}}}\left\{2^{-1}\|y - \theta\|^2 + \lambda k(\theta)\right\}, 
\end{equation}
where $k(\theta)$ is given in \Cref{def:0}, $\|\cdot\|$ is the vector (or Euclidean) $\ell_2$-norm and $\lambda > 0$ a tuning parameter. 
DCART modifies the above, impractical optimization problem by restricting only to 
dyadic rectangular partitions. This leads to  significant gains in computational efficiency without sacrificing on the statistical performance.
Indeed, decision-tree-based algorithms have been shown to be optimal under various settings for the purpose of signal estimation; see \cite{chatterjee2019adaptive, kaul2021segmentation}.
In this paper, we further demonstrate their effectiveness  for the different task of partition recovery. In particular,  we show how simple modifications of the  DCART (or ORT) estimator yield practicable procedures for partition recovery with  good theoretical guarantees and derive novel localization rates.

Note that, there is a wide array of applications focusing on detecting the regions rather than estimating the background signals, especially in surveillance and environment monitoring.  Our work is motivated by all the applications/problems considered in the large literature on biclustering, where the underlying signal is assumed to be piecewise constant. Estimating the boundary of the partition is the most refined and difficult task in these settings. Thus, any of the many scenarios in which biclustering is relevant can be used to motive our task. An analogous observation holds also for the more general problem of identifying an anomalous cluster (sub-graph) in a network, a problem that has been tackled (for testing purposes only) by \cite{arias-castro2011}, the reference therein provide numerous examples of applications.  On a high-level, the relationship between the partition and signal recoveries can be thought of the relationship between the estimation consistency and support consistency in a high-dimensional linear regression problems. They can be done by almost identical algorithms but the theoretical results rely on different sets of conditions.

The paper is organized as follows. In the rest of this section we formalize the problem settings and the task of partition recovery, and describe the DCART procedure. We further summarize our main findings and discuss  related literature. \Cref{sec-theory} contains our main results about one- and two-sided consistency of DCART and its  modification. In \Cref{sec-optimal} we derive a minimax lower bound stemming from the case of one rectangular region of elevated signal against pure background. Illustrative   simulations corroborating our findings can be found in \Cref{sec-numerical}. The Supplementary Material contains the proofs.

\noindent {\bf Notation}\\
We set $N = n^d$, the size of the lattice $L_{d, n}$, where we recall that $d \geq 2$ is assumed fixed throughout.  For any integer $m \in \mathbb{N}^*$, let $[m] = [1, m]$.  Given a rectangular partition $\Pi$ of $L_{d, n}$, let $S(\Pi)$ be the linear subspace of $\mathbb{R}^{L_{d,n}}$ consisting of vectors with constant values  on each rectangle in $\Pi$ and let $O_{S(\Pi)}(\cdot)$ be the orthogonal projection onto $S(\Pi)$.  For any $R \subset L_{d, n}$ and $\theta \in \mathbb{R}^{L_{d, n}}$, let $\bar{\theta}_R = |R|^{-1}\sum_{i \in R}\theta_i$, where $|\cdot|$ is the cardinality of a set.  
Two rectangles   $R_1, R_2 \in \Pi$ are said to be {\it adjacent} if there exists $l \in [d]$ such that $R_1$ and $R_2$ share a boundary along $e_l$ and one is a subset of the other in the hyperplane defined by $e_l$, the $l$th standard basis vector in $\mathbb{R}^d$.  See \Cref{def-adjacent} for a rigorous definition. This concept of adjacency is specifically tailored to -- and in fact only valid for -- dyadic (and hierarchical, in the sense specified by \cite{chatterjee2019adaptive}) rectangular partitions, which are most relevant for this paper.  For any subsets $A, B \subset L_{d, n}$, define $\mathrm{dist}(A, B) = \min_{a \in A, b \in B}\|a-b\|$.  Throughout this paper, we will use the $\ell_2$-norm as the vector norm.

\subsection{Problem setup}\label{sec-prob-setup}

We begin by introducing two key parameters for the model specified in \eqref{eq-model} and a well-defined  notion of rectangular partition induced by $\theta^*$.
\begin{definition}[Model parameters, induced partitions] \label{def1}
	Let $\theta^*$  as in \eqref{eq-model} and  $\{R_j^*\}_{j \in [m]}$ be a rectangular partition of $L_{d, n}$ associated with $\theta^*$.  Consider the graph $G^* = (E^*, V^*)$, where $V^* = [m]$ and $E^* = \{(i, j): \, \bar{\theta}^*_{R_i^*} = \bar{\theta}^*_{R_j^*},\, R_i^*\mbox{ and } R_j^* \mbox{ are adjacent}\}$.  Let $\{C_l^*\}_{l \in [L]}$ be all connected components of $G^*$ and define $\Lambda^* = \Lambda^*(\theta^*) = \{\cup_{j \in C_1^*}R_j^*, \ldots, \cup_{j \in C_L^*}R_j^*\}$ as the partition (not necessarily rectangular) induced by $\theta^*$.  We say that the union of rectangles  $\cup_{j \in C_s^*} R_j^*$ and $\cup_{j \in C_t^*}R_j^*$, $s, t \in [L]$, $s \neq t$, are adjacent, if and only if there exists $(i, j) \in C_s^* \times C_t^*$ such that $R_i^*$ and $R_j^*$ are adjacent.
	
	Let $\kappa$ and $\Delta$ be the minimum jump size and minimal rectangle size, respectively, formally defined as
	\[
	\kappa = \underset{a \in  A, b\in B, \, \, A, B \in \Lambda^*, \,  \theta_a^* \neq \theta_b^* }{\min}\,  \vert  \theta_a^*-\theta_b^*\vert \quad \mbox{and} \quad \Delta  = \underset{j \in [m]}{\min } \vert R_j^*\vert.
	\]
\end{definition}
It is important to emphasize the difference between a partition associated with  $\theta^*$, as described in \Cref{def:0}, which may not be unique, and the partition  $\Lambda^*$  induced by $\theta^*$ of \Cref{def1}, which is instead unique and thus describes a well-defined functional of $\theta^*$. The parameters $\kappa$ and $\Delta$ capture two complementary aspects of the intrinsic difficulty of the problem of estimating $\Lambda^*$; intuitively, one would expect the partition recovery task to be  more difficult when $\kappa$ and $\Delta$ are small (and $\sigma$ is large). Below, we will prove rigorously that this intuition is indeed correct. When $d=1$, both parameters, along with $\sigma$, have in fact been shown to fully characterize the change point localization task: see, e.g., \cite{wang2020univariate, verzelen2020optimal}. 

The partition recovery task can therefore be formulated as that of constructing an estimator $\widehat{\Lambda}$ of $\Lambda^*$, the induced partition  of $\theta^*$,  such that, as the sample size $N$ grows unbounded and with probability tending to one,  
\begin{equation}\label{eq-est-goal}
	|\widehat{\Lambda}| = |\Lambda^*| \quad \mbox{and} \quad \Delta^{-1}d_{\mathrm{Haus}}(\widehat{\Lambda}, \Lambda^*) = \Delta^{-1} \max_{A \in \Lambda^*} \min_{B \in \widehat{\Lambda}} |A \triangle B| \to 0,
\end{equation}
where $A \triangle B$ is the symmetric difference between $A$ and $B$. We refer to $d_{\mathrm{Haus}}(\widehat{\Lambda}, \Lambda^*)$ as the localization error for the partition recovery problem. 

\noindent \textbf{The dyadic classification and regression trees (DCART) estimator.}  In order to produce a computationally efficient estimator of $\Lambda^*$ satisfying the consistency requirements \eqref{eq-est-goal}, we deploy the DCART procedure \cite{donoho1997cart}, which can be viewed as an approximate  solution  to the problem in \eqref{eq-NP}.  Instead of optimizing over all vectors in $\mathbb{R}^{L_{d, n}}$, DCART minimizes the objective function only over vectors associated with a \emph{dyadic rectangular partition}, which is defined as follows.  Let $R = \prod_{i \in [d]} [a_i, b_i] \subset L_{d, n}$ be a rectangle.  A \emph{dyadic split} of $R$ chooses a coordinate $j \in [d]$, $l$ the middle point of $[a_j, b_j]$, and splits $R$ into
\[
R_1=\prod_{i \in [j-1]} [a_i,b_i] \times [a_j,l] \times \prod_{i\in[j+1, d]}[a_i,b_i] \mbox{ and } R_2=\prod_{i\in [j-1]} [a_i,b_i]   \times [l+1, b_j] \times   \prod_{i\in[j+1, d]} [a_i,b_i],
\]
with $[0] = [c_2, c_1] = \emptyset$, $c_2 > c_1$.  Assuming that $n$ is a power of 2,  starting from  $L_{d,n}$ itself, we proceed iteratively as follows.  Given the partition $\{R_u\}_{u \in [k]}$, one chooses a rectangle $R_u$ and performs a dyadic split on $R_u$ that leads to the largest reduction in the objective function.  Any partition constructed through a sequence of such steps is called a dyadic rectangular partition.  	With a pre-specified $\lambda > 0$, the DCART estimator is 
\begin{equation} \label{eqn:cart}
	\tilde{\theta} =  O_{S(\widetilde{\Pi})}(y), \, \mbox{ where }  \widetilde{\Pi}  \,\in \,\underset{  \Pi    \in \mathcal{P}_{\mathrm{dyadic},d,n}  }{\arg \min}\left\{2^{-1} \|   y-  O_{S(\Pi)}(y) \|^2   +  \lambda  \vert \Pi\vert    \right\},
\end{equation}
where  $\mathcal{P}_{\mathrm{dyadic},d,n}$  is the set of all dyadic rectangular partitions of $L_{d,n}$. As shown in \cite{donoho1997cart} and \cite{chatterjee2019adaptive}, the DCART estimator can be obtained via dynamic programming with a computational cost of $O(N)$.
Given any solution to (\ref{eqn:cart}), a natural (though, as we will see,  sub-optimal) estimator of the induced partition  of $\theta^*$ is $\widetilde{\Pi}$, the  partition associated with the resulting DCART estimator $\tilde{\theta}$. Importantly, by the property of DCART, and using the fact that the Gaussian errors have a Lebesgue density, $\widetilde{\Pi}$ is in fact a dyadic-rectangular partition  and is unique with probability one, and thus the resulting estimator is well-defined.  (Equivalently, the partition associated with $	\tilde{\theta}$ and the one induced by $\tilde{\theta}$ coincide.)

\subsection{Summary of our results}\label{sec:summary}
We briefly summarize the contributions made in this paper.

\textbf{One-sided consistency of DCART.}  Though  DCART is known to be a minimax rate-optimal estimator of $\theta^*$~\citep{chatterjee2019adaptive}, for the task of partition recovery its associated partition $\widetilde{\Pi}$ has sub-optimal performance. Indeed, due  to the  nature of the procedure, it is easy to construct cases in which the DCART over-partitions. See \Cref{fig-4}.
In these situations, DCART falls short with respect to the target conditions for consistency described in \eqref{eq-est-goal}.  Nonetheless, it is possible to prove a weaker one-sided consistency guarantee, in the sense that every resulting DCART rectangle is almost constant. In detail, let $\mathcal{R} = \{R_l\}_{l \in [\tilde{m}]}$ be the rectangular partition defined by $\widetilde{\Pi}$ in \eqref{eqn:cart}.  Then, we show in \Cref{sec-one-sided} that, for any $R_i \in \mathcal{R}$, there exists $S_i \subset R_i$ such that $\theta^*_t = \theta^*_u$, $u, t \in S_i$, and $\sum_{i \in [\tilde{m}]} \vert  R_i \backslash S_i\vert    \lesssim   \kappa^{-2} \sigma^2   k_{\mathrm{dyad}}(\theta^*) \log (N).$
Throughout, the quantity $k_{\mathrm{dyad}}(\theta^*)$ refers to the smallest positive integer $k$ such that there is a $k$-dyadic-rectangular-partition of $L_{d, n}$ associated with~$\theta^*$. 

\textbf{Two-sided consistency of DCART: A two-step estimator.}  In order to resolve the unavoidable over-partitioning issue with the naive DCART partition estimator and in order to prevent the occurrence of spurious clusters, we develop a more sophisticated two-step procedure. In the first step we use a variant of DCART that discourages the creation of rectangles of small volumes. In the second step, we apply a pruning algorithm merging  rectangles when their values are similar and the rectangles are not far apart.  With probability tending to one as $N \to \infty$, the final output $\widehat{\Lambda}$ satisfies \eqref{eq-est-goal} with $d_{\mathrm{Haus}}(\widehat{\Lambda}, \Lambda^*) \leq \kappa^{-2} \sigma^2   k_{\mathrm{dyad}}(\theta^*) \log (N)$.  This result is the first of its kind in the setting of lattice with arbitrary dimension $d \geq 2$.  This is shown in \Cref{sec-two-sided}.  

\textbf{Optimality: A regular boundary case.}  In \Cref{sec-optimal}, we consider the special case in which, for each rectangle in the  rectangular partitions induced by $\theta^*$  only has $O(1)$-many rectangles within distance of order $\sigma^2 \kappa^{-2} k_{\mathrm{dyad}}(\theta^*) \log (N)$.  While more restrictive than the scenarios we study in Sections~\ref{sec-one-sided} and \ref{sec-two-sided}, this setting is broader than the ones adopted in the cluster detection literature \citep[e.g.][]{arias2011detection, addario2010combinatorial}.  In this case, with probability approaching one as $N \to \infty$, the estimator $\widehat{\Lambda}$ satisfies \eqref{eq-est-goal} and $d_{\mathrm{Haus}}(\widehat{\Lambda}, \Lambda^*) \leq \kappa^{-2} \sigma^2 \log (N)$.  This error rate is shown to be minimax optimal, with a supporting minimax lower bound result in Proposition~\ref{prop-lb}. 

\subsection{Related and relevant literature}\label{sec:literture}
The problem at hand is closely related to several recent research threads involving detection and estimation of a structured signal.
When $d=2$, our settings can be viewed as a generalization of those used for the purpose of biclustering, i.e.~detection and estimation of sub-matrices. Though relatively recent, the literature on this topic is extensive, and the problem has been largely solved, both theoretically and methodologically. See, e.g., \cite{10.1214/09-AOAS239}, \cite{bicl}, \cite{butucea2013}, \cite{10.1214/14-AOS1300}, \cite{10.3150/11-BEJ394}, \cite{liu.castr:17}, \cite{ARIASCASTRO201729},
\cite{10.1214/16-AOS1488}, \cite{butucea:15}, \cite{JMLR:v17:15-617}, \cite{JMLR:v18:17-297}, \cite{DBLP:journals/jmlr/ChenX16}  and \cite{10.1214/09-AOAS239}. 

In the more general settings postulating a structured signal supported over a graph (including the grid graph), sharp results for the detection problem of testing the existence of a sub-graph or cluster in which the signal is different from the background are available in the literature: see, \cite{arias2008searching}, \cite{arias2011detection}, \cite{addario2010combinatorial}. 
Concerning the estimation problem, \cite{tibshirani2011solution}, \cite{sharpnack2012sparsistency}, \cite{chatterjee2019adaptive}, \cite{fan2018approximate} and  others, focused on de-noising the data and upper-bounding $\|\hat{\theta} - \theta^*\|_*$, where $\hat{\theta}$ is an estimator of $\theta^*$ and $\|\cdot\|_*$ is some vector norm.  In yet another stream of work \citep[e.g.][]{han2019global, brunel2013adaptive, korostelev1991asymptotically} concerned with empirical risk minimization, the problem is usually formulated as identifying a single subset.  More discussions can be found in \Cref{sec-compar}.

What sets  our contributions apart from those in the literature referenced above, which have primarily targeted detection and signal  estimation,  is the focus on the arguably different task of partition recovery. As a result, the estimation bounds we obtain are, to the best of our knowledge, novel as they do not stem directly from  the existing results.

It is also important to mention how the partition recovery task can be cast as a univariate change point localization problem. Indeed, when $d=1$, the two  coincide; see \cite{wang2020univariate, verzelen2020optimal}. However, the case of $d \geq 2$ becomes significantly more challenging  due to the lack of a total ordering over the lattice. Consequently, our results imply also novel localization rates for change point analysis in  multivariate settings.

\section{Consistency rates for the partition recovery problem}\label{sec-theory}

In this section, we investigate the theoretical properties of DCART  and of a two-step estimator also based on DCART for partition recovery.  We remark that instead of DCART, it is possible to deploy the ORT estimator \cite{chatterjee2019adaptive} or the NP-hard estimator \eqref{eq-NP} in our algorithms.
Our theoretical results still hold by simply replacing the term $k_{\mathrm{dyad}}(\theta^*)$, in both the upper bound and the choice of tuning parameters,  with the smallest $k$ such that there is a $k$-hierarchical-rectangular-partition ($k_{\mathrm{hier}}(\theta^*)$) or $k$-rectangular-partition ($k(\theta^*)$) of $L_{d, n}$ associated with $\theta^*$, respectively.  Thus, using these more complicated methodologies that scan over larger classes of rectangular partitions will result in smaller upper bounds in Theorems~\ref{thm1}, \ref{thm3} and \ref{cor1}.  See \cite{chatterjee2019adaptive} for details about the relationship of $k(\theta^*)$, $k_{\mathrm{dyad}}(\theta^*)$ and $k_{\mathrm{hier}}(\theta^*)$.

\subsection{One-sided consistency:  DCART }\label{sec-one-sided}

As illustrated in \Cref{fig-4}, the DCART procedure will  produce too fine a partition in many situations, even if the signal is directly observed (i.e., there is no noise). Thus, the naive partition estimator based on the constancy regions of the DCART estimator $\tilde{\theta}$ as in \eqref{eqn:cart} will inevitably suffer from the same drawback. Nonetheless, it is still possible to demonstrate a one-sided type of accuracy and even consistency for such a simple and computationally-efficient estimator. Specifically, in our next result we show that in every dyadic rectangle supporting the DCART estimator, $\theta^*$ has almost constant mean. The reverse does not hold however, as there is no guarantee that every rectangle in the partition induced by the true signal $\theta^*$ is mostly covered  by one dyadic DCART rectangle.

\begin{theorem} \label{thm1}
	Suppose that the data satisfy \eqref{eq-model} and that $\tilde{\theta}$ is the DCART  estimator \eqref{eqn:cart} obtained with  tuning parameter $\lambda = C\sigma^2 \log(N)$, where $C > 0$ is a sufficiently large absolute constant.  Let $\{R_j\}_{j \in [k(\tilde{\theta})]}$  be the associated  partition.  For any $j \in [k(\tilde{\theta})]$, let $S_j \subset R_j$ be the largest subset of $R_j$ such that $\theta^*$ is constant on $S_j$.  Then there exist absolute constants $C_1, C_2, C_3, C_4, C_5 > 0$ such that, with probability at least $1 - C_1 \exp\{-C_2\log(N)\}$, the following hold:
	\begin{itemize}
		\item global one-sided consistency:
		\begin{equation}\label{eq-thm1-1}
			\sum_{j \in [k(\tilde{\theta})]} \vert R_j \backslash S_j \vert\leq C_3\kappa^{-2} \sigma^2  k_{\mathrm{dyad}}(\theta^*)\log(N);
		\end{equation}	
		\item local one-sided consistency: for any $j \in [k(\tilde{\theta})]$, if $R_j \setminus S_j \neq \emptyset$,  then
		\begin{equation}\label{eq-thm1-2}
			\vert  R_j \backslash S_j \vert \leq C_4 \kappa_j^{-2}\sigma ^2   k_{\mathrm{dyad}}(\theta^*_{R_j}) \log(N),
		\end{equation}
		where  $\kappa_j = \min_{s, t \in R_j \,:\,\theta^*_s \neq  \theta_t^*} |\theta^*_s - \theta_t^*|$; and
		\item control on  over-partitioning:
		\begin{equation}\label{eq-thm1-3}
			k(\tilde{\theta}) \leq 2 k_{\mathrm{dyad}}(\theta^*) +  C_5.	
		\end{equation}
	\end{itemize}
\end{theorem}

We remark that $k(\tilde{\theta}) = k_{\mathrm{dyad}}(\tilde{\theta})$ due to the construction of $\tilde{\theta}$.
\Cref{thm1} consists of three results.  We have mentioned the over-partitioning issue of  DCART.  The bound \eqref{eq-thm1-3} shows that the over-partitioning is upper bounded, in the sense that the size of the partition induced by DCART is in fact of the same order of the size of the dyadic rectangular partition associated with $\theta^*$.

For each resulting rectangle $R_j$, \eqref{eq-thm1-2} shows that it is almost constant, in the sense that  if the signal possesses different values in $R_j$, then $R_j$ includes a subset $S_j$ which has constant signal value and the size $|R_j \setminus S_j|$ is upper bounded by $\kappa_j^{-2}\sigma ^2   k_{\mathrm{dyad}}(\theta^*_{R_j}) \log(N)$, where $\kappa_j$ is the smallest jump size within~$R_j$.   We note that since $k_{\mathrm{dyad}}(\theta^*_{R_j}) \leq k_{\mathrm{dyad}}(\theta^*)$, if $k_{\mathrm{dyad}}(\theta^*)$ is assumed to be a constant as in the cluster detection literature \citep[e.g.][]{arias-castro2011}, then for general $d \in \mathbb{N}^*$, \eqref{eq-thm1-2} has the same estimation error rate as that in the change point detection literature \citep[e.g.][]{wang2020univariate, verzelen2020optimal}. 

\begin{figure}
	\begin{center}
		\includegraphics[width = 0.5\textwidth, height = 0.15\textheight]{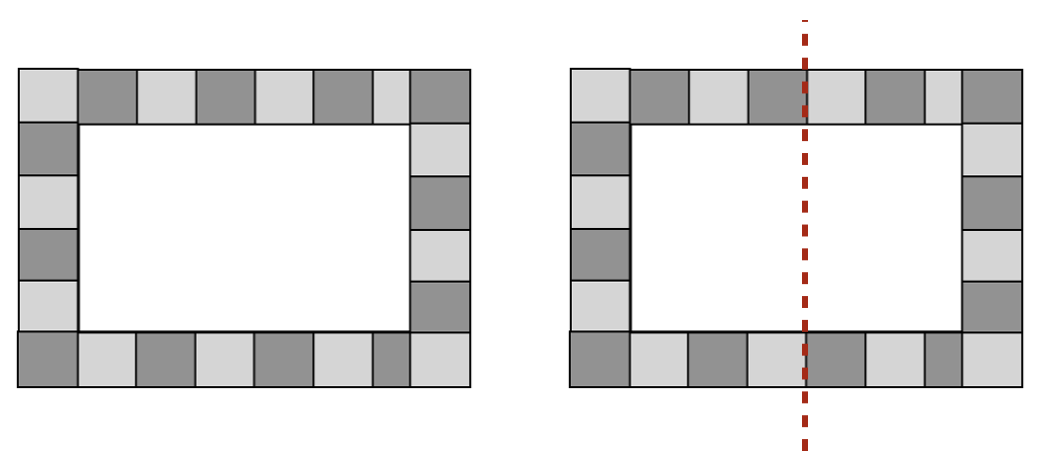}	
		\caption{The left panel is the true signal.  The first dyadic split always cuts a rectangle into two and ends up with over-partitioning, with the right panel as an example. }\label{fig-4}
	\end{center}
\end{figure}

The result in \eqref{eq-thm1-2} provides an individual recovery error, with the individual jump size $\kappa_j$, while paying the price  $k_{\mathrm{dyad}}(\theta^*_{R_j})$.  In \eqref{eq-thm1-1} we show that globally, when we add up the errors in all resulting rectangles, the overall recovery error is of order $\kappa^{-2} \sigma^2  k_{\mathrm{dyad}}(\theta^*)\log(N)$.  When $d = 1$, \cite{verzelen2020optimal} shows the minimax rate of the $L_1$-Wasserstein distance between the vectors of true change points and of the change point estimators is of order $\kappa^{-2} \sigma^2 K$, where $K$ is the true number of change points.  Comparing with this result, \eqref{eq-thm1-1} can be seen as delivering a ``one-sided'' nearly optimal rate, saving for a logarithmic factor.

In the change point localization case, i.e.~when $d = 1$, one can show that $k(\theta^*_{R_j}) \leq 3$ \citep[see, e.g.][]{wang2020univariate} just  assuming mild conditions on $\Delta$. However, as soon as $d \geq 2$ this is no longer the case. As an illustration, consider the left plot of \Cref{fig-4}, where the whole rectangle is $R_j$ and the white one in the middle is $S_j$. Without further constraints on each component, having only conditions on $\Delta$ will not prevent a very fragmented boundary, which can increase the term $k_{\mathrm{dyad}}(\theta^*_{R_j})$ in~\eqref{eq-thm1-2}.

\subsection{Two-sided consistency: A two-step estimator}\label{sec-two-sided}

As we have seen in \Cref{sec-one-sided}, despite having the penalty term $\lambda|\Pi|$ to penalize over-partitioning in the objective function \eqref{eqn:cart}, since the optimization of DCART only restricts to all dyadic partitions, the naive DCART estimator still suffers from over-partitioning.  To address this issue, we propose a two-step estimator which builds upon DCART and merges the corresponding rectangles if they are close enough and if their estimated means are similar. The procedure will not only be guaranteed to return, with high probability, the correct number of rectangles in the partition induced by the signal $\theta^*$, but also fulfill the target for (two-sided) consistency specified in \eqref{eq-est-goal}.

\textbf{The two-step estimator}.  Our two-estimator starts with a constrained DCART, prohibiting splits resulting in  rectangles that are too small.  \emph{The first step} estimation is defined as
\begin{equation} \label{eqn:cart2}
	\widehat{\theta} =  O_{S(\widehat{\Pi})}(y), \quad \mbox{with}\quad \widehat{\Pi}  \,\in \,\underset{  \Pi    \in \mathcal{P}_{\mathrm{dyadic},d,n}(\eta)  }{\arg \min}\left\{2^{-1} \|   y-  O_{S(\Pi)}(y) \|^2   +  \lambda_1  \vert \Pi\vert    \right\},
\end{equation}
where  $\lambda_1, \eta > 0$ are tuning parameters and $\mathcal{P}_{\mathrm{dyadic}, d, n}(\eta)$ is the set of rectangular partitions where every  rectangle is of size  at least  $\eta$. 

\emph{The second step} merges rectangles in the partition associated with the estimator from the first step to overcome over-partitioning.  To be specific, let $\{R_l\}_{l \in [k(\widehat{\theta})]}$ be a rectangular partition of $L_{d, n}$ associated by $\widehat{\theta}$.  For each $(i, j) \in [k(\widehat{\theta})] \times [k(\widehat{\theta})]$, $i < j$, let $z_{(i, j)} = 1$ if 
\begin{equation}\label{eq-2-step-cri-1}
	\mathrm{dist}(R_i,R_j)\,\leq\,  \eta
\end{equation}
and
\begin{equation} \label{eqn:est_cond2}
	\frac{1}{2}\left[    \sum_{l\in  R_i}   (Y_l -    \bar{Y}_{ R_i } )^2  +    \sum_{l\in  R_j}   (Y_l -    \bar{Y}_{ R_j } )^2     \right]     +  \lambda_2  \,>\,  \frac{1}{2}  \sum_{l\in R_i\cup R_j  }   (Y_l -    \bar{Y}_{ R_i\cup R_j  } )^2,
\end{equation}
where $\lambda_2 > 0$ is a tuning parameter; otherwise, we let $z_{(i, j)} = 0$.  With this notation, let $E = \{(i, j) \in [k(\widehat{\theta})] \times [k(\widehat{\theta})]:\, z_{(i,j)} =1 \}$ and  let $\{\widehat{\mathcal{C}}_l\}_{l \in [L]}$ be the collection of all the connected components of the undirected graph $([k(\widehat{\theta})], E)$.  \emph{The final output} can be written as 	
\begin{equation}\label{eq-hat-Lambda}
	\widehat{\Lambda}\, =\, \left\{  \cup_{j \in \widehat{\mathcal{C}}_1} R_j,\ldots,\cup_{j \in \widehat{\mathcal{C}}_L} R_j \right\}.
\end{equation}
Notice that the main computational burden is to find the DCART estimator which has a cost of $O(N)$. From the output of DCART, the computation of the quantities in (\ref{eq-2-step-cri-1})  and (\ref{eqn:est_cond2}) can be done in $O(k(\hat{\theta})^2 )$.
Before describing the favorable properties of the estimator, we state our main assumption.

\begin{assumption} \label{as5}
	If  $A,B \in  \Lambda^*$ with  $A \neq B$   and $\bar{\theta}^*_A = \bar{\theta}_B^*$, then 
	we have that 
	\begin{equation}\label{eq-as5-1}
		\mathrm{dist}(A,B) \, \geq  \, c k_{\mathrm{dyad}}(\theta^*) \kappa^{-2} \sigma^2\log (N),
	\end{equation}
	for some large enough constant $c>0$.   Furthermore, we assume that 
	\begin{equation} \label{eqn:snr}
		\kappa^2 \Delta \geq  ck_{\mathrm{dyad}}(\theta^*)\sigma^2 \log (N).
	\end{equation}
\end{assumption}

\Cref{as5} simply states that if two elements of $\Lambda^*$  have the same signal values then they should be sufficiently apart from each other. \Cref{as5} also specifies a signal-to-noise ratio type of condition.  When $d = 1$, it is well known in the change point literature  \citep[e.g.][]{wang2020univariate} that the optimal signal-to-noise ratio for localization is of the form $\kappa^2 \Delta \gtrsim \sigma^2 \log (N)$. The condition in \eqref{eqn:snr} has an additional $k_{\mathrm{dyad}}(\theta^*)$ factor.  It is an interesting open problem to determine whether this additional term can be avoided when $d \geq 2$.

\begin{theorem}\label{thm3}
	Assume that \Cref{as5} holds.  Suppose that the data satisfy \eqref{eq-model} and $\widehat{\Lambda}$ is the two-step estimator, with tuning parameters $\lambda_1 = C_1 \sigma^2\log(N)$, $\lambda_2 = C_2 k_{\mathrm{dyad}}(\theta^*) \sigma^2 \log(N)$ and 
	\begin{equation} \label{eqn:snr4}
		c_1   k_{\mathrm{dyad}}(\theta^*)  \sigma^2 \kappa^{-2}\log(N) \leq  \eta \leq 	\Delta/c_2, 
	\end{equation}
	where $c_1, c_2, C_1, C_2 > 0$ are absolute constants.  Then with probability at least $1 - N^{-c}$, it holds that 
	\begin{equation}
		\label{eqn:upper_3}
		\vert  \widehat{\Lambda}\vert  \,=\,  \vert \Lambda^*\vert \quad \mbox{and} \quad d_{\mathrm{Haus}}(\widehat{\Lambda}, \Lambda^*) \,\leq\, C  \sigma^2 \kappa^{-2} k_{\mathrm{dyad}}(\theta^*)\log (N),
	\end{equation}
	where $c, C > 0$ are absolute constants.
\end{theorem}

\Cref{thm3} shows that the two-step estimator overcomes the over-partitioning issue of DCART  and  is consistent for the partition recovery problem provided that $\kappa^{-2}\sigma^2  k_{\mathrm{dyad}}(\theta^*)\log (N)/\Delta \rightarrow 0$. The resulting  estimation error is of order at most $\kappa^{-2}\sigma^2  k_{\mathrm{dyad}}(\theta^*)\log (N)$.  

In view of \Cref{as5} and \Cref{thm3}, intuitively speaking, \eqref{eqn:snr4} ensures that if there are two separated regions where the true signal takes on the same value, then the they should be far apart; otherwise, our algorithm would not be able to tell if they should be merged together or keep separated. Eq.~\eqref{eqn:upper_3} requires that the signal strength is large enough. Technically speaking, if \eqref{eqn:snr4} is changed to another quantity, denoted by $w$, then the final result of \Cref{thm3} would be  
\[
d_{\mathrm{Haus}}(\widehat{\Lambda}, \Lambda^*) \,\leq\, C \left\{(w \wedge \Delta) \vee \sigma^2 \kappa^{-2} k_{\mathrm{dyad}}(\theta^*)\log (N)\right\},
\]
where the term $w \wedge \Delta$ is due to the definition of $\Delta$.  This means that the final rate in \Cref{thm3} is determined jointly by \Cref{as5} and an optimal rate.

There are three tuning parameters required by the two-step estimator.  Practical guidance on how to pick them will be provided in \Cref{sec:tuning_parameters}.  The tuning parameter $\lambda_1$ in \eqref{eqn:cart2} is the same as the one in \eqref{eqn:cart} and their theoretical rates are determined by the maximal noise level of a Gaussian distribution  over all possible rectangles using a union bound argument.  The tuning parameter $\lambda_2$ is used in the merging step \eqref{eqn:est_cond2}, penalizing over-partitioning.  Since the candidate rectangles in \eqref{eqn:est_cond2} are the estimators from the first step, these rectangles carry the estimation error from the first step.  An intermediate result from the proof of \Cref{thm3} unveils a similar result to \Cref{thm1}, that for each $R_j$ involved in \eqref{eqn:est_cond2}, there exists a subset $S_j \subset R_j$ having constant signal values satisfying that 
\[
\sum_{j \in [k(\widehat{\theta})]} \vert  R_j \backslash S_j \vert    \lesssim \kappa^{-2} \sigma^2  k_{\mathrm{dyad}}(\theta^*)\log (N).
\]
This suggests that the right choice for $\lambda_2$ should be able to counter this extra $k_{\mathrm{dyad}}(\theta^*)$ factor.  
Finally, the tuning parameter $\eta$ appears twice in the estimation procedure: as a lower bound on the size of the rectangles obtained in first step as  \eqref{eqn:cart2} and as an upper bound on the distance between two rectangles in \eqref{eq-2-step-cri-1}; see \eqref{eqn:snr4}.  The value of $\eta$ should be at least as large as the one-sided upper bound on the recovery error, in order to ensure that over-partitioning cannot occur. As the same time, it should not be chosen  too large, or otherwise the procedure may erroneously prune small true rectangles. By this logic, $\eta$ should not exceed the   minimal size of the true rectangles; this is indeed the  upper bound in condition \eqref{eqn:snr4}. Finally, we would like to point out that, similar conditions are even necessary in some change point localization ($d = 1$) procedures, see  \citep[e.g.][]{wang2020univariate}.

In practice, one may be tempted to abandon the tuning parameter $\eta$, and only prune the DCART output using \eqref{eqn:est_cond2}.  
If one still wants the result to satisfy \eqref{eq-est-goal}, then stronger conditions are needed and worse localization rates are obtained.  We include this result in Section~\ref{sec:two_step} in the supplementary material.  

\subsection{Optimality in the regular boundary case}\label{sec-optimal}

We use a two-step procedure to improve the partition recovery performances of DCART and show the error is of order $\kappa^{-2}\sigma^2 k_{\mathrm{dyad}}(\theta^*) \log(N)$.  A natural question in order is whether one can further expect to improve this rate.  In Proposition~\ref{prop-lb}, we show a minimax lower bound result.  

\begin{proposition}\label{prop-lb}
	Let $\{y_i\}_{i \in L_{d, n}}$ satisfy (\ref{eq-model}) and 
	\begin{equation}\label{eq-lb-S}
		\theta^*_i = \kappa, \quad \mbox{if } i \in S; \quad \theta^*_i = 0, \quad \mbox{if } i \in L_{d, n}\setminus S,
	\end{equation}
	where $S \subset L_{d, n}$ is a rectangle and $|S| = \Delta > 0$.  Let $P^N_{\kappa, \Delta, \sigma}$ denote the corresponding joint distribution.  Consider the class of distributions
	\[
	\mathcal{P}_N = \left\{P^N_{\kappa, \Delta, \sigma}: \, \Delta < N/2, \,    \kappa^2 \Delta /\sigma^2 \geq  \log(N)/6 \right\}. 
	\]	
	Then for $N \geq 2^6$, it holds that $\inf_{\widehat{S}} \sup_{P \in \mathcal{P}_N} \mathbb{E}_P \left\{|\widehat{S} \triangle S|\right\} \geq \sigma^2\kappa^{-2} \log(N)/36$,
	where the infimum is over all estimators $\widehat{S}$ of $S$.
\end{proposition}

\Cref{prop-lb} shows that when the induced partition of $\theta^*$ consists of one  rectangle and its complement, i.e.~when $k_{\mathrm{dyad}}(\theta^*) = O(1)$, the minimax lower bound on the estimation error is of order $\kappa^{-2} \sigma^2 \log(N)$.  Recalling the estimation errors we derived for DCART and the two-step estimator in Theorems~\ref{thm1} and~\ref{thm3}, when $k_{\mathrm{dyad}}(\theta^*) = O(1)$, the results thereof are minimax optimal.

The assumption $k_{\mathrm{dyad}}(\theta^*) = O(1)$ is fairly restrictive, though, using our notation, the case of $|\Lambda^*| =2$ is in fact used in the cluster detection literature \citep{arias-castro2011, arias2008searching}.  In fact, in order to achieve the optimal estimation rate indicated in \Cref{prop-lb}, we only need a boundary regularity condition, in the sense that for every rectangle in the partition induced by $\theta^*$, there are only $O(1)$-many other rectangles nearby.  This condition is formally stated in \Cref{as6}.  

\begin{assumption} \label{as6}
	There exists constant  $C,c >0$  such that  for any   $A \in \Lambda^*$ it holds that 
	\[
	\left\vert \left\{  B  \in \Lambda^* \backslash\{A\} \, :\,\,\,\,\,    \mathrm{dist}(A,B) \leq  c \sigma^2\kappa^{-2} k_{\mathrm{dyad}}(\theta^*) \log (N)    \,\,\right\} \right\vert\,\leq \, C.
	\] 
\end{assumption}

\Cref{as6} asserts that within $c\kappa^{-2}\sigma^2 k_{\mathrm{dyad}}(\theta^*) \log(N)$ distance, each element of $\Lambda^*$ only has $O(1)$-many elements nearby.  This condition shares the same spirit of requiring the cluster boundary to be a bi-Lipschitz function in the cluster detection literature \citep[e.g.][]{arias2011detection}.

\begin{corollary} \label{cor1}
	Assume that Assumptions~\ref{as5} and \ref{as6} hold.  Suppose that the data satisfy \eqref{eq-model} and $\widehat{\Lambda}$ is the two-step estimator defined in \eqref{eq-hat-Lambda}, with tuning parameters $\lambda_1 = C_1 \sigma^2\log(N)$, $\lambda_2 = C_2 k_{\mathrm{dyad}}(\theta^*) \sigma^2 \log(N)$ and $c_1  \kappa^{-2}k_{\mathrm{dyad}}(\theta^*)  \sigma^2 \log(N) \leq  \eta \leq 	 \Delta/c_2$, 
	where $c_1, c_2, C_1, C_2 > 0$ are absolute constants.  Then with probability at least $1 - N^{-c}$, it holds that $\vert  \widehat{\Lambda}\vert  \,=\,  \vert \Lambda^*\vert$ and 
	\begin{equation}\label{eq-cor1-1}
		\max_{A \in \Lambda^*} \min_{\widehat{A} \in \widehat{\Lambda}} |\widehat{A} \triangle A^*| \leq \frac{C \sigma^2 \log(N)}{\kappa^2}  \min\left\{  k(\theta^*), \, \underset{ B\in \Lambda^*,\,\,\mathrm{dist}(A,B)\leq c \sigma^2 \kappa^{-2} k_{\mathrm{dyad}}(\theta^*) \log (N)}{\max}\vert B\vert/\eta \right \}.
	\end{equation}
	where $c, C > 0$ are absolute constants. 
	In particular,  if $\eta \asymp \Delta$ and $\vert A \vert  \asymp \Delta$, for all $A \in \Lambda^*$, then
	\begin{equation}\label{eq-cor1-2}
		d_{\mathrm{Haus}}(\widehat{\Lambda}, \Lambda^*) \,\leq \,  C  \, \sigma^2 \kappa^{-2} \log (N).
	\end{equation}
\end{corollary}

Corollary~\ref{cor1} shows that even if $k_{\mathrm{dyad}}(\theta^*)$ is diverging as the sample size grows unbounded, one can still achieve the minimax optimal estimation error rate $\kappa^{-2}\sigma^2 \log(N)$, with properly chosen tuning parameters and additional regularity conditions on the partition.  An interesting by-product in deriving this rate is \eqref{eq-cor1-1}, which characterizes the effect of the number of nearby rectangles in the estimation error for individual elements in $\Lambda^*$.

\section{Experiments}\label{sec-numerical}

In this section, we demonstrate in simulation the numerical performances of the two-step estimator for the task of partition recovery.  The code is by courtesy of the authors of \cite{chatterjee2019adaptive} and all of our experiments are done in a 2.3 GHz 8-Core Intel Core i9 machine. Our code can be found in \url{https://github.com/hernanmp/Partition_recovery}. We focus on the naive two-step estimator detailed in \Cref{sec:two_step} and denoted here as $\widehat{\Lambda}$.  The implementation details regarding choice of tuning parameters are discussed in \ref{sec:tuning_parameters}.  

\begin{table}[t!]
	\centering
	\caption{\label{tab1} Performance  evaluations   over 50 repetitions under different scenarios. The performance metrics $\text{dist}_1$ and $\text{dist}_2$ are defined in the text.   The numbers in parenthesis denote standard errors. }
	\medskip
	\setlength{\tabcolsep}{3.4pt}
	\begin{scriptsize}
		\begin{tabular}{|rr|rr|rr|rr|rr|rr|} 
			\hline
			\multicolumn{2}{|c|}{Setting} & 	\multicolumn{2}{c|}{$\text{dist}_1$}    & 	\multicolumn{2}{c|}{$\text{dist}_2$}&     \multicolumn{2}{c|}{Setting}	&\multicolumn{2}{c|}{$\text{dist}_1$} &    \multicolumn{2}{c|}{$\text{dist}_2$} \\
			&  $\sigma$            &  $\widehat{\Lambda}$                      &TV-based          &$\widehat{\Lambda}$   &     TV-based        &                &  $\sigma$     &     $\widehat{\Lambda}$           &TV-based          & $\widehat{\Lambda}$         &   TV-based \\ 
			\hline	   
			$1$            &       $0.5$ & \textbf{ 35.8(12.2)}      & 51.6(21.9)         &\textbf{0.0(0.0)}             & 0.1(0.3)             &          2                       &      $0.5$   & 462.6(637.7)  &    \textbf{418.6(246.9)}                      &\textbf{0.2(0.4)}  &   0.4(0.6)\\            
			$1$            &       $1.0$ &\textbf{196.1(401.8)} &   582.3(2429.5)  & \textbf{0.0(0.2)}            &  0.3(0.5)          &          2                       &      $1.0$   &  \textbf{2617.7(4047.7)}&         8630.6(6049.3)            & \textbf{0.7(0.6)}  &  1.4(0.7)   \\     
			$1$            &       $1.5$ &\textbf{298.0(878.7)} &4513.6(5970.9)     & \textbf{0.1(0.3)}           & 0.5(0.5)        &          2                       &      $1.5$   &\textbf{4706.7(5213.7)}&      11477.3(5213.6)                                      &\textbf{ 1.2(0.9)}  & 1.9(0.4) \\      
			\hline
			$3$            &       $0.5$ & \textbf{62.1(231.4)}   &123.0(44.1)         &  \textbf{0.0(0.1)}          &    0.2(0.4)        &          4                       &      $0.5$   &86.3( 231.1)              &    \textbf{52.8(21.7)}              & \textbf{0.2(0.4)}  &  0.3(0.4) \\            
			$3$            &       $1.0$& \textbf{150.8(227.3)}   &  1012.7(752.5)   &   \textbf{0.1(0.3)}           &    0.7(0.5)     &          4                      &      $1.0$   &  119.6(189.3)              &    \textbf{87.8(82.0)}         & \textbf{0.2(0.4)}     &   1.1(1.1)  \\     
			$3$            &       $1.5$ &  \textbf{270.8(479.0)} & 12732.6(3296.8) &  \textbf{0.2(0.5)}            &      1.9(0.4)      &          4                       &      $1.5$   &   399.3(437.1)       &   \textbf{217.6(233.5)}        &          \textbf{0.4(0.7)} &     1.4(1.1)\\             
			\hline  
		\end{tabular}
	\end{scriptsize}
\end{table}

We adopt $\text{dist}_1 = d_{\mathrm{Haus}}(\widehat{\Lambda}, \Lambda^*)$ and $\text{dist}_2 = ||\widehat{\Lambda}| - |\Lambda^*||$ as the measurements.  For each scenario depicted in \Cref{fig1}, we report the mean and standard errors of $\text{dist}_1$ and $\text{dist}_2$ over 50 Monte Carlo simulations.  

As a competitor benchmark, we  consider a similar pruning algorithm based on the total variation estimator \cite{rudin1992nonlinear, tansey2015fast}, namely TV-based, instead of ours based on DCART. The implementation details are discussed in Section \ref{sec:tv}.

\begin{figure}[h!]
	\begin{center}
		\includegraphics[width=1.52in,height=1.52in]{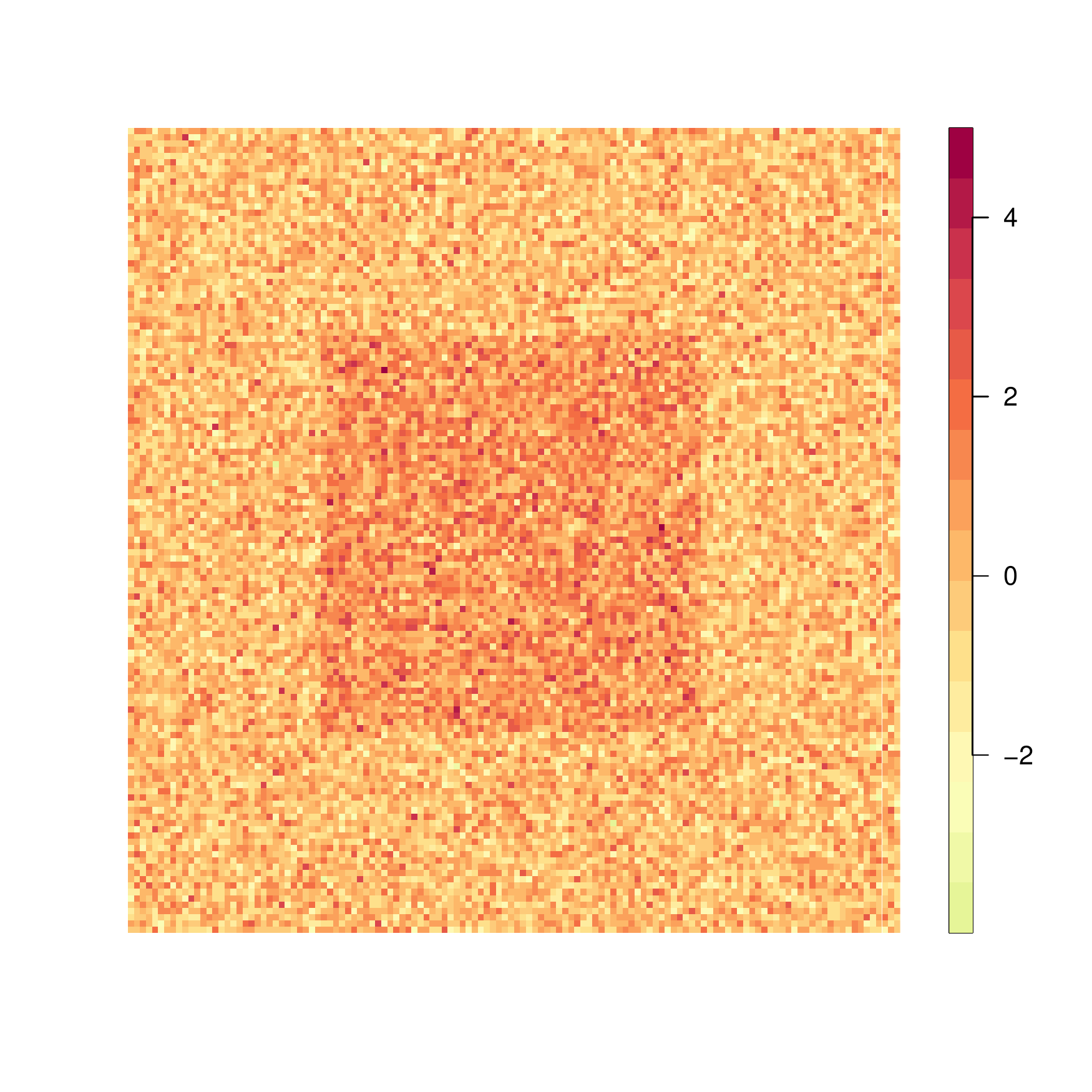} 
		\includegraphics[width=1.52in,height=1.52in]{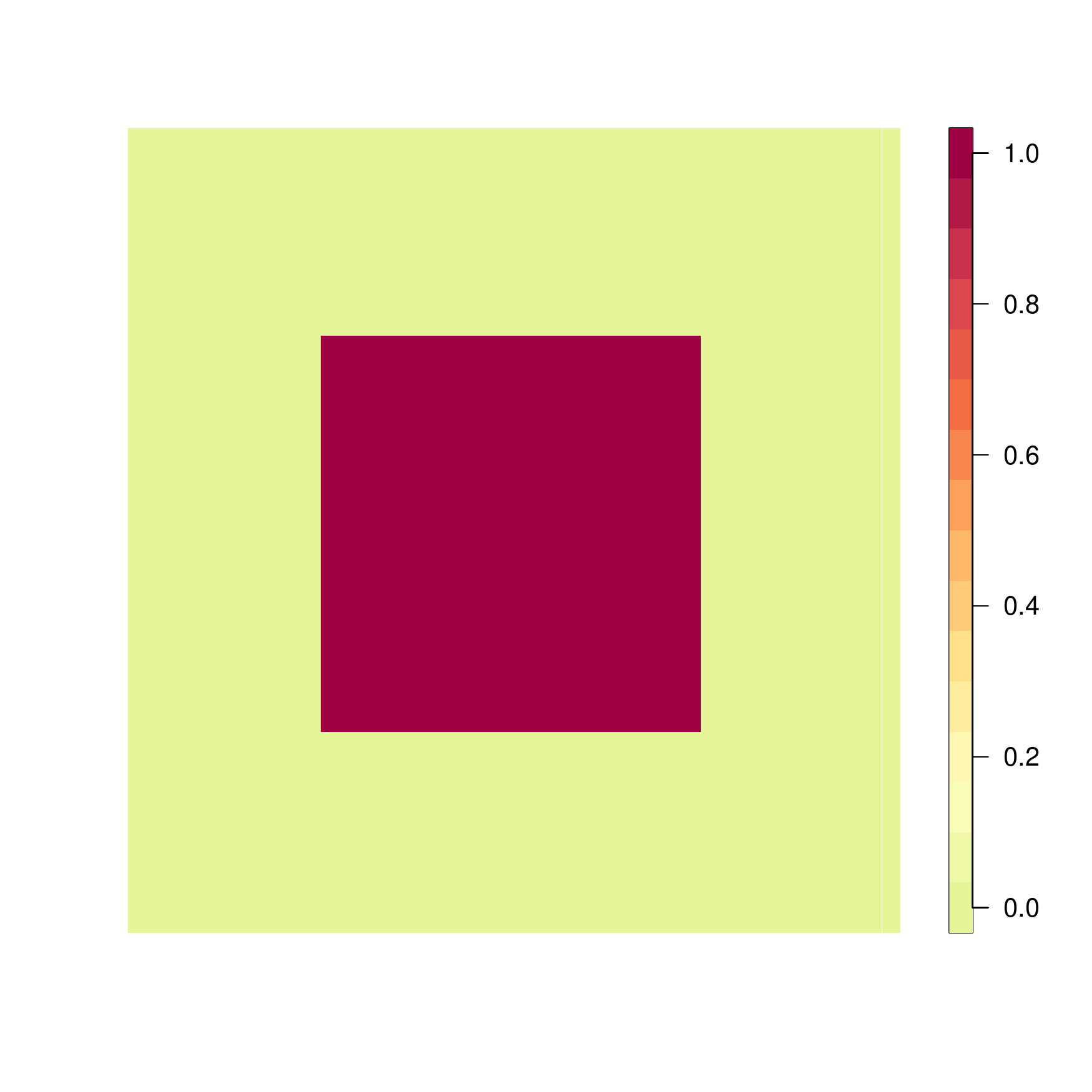} 
		\includegraphics[width=1.52in,height=1.52in]{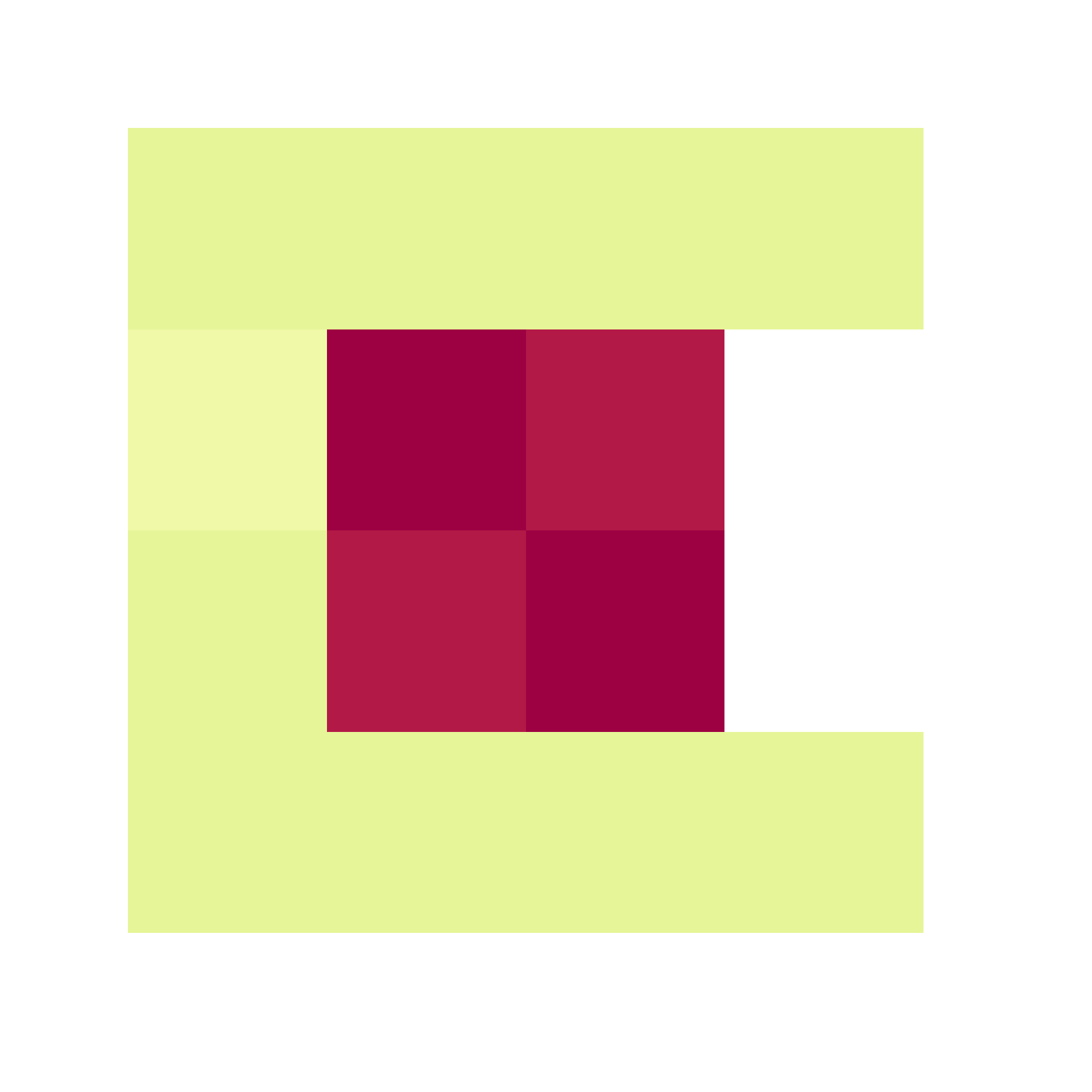} 
		\includegraphics[width=1.52in,height=1.52in]{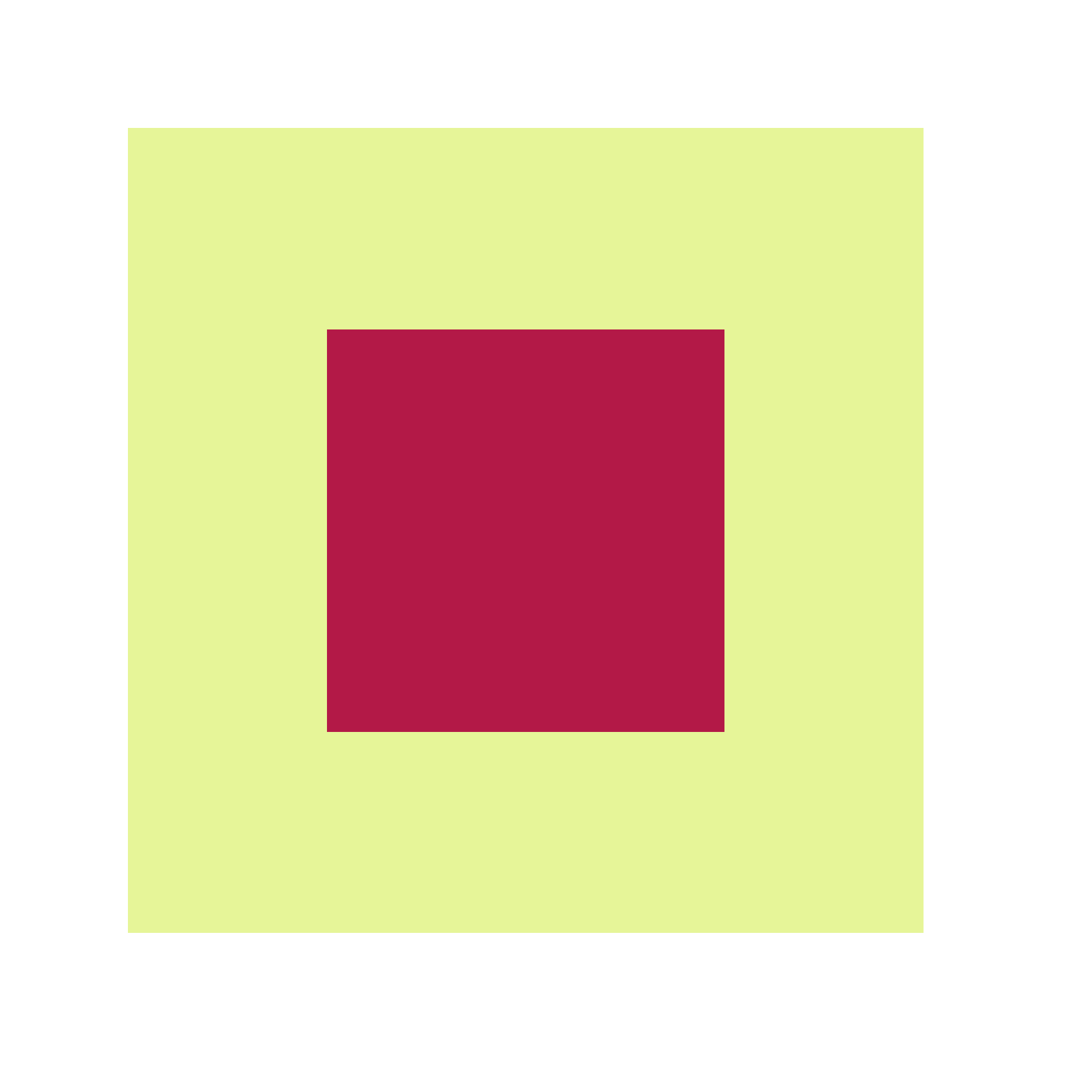}
		\includegraphics[width=1.52in,height=1.52in]{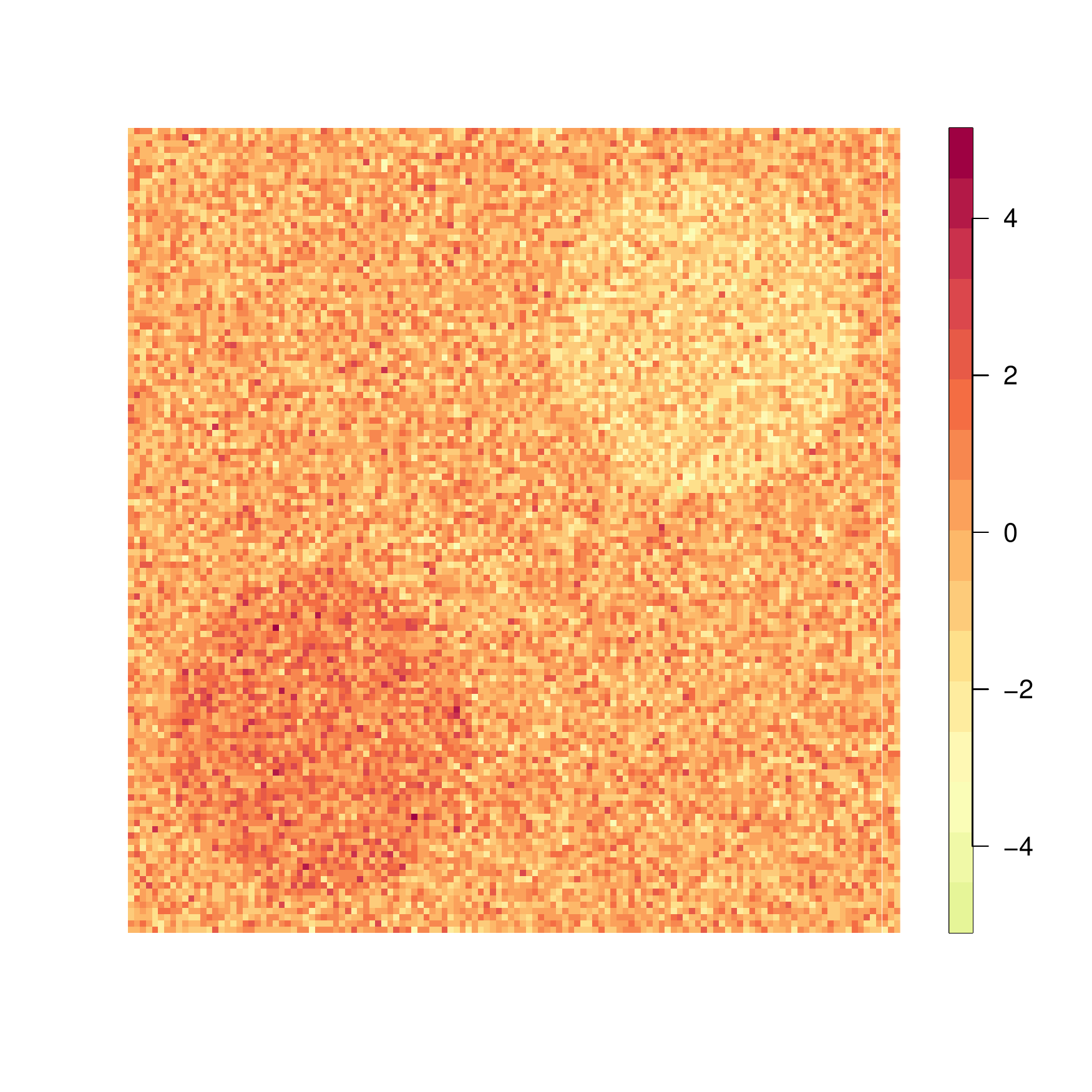} 
		\includegraphics[width=1.52in,height=1.52in]{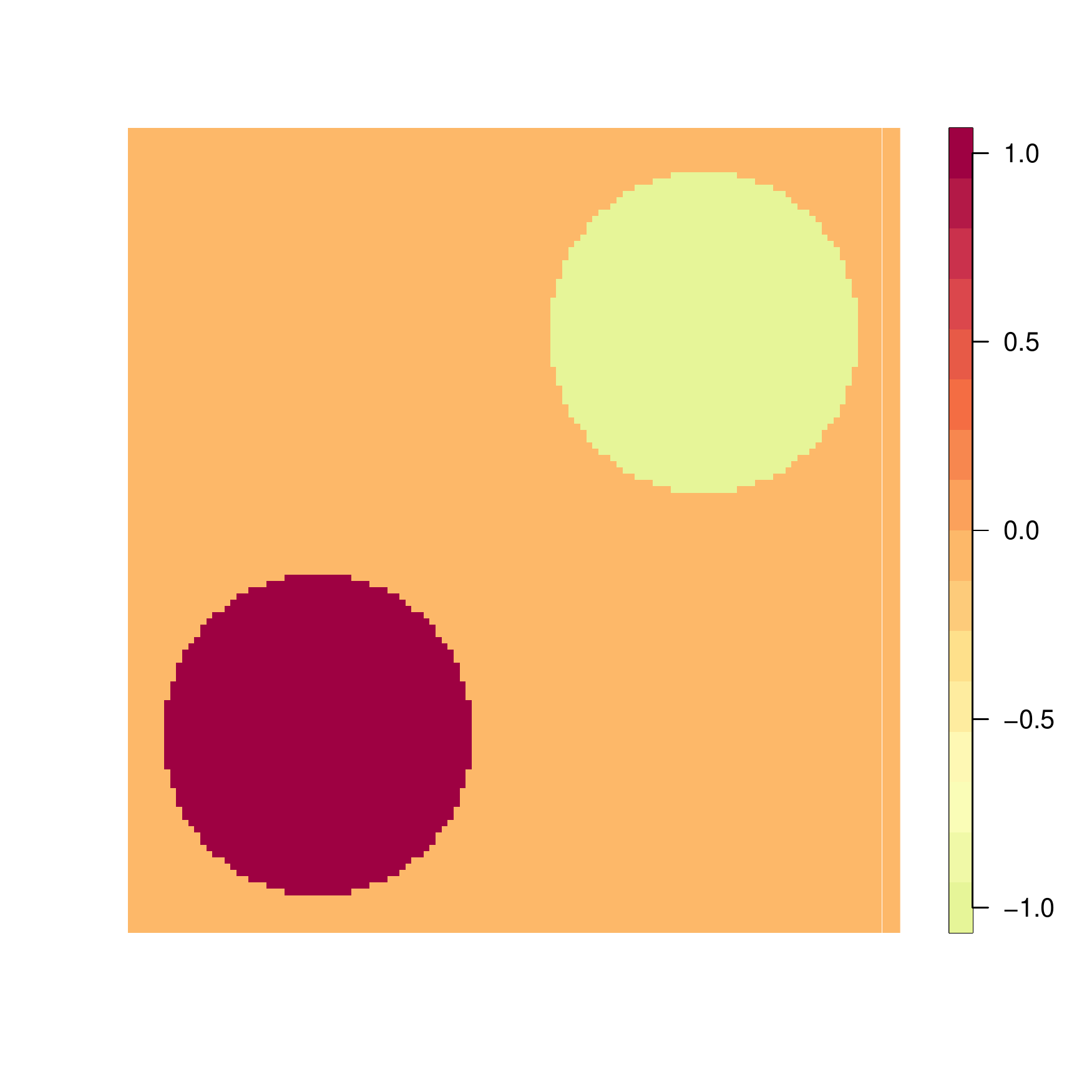} 
		\includegraphics[width=1.52in,height=1.52in]{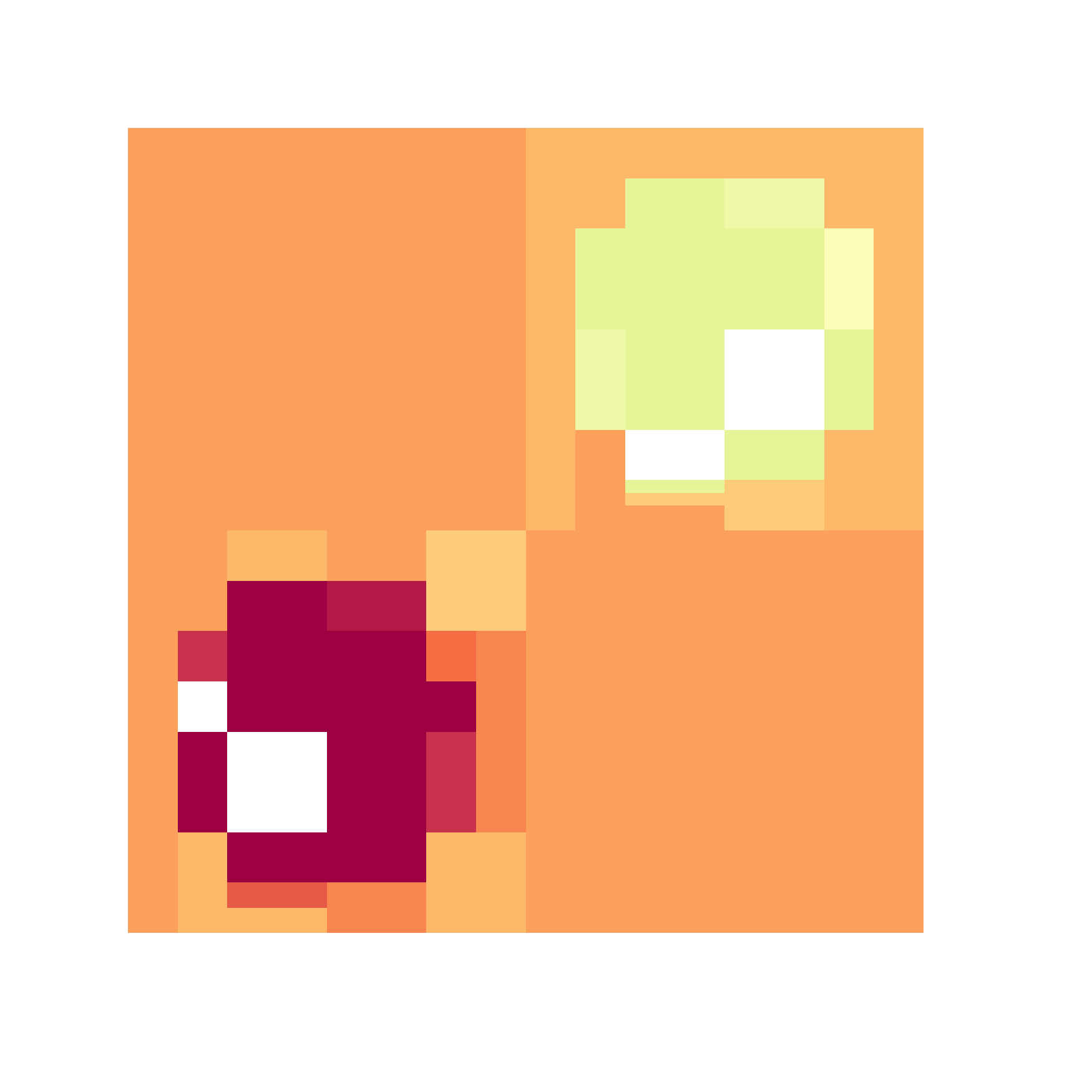} 
		\includegraphics[width=1.52in,height=1.52in]{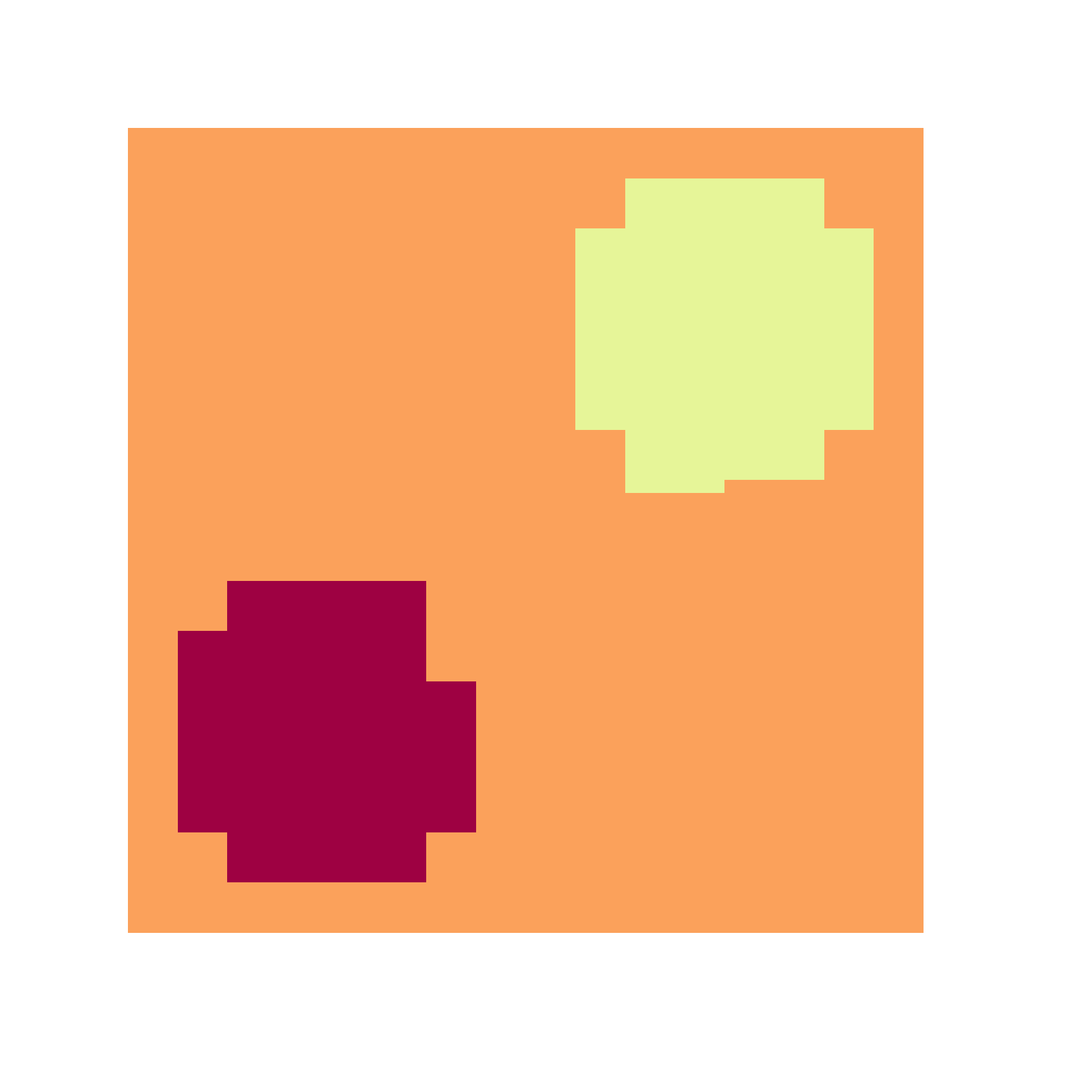}
		\includegraphics[width=1.52in,height=1.52in]{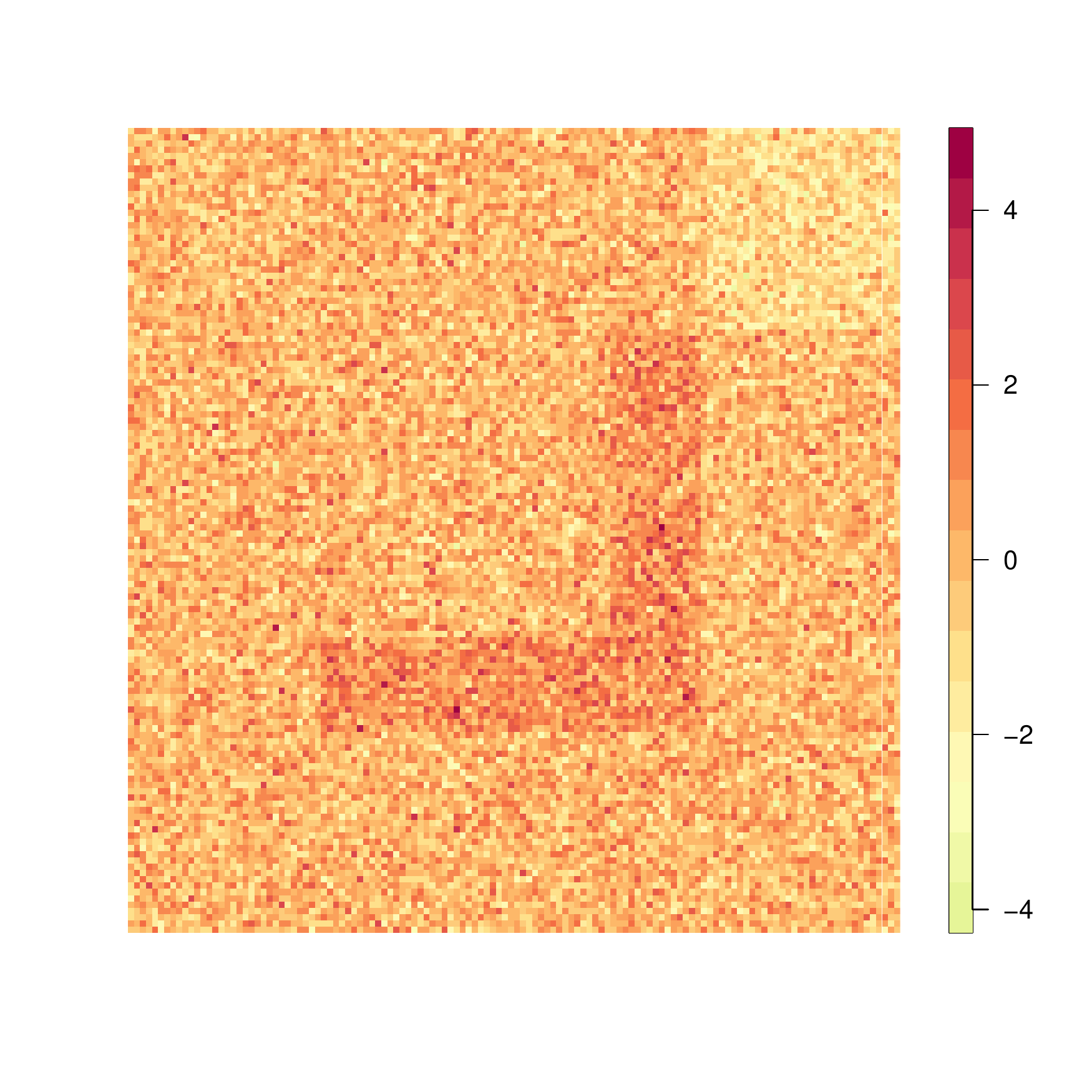} 
		\includegraphics[width=1.52in,height=1.52in]{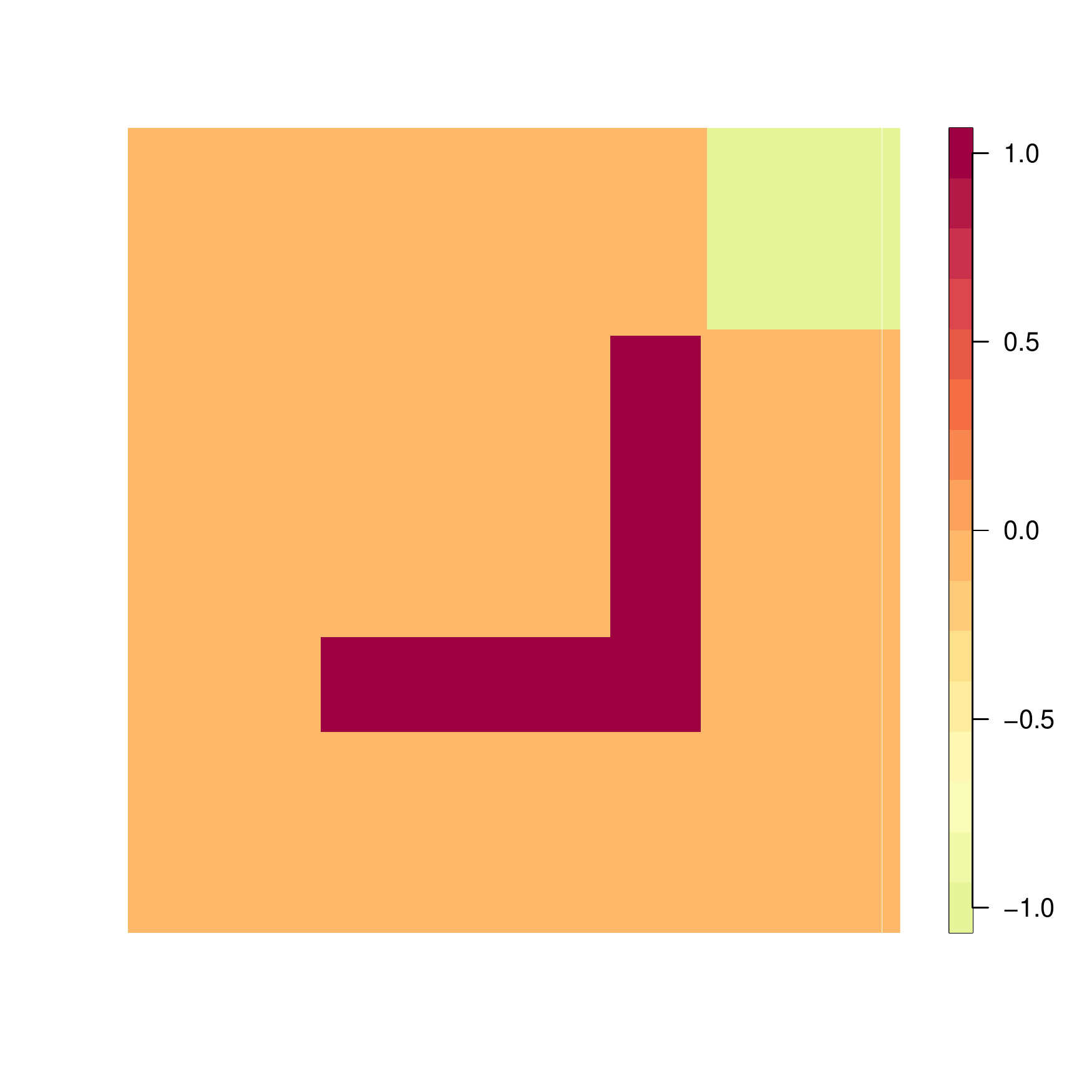} 
		\includegraphics[width=1.52in,height=1.52in]{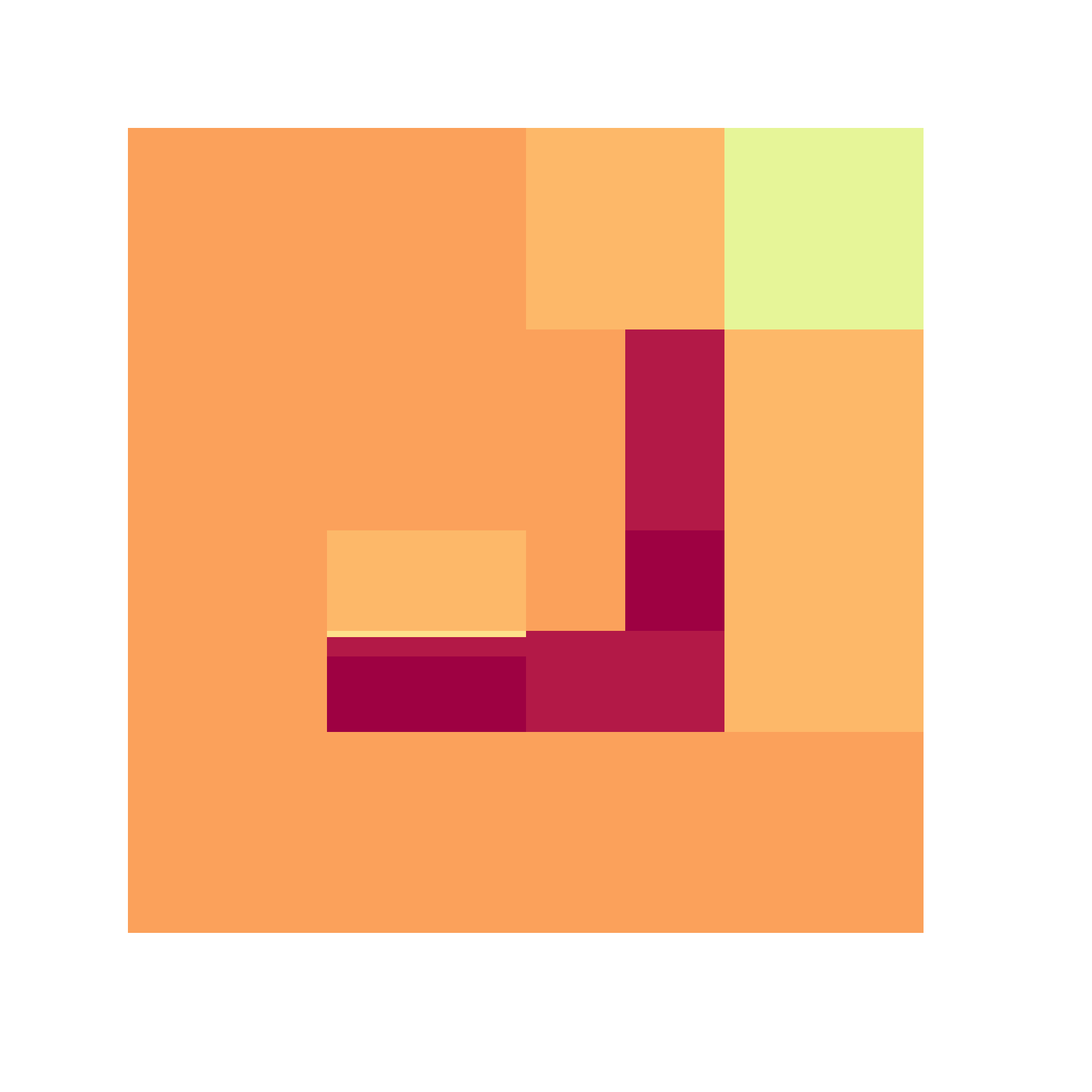} 
		\includegraphics[width=1.52in,height=1.52in]{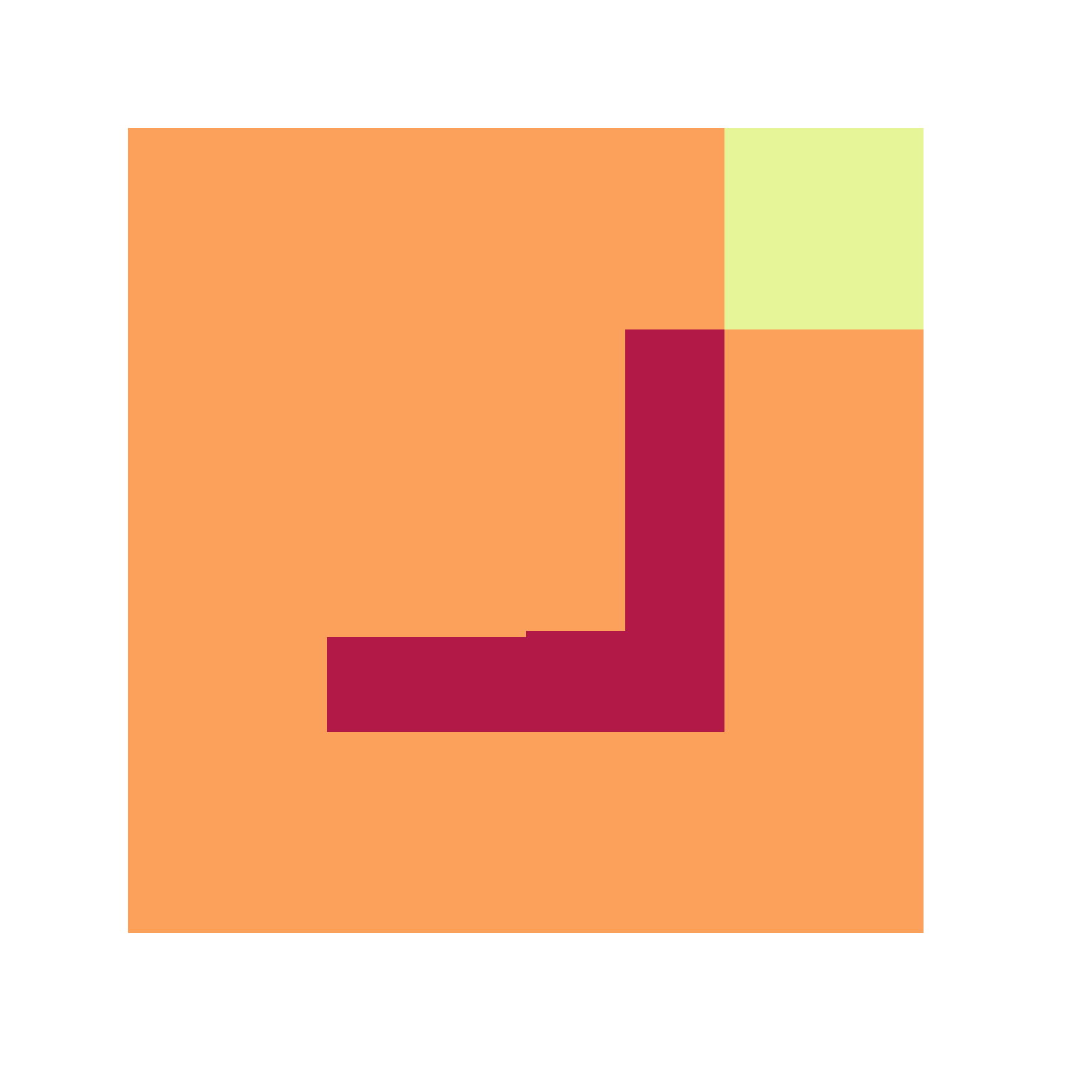} 		
		\includegraphics[width=1.52in,height=1.52in]{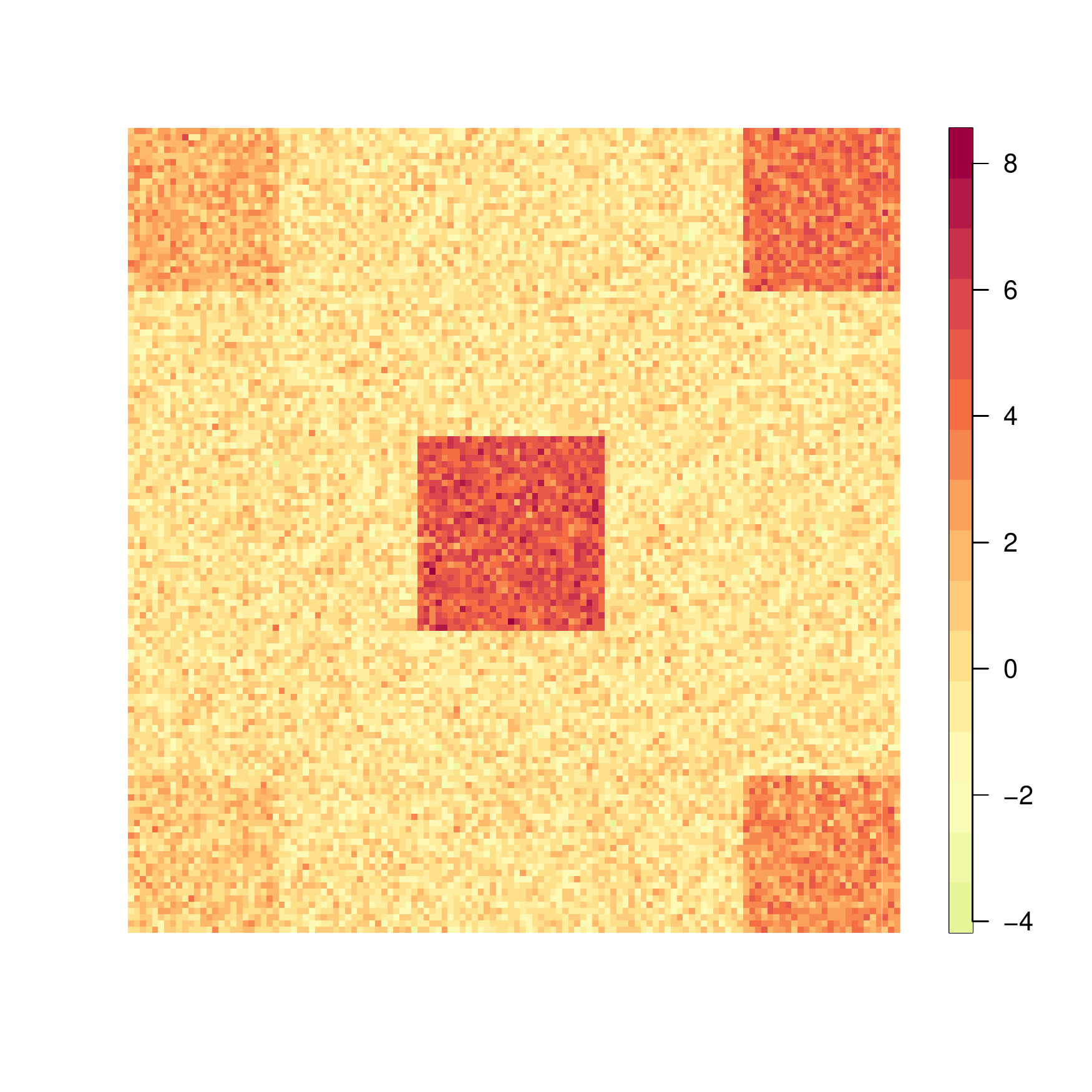} 
		\includegraphics[width=1.52in,height=1.52in]{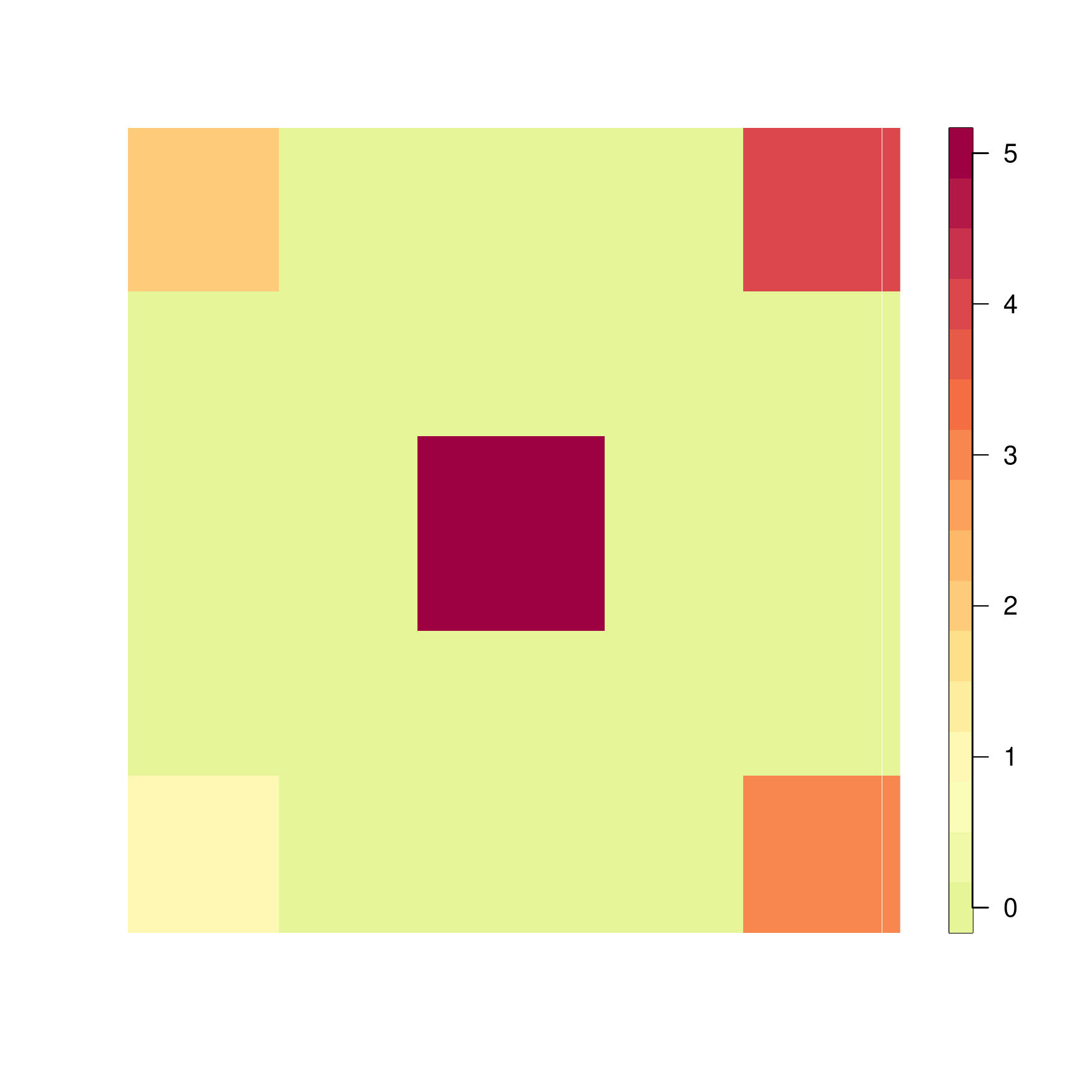} 
		\includegraphics[width=1.52in,height=1.52in]{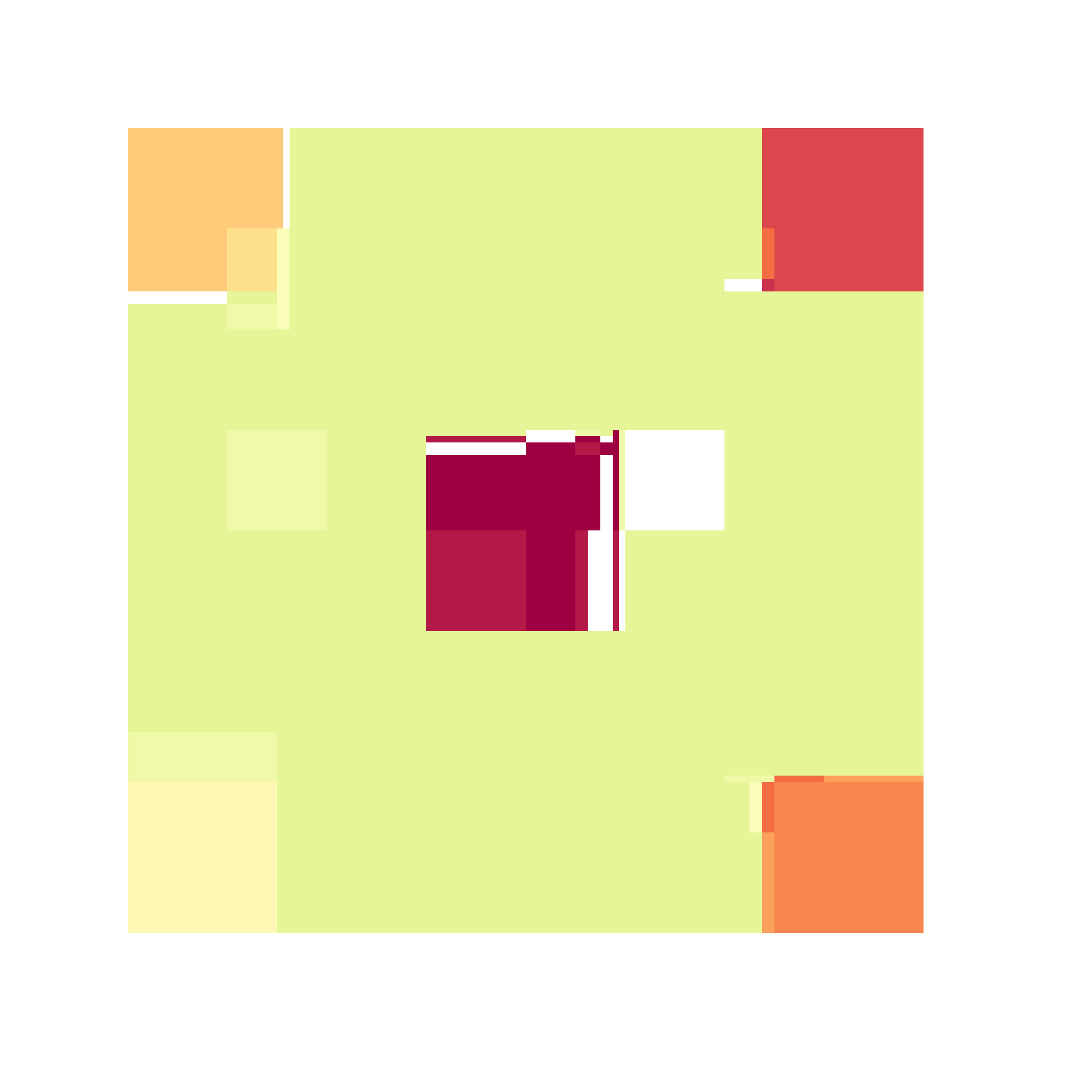} 
		\includegraphics[width=1.52in,height=1.52in]{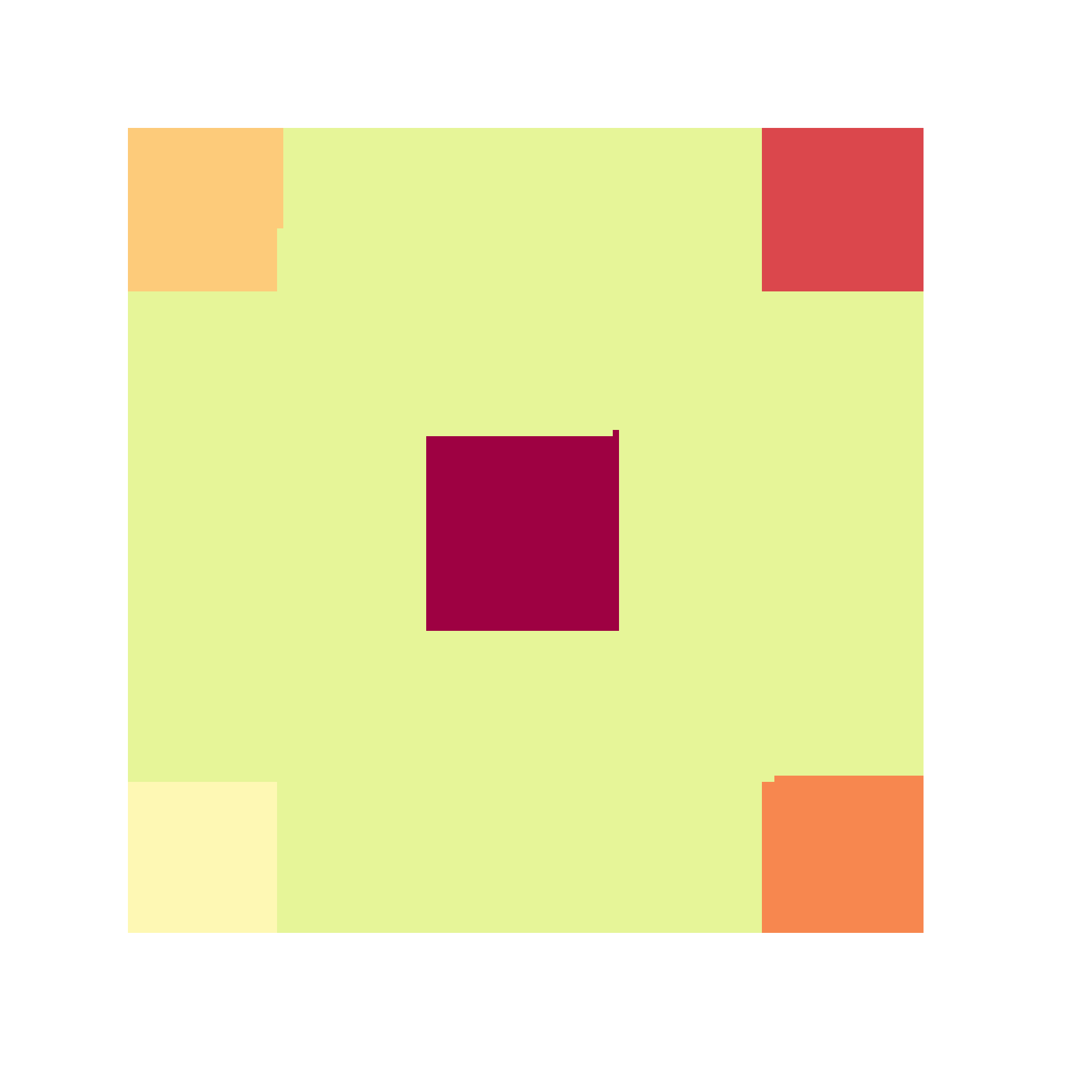} 			
		\caption{\label{fig1} From top to bottom: Scenarios 1 to 4.  From left to right: An instance of $y$, the signal $\theta^*$, DCART, and DCART after merging.  In each case  the data are generated with $\sigma =1$.}
	\end{center}
\end{figure}

For each scenario considered, we vary the noise level as $\sigma \in \{0.5,1,1.5\}$ and set  $(d, n) = (2, 2^7)$.  In each instance, the data are generated as $y \sim    \mathcal{N}(\theta^*,\sigma^2 I_{ L_{d,n}}  )$.  Detailed descriptions are in Section \ref{sec:scenarios}, and visualizations of the signal patterns are in the  second
column in \Cref{fig1}, while  the third and fourth columns depict $\tilde{\theta}$, the DCART estimator, and $\widehat{\theta}$, our  two-step estimator, respectively.  We can see that our two-step estimator correctly identifies the partition and improves upon DCART for the purpose of recovery partition.  It is worth noting that even when the partition is not rectangular, as shown in the second  row in \Cref{fig1}, our two-step estimator is still able to accurately recover a good rectangular partition.  

From Table \ref{tab1}  we see that   in terms of the metric $\text{dist}_2$ our two-step estimator outperforms TV-based estimator in all cases. Furthermore, the same is also true for most cases in terms of the metric $\text{dist}_1$.


\section{Conclusions}

In this paper we study the partition recovery problem over $d$-dimensional lattices. We show how a simple modification of DCART enjoys one-sided consistency.  To correctly identify the size of the true partition and obtain better consistency guarantees, we further propose a more sophisticated two-step estimator which is shown to be minimax optimal in some regular cases. 

Throughout this paper, we discuss partition recovery on lattice grids.  In fact, to deal with non axis-aligned data, one can construct a lattice in the domain of the features, average the observations within each cell of the lattice and ignore the cells without observations.   More details can be found in \Cref{sec-nonaxix} in the supplementary materials.

Finally,  an important open question remains regarding the necessity to include the factor $k_{\mathrm{dyad}}(\theta^*)$ in the signal-to-noise condition \eqref{eqn:snr} and in the estimation rates.   We leave for future work to investigate what the optimal estimation rates are in general  when $k_{\mathrm{dyad}}(\theta^*)$ is allowed to diverge.

\section*{Acknowledgment}


	Funding in direct support of this work: NSF DMS 2015489 and EPSRC EP/V013432/1.
	
\appendix

\section{Comparisons with some existing results}\label{sec-compar}

The procedures proposed in this paper all rely crucially on the DCART estimator  $\tilde{\theta}$ defined in \eqref{eqn:cart}.  As shown recently in \cite{chatterjee2019adaptive}, $\tilde{\theta}$ is such that $\mathbb{E}\{\|\tilde{\theta} - \theta^*\|^2\} \lesssim \sigma^2 k_{\mathrm{dyad}}(\theta^*) \log(N)$, a rate that is the minimax optimal.  \cite{chatterjee2019adaptive} also studies the de-noising performances of other rectangular partition estimators.   \cite{fan2018approximate} studied the de-noising performances of an $\ell_0$-penalized estimator for a structured signal supported over general graphs and obtained the same rates.  Both the DCART and the estimator proposed in \cite{fan2018approximate} are based on $\ell_0$-penalization.  A different approach is to instead rely on $\ell_1$-penalizations \citep[e.g.][]{tibshirani2011solution, sharpnack2012sparsistency}. 

In light of the de-noising rate, it is perhaps not surprising that the partition recovery estimation error rate of the DCART, shown in \Cref{thm1} is of order ``$\text{de-noising error bound}/\mbox{jump size}$'', but what is unsatisfactory for us is that when $d = 1$, this extra $k_{\mathrm{dyad}}(\theta^*)$ factor suggests the sub-optimality of the result.  For instance, both \cite{wang2020univariate} and \cite{verzelen2020optimal} showed that when $d = 1$, an $\ell_0$-penalized estimator is able to achieve a minimax optimal estimation error of order $\kappa^{-2}\sigma^2\log(N)$.  In \Cref{sec-optimal}, we have shown that the term $k_{\mathrm{dyad}}(\theta^*)$ can be avoided if further regularity condition is imposed.  It remains still an open problem without these regularity conditions, what the optimal estimation rate would be.  

It is also worth mentioning another stream of work, focusing on the detection boundary in detection a cluster of nodes in general graphs, including square lattices.  Although testing and estimation are two fundamentally different problems, often requiring different conditions, the detection boundaries derived thereof could be a useful reference evaluating the signal-to-noise ratio condition we impose in \eqref{eqn:snr}.  \cite{arias2008searching, arias2011detection, addario2010combinatorial}, among others, stated that the detection boundary, in our notation is $\kappa^2 \Delta \asymp $ a logarithmic term.  Such rate is derived for $k(\theta^*) = O(1)$ and suggests that our condition \eqref{eqn:snr} is optimal when $k(\theta^*) = O(1)$.  It remains an open problem to determine  the optimal estimation rate when $k(\theta^*)$ is allowed to diverge.

\section{Additional definitions}

We have repeatedly used a concept that two rectangles are adjacent.  In addition to the explanation in \Cref{def1}, we detail all the possible situations in \Cref{def-adjacent} below. 
\begin{definition}\label{def-adjacent}
	For two disjoint subsets $R_1, R_2 \subset L_{d, n}$, with $d > 1$,  $R_l =  \prod_{i=1}^{d}  [a^{(l)}_i,b_i^{(l)}]  $,  $l \in \{1,2\}$, we say that   $R_1$ and  $R_2$ are adjacent if there exists $i_0 \in [d]$,  such that  one of the following holds:
	\begin{itemize}
		\item  $b_{i_0}^{(1)}- a_{i_0}^{(2)}  = 1$ and   $\prod_{i \neq i_0} [a_i^{(1)},b_i^{(1)}] \subset  \prod_{i \neq i_0} [a_i^{(2)},b_i^{(2)}]$;
		\item $b_{i_0}^{(1)}- a_{i_0}^{(2)}  = 1$ and   $\prod_{i \neq i_0} [a_i^{(2)},b_i^{(2)}] \subset  \prod_{i \neq i_0} [a_i^{(1)},b_i^{(1)}]$;
		\item $a_{i_0}^{(1)} - b_{i_0}^{(2)} = 1$ and   $\prod_{i \neq i_0} [a_i^{(1)},b_i^{(1)}] \subset  \prod_{i \neq i_0} [a_i^{(2)},b_i^{(2)}]$;  
		\item $a_{i_0}^{(1)} - b_{i_0}^{(2)} = 1$ and   $\prod_{i \neq i_0} [a_i^{(2)},b_i^{(2)}] \subset  \prod_{i \neq i_0} [a_i^{(1)},b_i^{(1)}]$. 
	\end{itemize}
\end{definition}

\section{Proofs of main results}

This section contains the proofs of the main results from \Cref{sec-theory}.  \Cref{thm1} demonstrates the one-sided consistency of DCART. This result is not only interesting on its own but is also used repeatedly and in as essential away to prove two-sided consistency.  For readability, we express the two main claims of Theorem \ref{thm1}, namely \eqref{eq-thm1-1} and \eqref{eq-thm1-2}, as the  events
\begin{equation}\label{eq-a1-def}
	\mathcal{A}_1 \,=\,  \left\{ 	\sum_{j \in [k(\tilde{\theta})]} \vert R_j \backslash S_j \vert\leq C_3\kappa^{-2} \sigma^2  k_{\mathrm{dyad}}(\theta^*)\log(N)	\right\}
\end{equation}
and
\begin{equation}\label{eq-a2-def}
	\mathcal{A}_2\,=\, \left\{ \vert  R_j \backslash S_j \vert \leq C_4 \kappa_j^{-2}\sigma ^2   k_{\mathrm{dyad}}(\theta^*_{R_j}) \log(N), \quad j\in[k(\tilde{\theta})] \mbox{ and } R_j \setminus S_j \neq \emptyset \right\},
\end{equation}
respectively

\subsection{One-sided consistency of DCART}

\begin{proof}[Proof of \Cref{thm1}]
	For $j \in [k(\tilde{\theta})]$, if $R_j \setminus S_j \neq \emptyset$, then let $r_j$ be the smallest positive integer such that there exists a partition of $R_j$, namely $\{T_{j,1}, \ldots, T_{j, r_j }, S_j\}$ with $\theta_i^* =  a_{j,l}$, for all $i \in T_{j,l}$, $l \in [r_j]$.	
	
	Without loss of generality assume that $0 < \vert T_{j,1} \vert   \leq   \vert T_{j,2 }\vert  \leq    \ldots \leq  \vert T_{j,r_j} \vert \leq \vert S_j\vert$, for each $j \in [k(\tilde{\theta})]$.  Suppose that  $r_{j}$ is even.  Then
	\[
	\begin{array}{lll}
		\displaystyle 		\vert R_j \backslash S_j\vert   & = &\displaystyle    \sum_{l=1}^{  r_j/2  }   \vert  T_{j,   2l-1 }\vert    \,+\,\sum_{l=1}^{  r_j/2  }   \vert  T_{j,   2l }\vert \\
		& =&  \displaystyle    \sum_{l=1}^{  r_j/2  } \min\{ \vert  T_{j,   2l-1 }\vert ,\vert  T_{j,   2l }\vert  \}     \,+\,\sum_{l=1}^{  r_j/2 -1 }   \min\{\vert  T_{j,   2l }\vert , \vert  T_{j,   2l+1 }\vert \}   +    \min\{  \vert  T_{r_j} \vert ,\vert S_j\vert  \}\\
		& \leq& \displaystyle 2\sum_{l=1}^{r_j/2}   \frac{  \vert   T_{j,   2l-1 }\vert \,\vert  T_{j,   2l }\vert  }{\vert   T_{j,   2l-1 }\vert + \vert  T_{j,   2l }\vert }   +  2\sum_{l=1}^{r_j/2-1}   \frac{  \vert   T_{j,   2l }\vert \,\vert  T_{j,   2l +1}\vert  }{\vert   T_{j,   2l }\vert + \vert  T_{j,   2l+1 }\vert }   + 2 \frac{\vert T_{r_j} \vert   \, \vert   S_j \vert }{\vert T_{r_j} \vert   + \vert   S_j \vert}\\
		& \leq&  \displaystyle \frac{2}{\kappa_j^2} \sum_{l=1}^{r_j/2}   \frac{  \vert   T_{j,   2l-1 }\vert \,\vert  T_{j,   2l }\vert  }{ | T_{j,   2l-1 }\vert+\vert  T_{j,   2l }\vert } (  a_{j,2l-1 } - a_{j,2l }  )^2  \\
		&& \displaystyle + \frac{2}{\kappa_j^2} \sum_{l=1}^{r_j/2-1}   \frac{  \vert   T_{j,   2l}\vert \,\vert  T_{j,   2l+1 }\vert  }{\vert   T_{j,   2l}\vert +\vert  T_{j,   2l+1 }\vert } (  a_{j,2l } - a_{j,2l +1}  )^2  +   \frac{2}{\kappa_j^2}\frac{ \vert   S_j \vert \, \vert T_{r_j} \vert   }{\vert   S_j \vert + \vert T_{r_j} \vert  }  (     \bar{\theta}^*_{S_j}  -   a_{r_j}     )^2 \\
		& \leq&
		\displaystyle \frac{2}{\kappa_j^2}  \sum_{l=1}^{r_j/2}   \sum_{i   \in  T_{j,   2l-1 } \cup T_{j,   2l }} ( \theta^*_i -    \bar{\theta}^*_{  T_{j,   2l-1 } \cup T_{j,   2l }}    )^2  \\
		& &   \displaystyle +\frac{2}{\kappa_j^2}   \sum_{l=1}^{r_j/2-1} \sum_{i   \in  T_{j,   2l} \cup T_{j,   2l +1}} ( \theta^*_i -    \bar{\theta}^*_{  T_{j,   2l } \cup T_{j,   2l +1}}    )^2  +   \frac{2}{\kappa_j^2} \sum_{i \in S_j \cup    T_{r_j} } ( \theta_i^* -   \bar{\theta}^*_{   S_j  \cup    T_{r_j} }      )^2\\ 
		&\leq  & \displaystyle  \frac{2}{\kappa_j^2}  \sum_{l=1}^{r_j/2}   \sum_{i   \in  T_{j,   2l-1 }\cup T_{j,   2l } } ( \theta^*_i -    \bar{\theta}^*_{R_j}   )^2   \,+\,  \frac{2}{\kappa_j^2}  \sum_{l=1}^{r_j/2-1}   \sum_{i   \in  T_{j,   2l}\cup T_{j,   2l +1} } ( \theta^*_i -    \bar{\theta}^*_{R_j}   )^2  \\
		& &   \displaystyle + \frac{2}{\kappa_j^2} \sum_{i \in S_j \cup    T_{r_j} } ( \theta_i^* -   \bar{\theta}^*_{   R_j  }      )^2 \leq \frac{4}{\kappa_j^2}     \sum_{i \in R_j }    ( \theta_i^* -   \bar{\theta}^*_{   R_j  }      )^2,
	\end{array}
	\]
	where the first inequality follows from \Cref{lem1}.  The same bounds holds also when $r_j$ is odd.  Hence,
	\begin{equation}
		\label{eqn:important}	 \begin{array}{lll}
			\displaystyle 	\vert R_j \backslash S_j\vert    &\leq &  \displaystyle \frac{8}{\kappa_j^2}     \sum_{i \in R_j }  ( \theta_i^* -   \bar{y}_{   R_j  }      )^2  +   \frac{8}{\kappa_j^2}     \sum_{i \in R_j }  ( \bar{\theta}^*_{   R_j  } -   \bar{y}_{   R_j  }      )^2\\ 
			&= &\displaystyle \frac{8}{\kappa_j^2}  \sum_{i \in R_j }  ( \theta_i^* -   \tilde{\theta}_{  i  }      )^2  + \frac{8}{\kappa_j^2}   \vert R_j\vert( \bar{\theta}^*_{   R_j  } -   \bar{y}_{   R_j  }      )^2.
		\end{array}
	\end{equation}
	
	Let $\Omega_1$ and $\Omega_3$ be the events defined below in \eqref{eq-omega-1-def} and \eqref{omega3}, respectively.  In the event $\Omega_1 \cap \Omega_3$, the result \eqref{eq-thm1-2}, i.e.~the event $\mathcal{A}_2$ defined in \eqref{eq-a2-def}, is a direct consequence of \eqref{eqn:important}.  
	
	Let $\Omega_2$ be the event defined in \eqref{omega2}.  In the event $\Omega_1 \cap \Omega_2 \cap \Omega_3$, it follows from \eqref{eqn:important} that \eqref{eq-thm1-1}, i.e.~the event $\mathcal{A}_1$ defined in \eqref{eq-a1-def}, holds. To be specific, we have that
	\begin{align*}
		\sum_{j=1}^{k(\tilde{\theta})} \vert R_j \backslash S_j\vert &  \leq \sum_{j=1}^{  k(\tilde{\theta})  }\bigg[\frac{8}{\kappa^2}  \sum_{i \in R_j }  ( \theta_i^* -     \tilde{\theta}_{  i  }      )^2  + \frac{8}{\kappa^2}   \vert R_j\vert( \bar{\theta}^*_{   R_j  } -   \bar{y}_{   R_j  }      )^2\bigg] \\
		& \leq \frac{8}{\kappa^2} \|\theta^*   - \tilde{\theta}\|^2 + \frac{8 c_1 c_3 \sigma^2  k_{\mathrm{dyad}}(\theta^*) \log (N)   }{\kappa^2} \leq \frac{C_3 \sigma^2 \log(N) k_{\mathrm{dyad}}(\theta^*)}{\kappa^2}.
	\end{align*}
	
	Finally, note that the final theorem claim \eqref{eq-thm1-3} is shown in Lemma \ref{lem8}.
\end{proof}

\subsection{Two-sided consistency of DCART: a two-step constrained estimator}

\begin{proof}[Proof of \Cref{thm3}]
	The proof of (\ref{eqn:upper_3}) is identical to that of Theorem \ref{thm2} with one difference. The rectangular partition induced by $\widehat{\theta}$ is such that    
	\[
	\min\{  \vert R_i \vert, \vert R_j\vert\} \geq  \eta \geq c_1\frac{k_{\mathrm{dyad}}(\theta^*)\sigma^2 \log (N)}{\kappa^2}.
	\]
	As a result, we do not need to account for the term $\sum_{j  \,:\,  \vert  R_j\vert \leq  \eta }  \vert  R_j\vert$,
	and this is  the only part of the proof  of Theorem \ref{thm2}   that requires the stronger requirement in Assumption \ref{as4}. The rest of the proof goes through using Assumption \ref{as5}.
\end{proof}

\subsection{Optimality: A regular boundary case}

\begin{proof}[Proof of Corollary~\ref{cor1}]
	Let $\widehat{\theta}$ be the estimator of $\theta^*$ defined in \eqref{eqn:cart2}.  Let $\{R_l\}_{l \in [k(\widehat{\theta})]}$ be a rectangular partition of $L_{d, n}$ induced by $\widehat{\theta}$ and let $S_j$ be the largest subset of $R_j$ with constant $\theta^*$ value, for $j \in [k(\widehat{\theta})]$.  Let $\mathcal{J} \subset [k(\widehat{\theta})]$, such that $R_j \setminus S_j \neq \emptyset$, $j \in \mathcal{J}$.  With the notation in \Cref{thm1}, define the event $\mathcal{A}_3$ as
	\begin{equation} \label{eqn:part4}
		\mathcal{A}_3 = \left\{ \vert  R_j \backslash S_j \vert  \leq  C_4\frac{\sigma ^2   k_{\mathrm{dyad}}(\theta^*_{R_j}) \log(N)}{\kappa^2}, \quad j \in \mathcal{J} \right\} \cap  \left\{    k(\widehat{\theta}) \leq  c_1  k_{\mathrm{dyad}}(\theta^*) \right\}.
	\end{equation}
	It follows from \eqref{eq-thm1-2} that the event $\mathcal{A}_3$ holds with probability at least $1-N^{-c}$ for some  positive constants $c, C_4$ and $c_1$.  The rest of this proof is conducted in the event $\mathcal{A}_3$.
	
	For any $j \in \mathcal{J}$.  Let $A, B \in \Lambda^*$ be that $A \neq B$, $\bar{\theta}^*_A = \bar{\theta}^*_B$ and $(R_j \cap A)\cup(R_j \cap  B)  \subset S_j$.  Then it follows from an almost identical argument as that in \textbf{Step 1.1} in the proof of Theorem \ref{thm2}, we see that Assumption \ref{as5} leads to a contradiction.  It follows  that  $S_j $ is a connected set in $L_{d,n}$. Hence, we let $A\in \Lambda^*$  be such that  $R_j \cap  A = S_j $.
	
	Suppose now that $B \in \Lambda^*$  and  $R_j \cap  B \subset R_j \backslash S_j$. Since  
	\[
	\vert  R_j \backslash S_j \vert  \,\leq \,    C _4\frac{\sigma ^2   k_{\mathrm{dyad}}(\theta^*_{R_j}) \log (N) }{\kappa^2}   \,\leq \,  C _4\frac{\sigma ^2   k_{\mathrm{dyad}}(\theta^*) \log (N)}{\kappa^2},
	\]
	recalling that $\mathrm{dist}(A, B) = \min_{a \in A, b \in B}\|a - b\|$, it holds that 
	\[
	\mathrm{dist}(A,B)  \,\leq  \,  C_5 \frac{\sigma ^2   k_{\mathrm{dyad}}(\theta^*) \log (N)}{\kappa^2} 
	\]
	for some constant $C_5 > 0$. Hence, 
	\[
	\left\vert\{ B \in \Lambda^*\,:\,    R_j \cap B  \subset R_j \backslash S_j  \} \right\vert \,\leq \,   \left\vert \left\{  B  \in \Lambda^* \backslash\{A\} :\,   \mathrm{dist}(A,B) \leq  \frac{c \sigma^2 k_{\mathrm{dyad}}(\theta^*) \log (N)}{\kappa^2} \right\} \right\vert\,\,\leq C,
	\]
	where the second inequality follows from Assumption \ref{as6}. As a result,
	\begin{equation}
		\label{eqn:error}
		\vert  R_j \backslash S_j \vert  \,\leq \,    C _4C\frac{\sigma ^2 \log (N)}{\kappa^2}.
	\end{equation}
	
	It follows from an identical argument as that in \textbf{Step 5} in the proof of Theorem \ref{thm2} that, for any $A \in \Lambda^*$ there exists  $\hat{A}\in \widehat{\Lambda}$  such that 
	\begin{align} \label{eqn:error2}
		\vert  \hat{A} \backslash A\vert   \leq  &\displaystyle    \sum_{  j \in I_A } \vert R_j \backslash  S_j  \vert \leq  C _4C\frac{\sigma ^2 \log (N)}{\kappa^2}   \vert\left\{  j \, :\,  R_j\cap A = S_j\right\}\vert \nonumber \\
		\leq & C _4 C\frac{\sigma ^2 k(\widehat{\theta})\log (N)}{\kappa^2}  \leq C _6\frac{\sigma ^2 k_{\mathrm{dyad}}(\theta^*)\log (N)}{\kappa^2},
	\end{align}
	for some constant $C_6>0$, where the second inequality follows from (\ref{eqn:error}). Hence,
	\begin{align*}
		\vert\left\{  j \, :\,  R_j\cap A = S_j\right\}\vert & \leq \eta^{-1} \sum_{j \in I_A} \vert  R_j \vert \leq \vert \hat{A } \vert/\eta \leq (|A| + |\hat{A}\setminus A|)/\eta \\
		& \leq \vert  A\vert/\eta + C _6\frac{\sigma ^2 k_{\mathrm{dyad}}(\theta^*)\log (N)}{\eta\kappa^2}   \leq  \frac{ C' \vert  A\vert      }{\eta},
	\end{align*}
	where $C' > 0$ is an absolute constant.  Combining the above with (\ref{eqn:error2})  we arrive at 
	\[
	\vert  \hat{A} \backslash A\vert  \leq \frac{C'' \sigma ^2 \log (N)}{\kappa^2}  \frac{  \vert  A\vert      }{\eta}.
	\]
	
	To bound the difference from the other direction, we have that 
	\begin{align} \label{eqn:error4}
		& \vert  A \backslash \hat{A}\vert \leq \sum_{  j \,:\,  R_j \cap A \notin \{\emptyset,  S_j  \}  } \vert R_j \backslash  S_j  \vert \nonumber \\
		\leq & \sum_{j\,:\, \exists   B  \in \Lambda^*, \,\,\,  R_j\cap  B = S_j, \,\,\, \mathrm{dist}(A,B)\leq  c \sigma^2 \kappa^{-2} k_{\mathrm{dyad}}(\theta^*) \log (N)       }\vert R_j \backslash  S_j  \vert \nonumber \\
		\leq &  C _7\frac{\sigma ^2 \log (N)}{\kappa^2}  \,\left\vert \left\{ j\,:\, \exists   B  \in \Lambda^*, \,\,\,  R_j\cap  B = S_j, \,\,\, \mathrm{dist}(A,B)\leq  c \sigma^2 \kappa^{-2} k_{\mathrm{dyad}}(\theta^*) \log (N) \right\} \right\vert  \nonumber \\
		\leq &  C _7\frac{\sigma ^2 \log (N)}{\kappa^2}  \, k(\widehat{\theta}) \leq  C _8\frac{\sigma ^2 \log (N)}{\kappa^2}  \, k_{\mathrm{dyad}}(\theta^*),
	\end{align}
	for some constants $C_7,C_8>0$, where the second inequality follows from  Assumption \ref{as5}, and the last from  (\ref{eqn:error}). Hence,
	\[
	\begin{array}{l}
		\left\vert \left\{ j\,:\, \exists   B  \in \Lambda^*, \,\,\,  R_j\cap  B = S_j, \,\,\, \mathrm{dist}(A,B)\leq \frac{  c \sigma^2k_{\mathrm{dyad}}(\theta^*) \log (N)}{\kappa^3}  \right\} \right\vert  \\
		\leq  \displaystyle \frac{1}{\eta}\sum_{  B\in \Lambda^*,\,\,\mathrm{dist}(A,B)\leq c \sigma^2 \kappa^{-2} k_{\mathrm{dyad}}(\theta^*) \log (N)}\,\, \sum_{j \in I_B} \vert  R_j \vert\\
		= \displaystyle  \frac{1}{\eta}\sum_{ B\in \Lambda^*,\,\,\mathrm{dist}(A,B)\leq c \sigma^2 \kappa^{-2} k_{\mathrm{dyad}}(\theta^*) \log (N)}\,\, \vert  \hat{B}\vert\\
		\leq \displaystyle  \frac{1}{\eta}\sum_{   B\in \Lambda^*,\,\,\mathrm{dist}(A,B)\leq c \sigma^2 \kappa^{-2} k_{\mathrm{dyad}}(\theta^*) \log(N) }\,\, \left[ \vert B\vert + \vert  B \backslash \hat{B}\vert  \right]\\
		\lesssim\displaystyle   \underset{ B\in \Lambda^*,\,\,\mathrm{dist}(A,B)\leq c \sigma^2 \kappa^{-2} k_{\mathrm{dyad}}(\theta^*) \log(N)}{\max}\,   \frac{\vert  B\vert  }{\eta} \,    +\,       \underset{  B\in \Lambda^*,\,\,\mathrm{dist}(A,B)\leq c \sigma^2 \kappa^{-2} k_{\mathrm{dyad}}(\theta^*) \log(N)}{\max}\,   \frac{\vert  B \backslash \hat{B}\vert  }{\eta}   \\
		\lesssim \displaystyle      \underset{B\in \Lambda^*,\,\,\mathrm{dist}(A,B)\leq c \sigma^2 \kappa^{-2} k_{\mathrm{dyad}}(\theta^*) \log(N)  }{\max}\,   \frac{\vert  B\vert  }{\eta} \,    +\, \frac{\sigma ^2 \log (N) }{\eta\kappa^2}  \, k_{\mathrm{dyad}}(\theta^*)\\
		\lesssim    \displaystyle      \underset{ B\in \Lambda^*,\,\,\mathrm{dist}(A,B)\leq c \sigma^2 \kappa^{-2} k_{\mathrm{dyad}}(\theta^*) \log(N)}{\max}\,   \frac{\vert  B\vert  }{\eta}, 
	\end{array}
	\]
	where the third inequality follows from Assumption \ref{as6},   the fourth  from  (\ref{eqn:error4}), and the last from (\ref{eqn:snr4}).  We therefore have shown \eqref{eq-cor1-1}.  The claim \eqref{eq-cor1-2} is a straightforward consequence of \eqref{eq-cor1-1} by letting $\eta \asymp \vert A\vert \asymp \Delta$. 
\end{proof}

\begin{proof}[Proof of Proposition~\ref{prop-lb}]
	We are using Fano's method in this proof.  To be specific, we are to use the version of Lemma~3 in \cite{yu1997assouad}.   
	
	Without loss of generality, we assume that $\Delta^{1/d}$ is a positive integer.  For $q$ to be specified, we further assume that $(n - \Delta^{1/d})/q$ and $q$ are both positive integers.  We construct a collection of distributions, each of which is defined uniquely with a subset $S$ defined in \eqref{eq-lb-S}.  Therefore the collection of distributions can be specified by the collection of subsets
	\[
	\mathcal{S} = \left\{\prod_{p = 1}^d [k_pq, k_pq + \Delta^{1/d}], \quad (k_1, \ldots, k_d) \in [0, (n - \Delta^{1/d})/q]^d\right\}.
	\] 
	We assume that the parameters $\kappa, \sigma, \Delta$ in this collection of distributions ensure that this collection of distributions belong to the subset $\mathcal{P} \subset \mathcal{P}_N$,
	\begin{equation}\label{eq-lb-pf-p-def}
		\mathcal{P} = \left\{P^N_{\kappa, \Delta, \sigma}: \, \Delta^{2d} \leq N, \, \kappa^2 \Delta/\sigma^2 = \log(N) /6\right\}.
	\end{equation}

	To justify the conditions of Lemma~3 in \cite{yu1997assouad}, we first notice that for each $S \in \mathcal{S}$, $|S| = \Delta$.  Secondly, for any $S_1, S_2 \in \mathcal{S}$, $S_1 \neq S_2$, it holds that 
	\[
	|S_1 \triangle S_2| \geq 2\Delta^{\frac{d-1}{d}}q
	\]
	and
	\[
	\mathrm{KL}(P_{S_1}, P_{S_2}) \leq \Delta \kappa^2/\sigma^2.
	\]
	Lastly, we note that $|\mathcal{S}| = (n - \Delta^{1/d})^d/q^d$.
	Then Lemma~3 in \cite{yu1997assouad} shows that
	\[
	\inf_{\widehat{S}} \sup_{P \in \mathcal{P}_N} \mathbb{E}_P \left\{|\widehat{S} \triangle S|\right\} \geq \inf_{\widehat{S}} \sup_{P \in \mathcal{P}} \mathbb{E}_P \left\{|\widehat{S} \triangle S|\right\} \geq \Delta^{\frac{d-1}{d}}q \left(1 - \frac{\Delta \kappa^2/\sigma^2 + \log(2)}{\log\left\{(n - \Delta^{1/d})^d/q^d\right\}}\right).
	\]
	We now take $q = \Delta^{1/d}/2$, such that due to the conditions in \eqref{eq-lb-pf-p-def}, it holds that
	\[
	\Delta^{\frac{d-1}{d}} q = \Delta/2 = \frac{\sigma^2 \log(N)}{12 \kappa^2} 
	\] 
	and have that 
	\begin{align*}
		& \inf_{\widehat{S}} \sup_{P \in \mathcal{P}_N} \mathbb{E}_P \left\{|\widehat{S} \triangle S|\right\} \geq \frac{\sigma^2 \log(N)}{12 \kappa^2} \left(1 - \frac{\log(N)/3}{\log(N)/2}\right) \geq \frac{d\sigma^2 \log(n)}{36 \kappa^2},
	\end{align*}
	where the first inequality holds provided $6\log(2) \leq d \log(n)$	and the conditions specified in \eqref{eq-lb-pf-p-def}.
\end{proof}

\section{A naive two-step estimator}  \label{sec:two_step}

In \Cref{sec-two-sided}, we proposed and studied a two-step constrained estimator, which builds and improve upon the DCART estimator, leading to a two-sided consistency guarantee for recovering the support of the true partition.  The two-step estimator studied in \Cref{sec-two-sided} starts with a constrained DCART estimator and prunes its output by merging certain pairs of rectangles.  It is natural to ask about the performances of a naive two-step estimator, which just prunes the DCART estimator without constraining it to only output large enough rectangles.  In this section, we study thee performance of this simpler estimator, which turns out to be worse than the two-step estimator studied in \Cref{sec-two-sided}.  The proof of \Cref{thm2} is repeatedly used in the proofs of two of our main results,  \Cref{thm3} and Corollary~\ref{cor1}.

Instead of requiring \Cref{as5} as in \Cref{sec-two-sided}, we impose a stronger assumption below. 
\begin{assumption} \label{as4}
	If  $A,B \in  \Lambda^*$ with  $A \neq B$   and $\bar{\theta}^*_A = \bar{\theta}_B^*$, then 
	we have that 
	\[
	\mathrm{dist}(A,B) \, \geq  \, c\frac{k_{\mathrm{dyad}}(\theta^*)^2 \sigma^2\log(N)}{\kappa^2},
	\]
	for some large enough constant $c>0$.   Furthermore, we assume that 
	\begin{equation}\label{eq-snr-s1-supp}
		\kappa^2 \Delta \geq  ck_{\mathrm{dyad}}(\theta^*)^2 \sigma^2 \log(N).
	\end{equation}
\end{assumption}

We first detail the pruning step of the naive two-step estimator.  Let $\tilde{\theta}$ be the DCART estimator with tuning parameter $\lambda_1$, defined in \eqref{eqn:cart}.  Let $\{R_l\}_{l \in [k(\tilde{\theta})]}$  be a rectangular-partition of $L_{d,n}$ induced by $\tilde{\theta}$.  Let $\lambda_2, \eta, \gamma > 0$ be tuning parameters for the pruning stage.  For each $(i, j) \in [k(\tilde{\theta})] \times [k(\tilde{\theta})]$, let $z_{(i, j)} = 1$ if 
\[
\mathrm{dist}(R_i,R_j) \leq \gamma, \quad \min\{\vert  R_i \vert,  \vert   R_ j\vert  \} \geq \eta
\]
and
\[
\frac{1}{2}\left[    \sum_{l\in  R_i}   (Y_l -    \bar{Y}_{ R_i } )^2  +    \sum_{l\in  R_j}   (Y_l -    \bar{Y}_{ R_j } )^2 \right]  +  \lambda_2  \,>\,    \frac{1}{2} \sum_{l\in R_i\cup R_j  }   (Y_l -    \bar{Y}_{ R_i\cup R_j  } )^2;
\]
otherwise, let $z_{(i, )} = 0$.  With this notation, let $E = \{(i, j)\in [k(\tilde{\theta})] \times [k(\tilde{\theta})]:\, z_{(i,j)} = 1\}$ and let $\{\mathcal{C}_l\}_{l \in [\hat{L}]}$ be the collection of all the connected components of the undirected graph $G_{\mathrm{naive}} :=([k(\tilde{\theta})] \backslash\mathcal{I}, E)$, where $\mathcal{I} = \{i   \in  [k(\tilde{\theta})]:\,  \vert  R_i \vert  \leq \eta  \}$.  Then assign each element  $i \in \mathcal{I}$ at random to one  of the components $\{\mathcal{C}_l\}_{l \in [\widecheck{L}]}$ and denote the resulting collection as $\{\widecheck{\mathcal{C}}_l\}_{l \in [\widecheck{L}]}$.  Finally, define 
\begin{equation}\label{eq-tilde-lambda-supp}
	\widecheck{\Lambda}\, =\, \left\{  \cup_{j \in \widecheck{\mathcal{C}}_1} R_j,\ldots,\cup_{j \in \widecheck{\mathcal{C}}_{\widecheck{L} } } R_j \right\}.
\end{equation}

\begin{theorem} \label{thm2}
	Suppose Assumption~\ref{as4} holds and that the data satisfy \eqref{eq-model} and let $\widecheck{\Lambda}$ be the naive two-step estimator defined in \eqref{eq-tilde-lambda-supp}, with tuning parameters $\lambda_1 = C_1\sigma^2 \log(N)$, $\lambda_2 = C_2 k_{\mathrm{dyad}}(\theta^*) \sigma^2 \log (N)$, $\gamma = C_{\gamma} k_{\mathrm{dyad}}(\theta^*) \eta$ and 
	\begin{equation} \label{eqn:snr2}
		c_1  \frac{ k_{\mathrm{dyad}}(\theta^*)  \sigma^2 \log (N) }{\kappa^2} \leq  \eta \leq 	 \frac{\Delta }{c_2 k_{\mathrm{dyad}}(\theta^*)},
	\end{equation}
	where $C_1, C_2, C_{\gamma}, c_1, c_2 > 0$ are absolute constants.  Then, with probability at least  $1-N^{-c}$, it holds that 
	\[
	\vert \widecheck{\Lambda}\vert  \,=\,  \vert \Lambda^*\vert \quad \mbox{and} \quad d_{\mathrm{Haus}}(\widecheck{\Lambda}, \Lambda^*) \,\leq \,  C k_{\mathrm{dyad}}(\theta^*)\eta,
	\]
	where $c, C >0$ are absolute constants.   
	
	If in addition, it holds that $\eta = C_{\eta} \kappa^{-2}k_{\mathrm{dyad}}(\theta^*)  \sigma^2 \log (N)$, where $C_{\eta} > 0$ is an absolute constant, then, with proability at least $1-N^{-c}$,
	\[
	\vert \widecheck{\Lambda}\vert  \,=\,  \vert \Lambda^*\vert \quad \mbox{and} \quad d_{\mathrm{Haus}}(\widecheck{\Lambda}, \Lambda^*)  \,\leq \,  C\frac{ k_{\mathrm{dyad}}(\theta^*)^2  \sigma^2 \log (N)}{\kappa^2}.
	\]
\end{theorem}

\begin{proof}
	The proof is conducted in the events $\cap_{i \in [5]} \Omega_i \cap \mathcal{A}_1 \cap \mathcal{A}_2$, where $\Omega_1$ is defined in \eqref{eq-omega-1-def}, $\Omega_2$ is defined in \eqref{omega2}, $\Omega_3$ is defined in \eqref{omega3}, $\Omega_4$ is defined in \eqref{eq-def-omega-4}, $\Omega_5$ is defined in \eqref{eqn:cond3}, $\mathcal{A}_1$ is defined in \eqref{eq-a1-def} and $\mathcal{A}_2$ is defined in \eqref{eq-a2-def}.  For any $A \in \Lambda^*$, define $I_A = \{j \in [k(\tilde{\theta})] \setminus \mathcal{I}: R_j \cap A = S_j\}$. 
	
	\textbf{Step 1.}  Due to \eqref{eq-snr-s1-supp} and \eqref{eqn:snr2}, we have that $\mathcal{I} \neq \emptyset$, which implies that $I_A \neq [k(\tilde{\theta})]$ and there exists $i \notin I_A$.  Let $i \notin I_A$.  
	
	\textbf{Step 1.1.}  First, we claim that it is impossible that	$R_i \cap A \subset S_i$ and $R_i \cap A \neq S_i$, with   $\vert R_i \vert  \geq \eta $.  Arguing by contradiction,  assume that there exists $B \in \Lambda^* \backslash \{A\}$  such that  $R_i\cap  B \subset S_i$ and   $\bar{\theta}^*_A =  \bar{\theta}^*_B$. Set
	\[
	\mathcal{T} \,=\,\left\{ B  \in  \Lambda^* \backslash \{A\} \, :\,   R_i\cap B \subset S_i,\,\,\,  \bar{\theta}^*_A =  \bar{\theta}^*_B\right\},
	\]
	and let  $p \in  R_i \cap A$, $q \in \underset{B \in \mathcal{T}}{\bigcup} R_i\cap B$ be such that 
	\[
	\| p-q\|\,\,=\,\,\underset{\tilde{p} \in  A ,  \,\,\,  \tilde{q} \in \underset{B \in \mathcal{T}}{\bigcup} R_i \cap 
		B }{\min}\,	\| \tilde{p}-\tilde{q}\|.
	\]
	Then from Assumption \ref{as4} we have that
	\begin{equation}
		\label{eqn:aux1}
		\| p-q\|	\geq  \, c\frac{k_{\mathrm{dyad}}(\theta^*)^2 \sigma^2\log(N) }{\kappa^2}.
	\end{equation}
	Let   $r^1,\ldots,r^d \in L_{d,n}$  such that for  $a\in\{1,\ldots,d\}$,
	\[
	r^a_b  \,=\, \begin{cases}
		p_b    & \text{if}   \,\,\,b \neq a,\\
		q_b & \text{if}   \,\,\,b = a.\\
	\end{cases}
	\]
	By construction we have that $r^1,\ldots,r^d \in  R_i$.  Furthermore, from (\ref{eqn:aux1}) there exists a $a_0 \in \{1,\ldots,d\}$ such that  
	\begin{equation}
		\label{eqn:aux2}
		\| p-r^{a_0}\|	\geq  \, c\frac{k_{\mathrm{dyad}}(\theta^*)^2 \sigma^2\log (N)}{d^{1/2} \kappa^2}.
	\end{equation}
	By the definitions of $p$ and $q$, it holds that
	\[
	\{ \lambda p + (1-\lambda)r^{a_{0}} \,:\,  \lambda \in (0,1)    \} \cap L_{d,n}  \subset  R_i \backslash  S_i.
	\]
	It then follows from  (\ref{eqn:aux2})  that
	\[
	\vert  R_i \backslash  S_i  \vert  \geq   \left\vert \{ \lambda p + (1-\lambda)r^{a_{0}} \,:\,  \lambda \in (0,1)    \} \cap L_{d,n}  \right\vert \geq    c\frac{k_{\mathrm{dyad}}(\theta^*)^2 \sigma^2\log (N)}{\kappa^2},
	\]
	which contradicts  the definition of $\mathcal{A}_2$.  
	
	\textbf{Step 1.2.} If $R_i \cap A = S_i$, then $\vert R_i  \vert \leq \eta$.
	
	\textbf{Step 1.3.} If $R_i \cap A \neq S_i$ and $\vert R_i \vert \geq \eta$, then  by \textbf{Step 1.1}, it holds that $R_i \cap  A \subset  R_i \backslash S_i$.  Hence,  since $I_A$ induces a connected sub-graph of $G_{\mathrm{naive}}$, a fact proved below in \textbf{Step 3.}, we obtain that  $R_i \cap A = \emptyset$.
	
	\textbf{Step 2.}  We claim that $I_A \neq \emptyset$.  Proceeding again by contradiction, we assume that for any $j \in [k(\tilde{\theta})]$  with  $R_j  \cap  A \neq \emptyset$, it is either the case that  $\vert  R_j \vert \leq \eta$ or the case that  $R_j \cap A \neq  S_j$.  It follows from \textbf{Step 1.1.}~that it is impossible to have $R_j \cap A \subset S_j$,  $R_j \cap A \neq S_j$ and $\vert R_j\vert \geq \eta$.  Thus, we obtain that 
	\begin{equation} \label{eqn:ineq}
		\vert  A \vert \leq \sum_{  j \,:\, \vert R_j\vert  \leq    \eta    } \vert R_j \vert  \,+\,  \sum_{j=1}^{k( \tilde{\theta} )} \vert  R_j    \backslash S_j \vert \leq  k_{\mathrm{dyad}}(\theta^*)\eta  \,+\,   \frac{k_{\mathrm{dyad}}(\theta^*) \sigma^2 \log (N)}{\kappa^2},
	\end{equation}
	where the second  inequality  holds  due to the definitions of $\Omega_2$ and $\mathcal{A}_2$.  Since   $\vert A \vert  \geq \Delta$, (\ref{eqn:ineq}) along with we constraint (\ref{eqn:snr2})  lead to a contradiction.		
	
	\textbf{Step 3.}  We then claim that $I_A$ induces a connected sub-graph of $G_{\mathrm{naive}}$.  To see this, suppose that $\{J_u\}_{u \in [l]}$ are the connected components of $I_A$ with the edges induced by $E$ and with $l>1$.  Then, if $i \in  J_a$ and $j \in J_b$ for $a, b\in [l]$, $a \neq b$, then it must be the case that $\mathrm{dist}(R_i,R_j) \geq \gamma$. Hence, $\mathrm{dist}(R_i\cap A,R_j \cap A) \geq \gamma$ for all $i \in  J_a$, $j \in J_b$, $a, b \in [l]$, $a \neq b$.  Since $A$ is connected in $L_{d,n}$   we obtain that $\vert  A  \backslash\cup_{a=1}^l\cup_{i \in J_a} (R_i \cap A) \vert \,\geq \,  \gamma$.  However, 
	\begin{equation}
		\label{eqn:ineq2}
		\begin{array}{lll}
			\vert  A  \backslash\cup_{a=1}^l\cup_{i \in J_a}R_i \cap A \vert   & \leq&\displaystyle \sum_{  j \,:\, \vert R_j\vert  \leq    \eta    } \vert R_j \vert  \,+\,  \sum_{j=1}^{k( \tilde{\theta} )} \vert  R_j    \backslash S_j \vert\\
			& \lesssim&  \displaystyle  k_{\mathrm{dyad}}(\theta^*)\eta \,+\,   \frac{k_{\mathrm{dyad}}(\theta^*) \sigma^2 \log (N)}{\kappa^2} < \gamma,				
		\end{array}    
	\end{equation}
	where the second  inequality  holds due to the definitions of $\Omega_2$ and $\mathcal{A}_1$.			Thus, we have arrived at a contradiction. For any $i \in [k(\tilde{\theta})]$, let $\widecheck{A} \in \widecheck{\Lambda}$ be $i \in \widecheck{A}$.  We then have $R_i \subset \widehat{A}$.
	
	\textbf{Step 4.}  For any $(i, j) \in [k(\tilde{\theta})] \times [k(\tilde{\theta})]$, we discuss the following two cases.
	
	\textbf{Case 1.} If $i, j \in I_A$, then  (\ref{eqn:cond2})  holds   by  the definition of $\mathcal{A}_1$, and (\ref{eqn:cond3}) holds  by the definition of $\Omega_5$.  Hence, if  $\mathrm{dist}(R_i,R_j)\leq \gamma $ then $(i, j) \in E$.				
	
	\textbf{Case 2.} If  $i \in I_A$  and  $j \in I_B$ for $B \in  \Lambda^* \backslash\{A\}$ with  $\bar{\theta}^*_A \neq  \bar{\theta}^*_B$, then  by the definition of $\Omega_4$, we have that
	\[
	\begin{array}{lll}
		\displaystyle 	\frac{\vert R_i \vert   \,\vert R_j \vert }{  \vert   R_i\vert  +  \vert R_j\vert   }\left(\bar{Y}_{R_i} -  \bar{Y}_{R_j} \right)^2 &\geq &\displaystyle   \frac{\eta}{4}\kappa^2 -  C_{\gamma} k_{\mathrm{dyad}}(\theta^*) \sigma^2 \log (N)\geq C \lambda_2,
	\end{array}
	\]
	provided that (\ref{eqn:snr2}) holds  for large enough $c_1$  and $\lambda_2 =  C_2 k_{\mathrm{dyad}}(\theta^*)\sigma^2\log(N)$ for an appropriate constant. It follows that $\{i,j\}\notin E$.
	
	\textbf{Step 5.}  Combining all of the above we obtain that  $\vert \Lambda^*\vert   =  \vert  \tilde{\Lambda} \vert$.  Let $\widecheck{A} \in \widecheck{\Lambda}$ be 
	\[
	\widecheck{A} \in \argmin_{B \in \widecheck{\Lambda}} |B \triangle A|.
	\]
	We have that
	\[
	\begin{array}{lll}
		\displaystyle   \vert  \widecheck{A} \backslash A\vert   &\leq  &   \displaystyle    \sum_{  j \,:\, \vert R_j \vert  \leq    \eta  } \vert R_j \vert   \,+\,\sum_{  j \in I_A } \vert R_j \backslash  S_j  \vert \leq   k(\tilde{\theta}) \eta  \,+\,\sum_{  j =1}^{ k(\tilde{\theta})  } \vert R_j \backslash  S_j  \vert \\
		& \leq &  \displaystyle k_{\mathrm{dyad}}(\theta^*)\eta   \,+\,   C_3\frac{k_{\mathrm{dyad}}(\theta^*) \sigma^2 \log(N)}{\kappa^2},
	\end{array}
	\]
	where the last  inequality  holds  by   the definitions of $\Omega_2$ and $\mathcal{A}_1$; and 
	\[
	\begin{array}{lll}
		\displaystyle   \vert  A \backslash \widecheck{A}\vert   &\leq  &   \displaystyle    \sum_{  j \,:\, \vert R_j \vert  \leq    \eta    } \vert R_j \vert   \,+\,\sum_{  j \,:\,  R_j \cap A \notin \{\emptyset,  S_j  \}  } \vert R_j \backslash  S_j  \vert \leq k(\tilde{\theta}) \eta  \,+\,\sum_{  j =1}^{ k(\tilde{\theta})  } \vert R_j \backslash  S_j  \vert \\
		& \leq &  \displaystyle k_{\mathrm{dyad}}(\theta^*)\eta  \,+\,  C_3 \frac{k_{\mathrm{dyad}}(\theta^*) \sigma^2 \log(N)}{\kappa^2}.
	\end{array}
	\]	
	We therefore conclude the proof.
\end{proof}

\section{Auxiliary results}

\noindent {\bf Noise assumption.} In the paper we make the assumption of Gaussian  i.i.d.~errors in \eqref{eq-model}, just like in \cite{chatterjee2019adaptive}. This is a technical condition required to justify the use of Gaussian concentration inequality for Lipschitz functions. It may be relaxed by assuming errors with, e.g., log-concave density.
Furthermore, it is possible to consider sub-Gaussian errors but this would involve extra logarithmic  factors in the assumptions and upper bound. 

Additional lemmas are collected here.  Lemmas~\ref{lem1} and \ref{lem2} follow exactly from \cite{wang2020univariate}, so we omit their proofs.

\begin{lemma}[Lemma 5 in \cite{wang2020univariate}]\label{lem1}
	Let  $I ,  J \subset   L_{d,n}$  with $I \cap J =\emptyset$ and  let  $Y \in \mathbb{R}^{L_{d,n}}$. Then 
	\[
	\displaystyle \sum_{i\in I \cup J } (Y_i -\bar{Y}_{ I \cup J  })^2   \,=\,   \sum_{i\in I } (Y_i -\bar{Y}_{ I  })^2 \,+\, \sum_{i\in  J } (Y_i -\bar{Y}_{ J  })^2  +     \frac{ \vert I \vert \vert J\vert  }{\vert I \vert +\vert J\vert}\left(  \bar{Y}_{ I  }-\bar{Y}_{ J }\right)^2.
	\] 
\end{lemma}

\begin{lemma}[Lemma 6 in \cite{wang2020univariate}] \label{lem2}
	Let  $\mathcal{I}$  be the set of rectangles that are subsets of  $L_{d,n}$. Then for $y \in \mathbb{R}^{L_{d,n}}$ defined in \eqref{eq-model},  the event  
	\begin{equation}
		\label{eqn:beta}
		\mathcal{B} \,=\, \left\{ \displaystyle   \underset{I,J    \in \mathcal{I}, \,\,  I \cap J   = \emptyset }{\max}   \,\,\sqrt{   \frac{  \vert I \vert \,  \vert J\vert   }{   \vert I \vert  +  \vert J \vert  } }\left\vert    \bar{Y}_I   - \bar{\theta}^*_I   -   \bar{Y}_J   + \bar{\theta}^*_J  \right\vert    \,\leq  \, C_{\mathcal{B}} \sigma \sqrt{\log (N)}  \right\}
	\end{equation}
	holds with probability at least $1 -N^{-c_{\mathcal{B}}}$, where $C_{\mathcal{B}}$ is a large enough constant  and  $c_{\mathcal{B}}$ depends on~$C_{\mathcal{B}}$.
\end{lemma}

\begin{lemma}
	\label{lem7}
	Let  $R \subset  L_{d,n}$  be a rectangle and denote by $\mathcal{P}_{\mathrm{dyadic,d,n}}(R)$ the set of all dyadic  partitions of  $R$. Define  $\beta_R \in \mathbb{R}^R$ as $\beta_R	 = \widetilde{\Pi}_R(y)$ where
	\begin{equation}
		\label{eqn:cartR}
		\widetilde{\Pi}_R  \,\in \,\underset{  \Pi    \in \mathcal{P}_{\mathrm{dyadic, d,n}}(R)  }{\arg \min}\left\{\frac{1}{2} \|   y_R-  O_{S(\Pi)}(y_R) \|^2   +  \lambda  \vert \Pi\vert    \right\}.
	\end{equation}  
	Then there exist positive constants   $c_1$ and $c_2$ that depend on $d$ such that  if $\lambda = C \sigma^2 \log (N)$ for a large enough constant $C>0$ it follows that the event 
	\begin{equation}\label{eq-omega-1-def}
		\Omega_1 \,=\,\left\{ \underset{R \subset  L_{d,n},   \,\,R   \,\,\text{rectangle}   }{\max}\, \{ \|    \beta_R     -   \theta_R^*\|^2    -   4\lambda k_{\mathrm{dyad}}(\theta_R^*)\}    \,\leq\, c_1\sigma^2\log (N)  \right\}
	\end{equation}
	holds 
	with probability at least  $1-N^{-c_2}$.
\end{lemma}

\begin{proof}
	First, proceeding as in the proof of Theorem 8.1 in \cite{chatterjee2019adaptive}, we obtain that 
	\begin{align} \label{eqn:basic}
		& \|\beta_R - \theta_R^*\|^2 \leq 2 \lambda k_{\mathrm{dyad}}(\theta_R^*)   + 2  (y_R-\theta^*_R)^{\top}(  \beta_R  -\theta^*_R  )-  2\lambda k_{\mathrm{dyad}}(\beta_R)  \nonumber \\ 
		\leq & 2 \lambda k_{\mathrm{dyad}}(\theta_R^*)  +   \frac{1}{2} \|\beta_R - \theta_R^*\|^2    +   2 \left \{(y_R-\theta^*_R)^{\top}\frac{(\beta_R  -\theta^*_R)}{\|\beta_R  -\theta^*_R  \|}\right\}^2-  2\lambda k_{\mathrm{dyad}}(\beta_R).  
	\end{align}
	Next, we denote by  $\mathcal{S}_R$ the collection of linear subspaces of $\mathbb{R}^R$ such that every $S \in \mathcal{S}_R$ is a linear subspace of $\mathbb{R}^R$ such that there is a partition of $R$ and $S$ consists of piecewise  constant  signals over this partition of $R$. Then
	\begin{align} \label{eqn:sec}
		& \frac{1}{2}	\|\beta_R - \theta_R^*\|^2 - 2 \lambda k_{\mathrm{dyad}}(\theta_R^*) \nonumber \\ 
		\leq & \underset{ k \in [|R|]  }{\max}\,\underset{ S \in \mathcal{S}_R, \,\text{Dim}(S) = k }{\sup}\,\,\,\underset{v \in S, \, v\neq \theta^*_R }{\sup}\,\left\{   2  \left\{(y_R-\theta^*_R)^{\top}\frac{(  v  -\theta^*_R  )}{\|  v  -\theta^*_R  \|}\right\}^2-   2\lambda k_{\mathrm{dyad}} \right\}.
	\end{align}
	However, from Lemma 9.1 in \cite{chatterjee2019adaptive},  for any  $c_1>1$, $S \in \mathcal{S}_R$ with $\text{dim}(S) = k \in [|R|]$, we have that 
	\begin{align*}
		& \mathbb{P}\left(\underset{v \in S, v\neq \theta^*_R}{\sup}\,\left\{   2  \left\{(y_R-\theta^*_R)^{\top}\frac{(  v  -\theta^*_R  )}{\|  v  -\theta^*_R  \|}\right\}^2-   2\lambda k \right\} \,\geq c_1 \sigma^2 \log (N)\right)\\
		\leq & \mathbb{P}\left(\underset{v \in S, v \neq \theta^*_R}{\sup}\,  2  \left\{(y_R-\theta^*_R)^{\top}\frac{(  v  -\theta^*_R  )}{\|  v  -\theta^*_R  \|}\right\}^2    \,\geq c_1 \sigma^2 \log (N)  +     2\lambda   k \right) \\
		\leq & 2\exp\left(  -    \frac{    c_1/2\log (N)    +     (2\lambda/\sigma^2 -2)  k   -4}{8}   \right).
	\end{align*}
	Since  $\vert \{ S\in \mathcal{S}_R\,:\,   \,\text{Dim}(S) = k \}\vert  \leq     \vert  R\vert^{2k}$, it follows by a union bound argument that for some  $c_2>0$, 
	\begin{align} \label{eqn:high_prob}
		& \mathbb{P}\left(\underset{ S \in \mathcal{S}_R, \,\text{Dim}(S) = k }{\sup}\,\underset{v \in S, v \neq \theta^*_R}{\sup}\,\left\{   2\left\{(y_R-\theta^*_R)^{\top}\frac{(  v  -\theta^*_R  )}{\|  v  -\theta^*_R  \|}\right\}^2-   2\lambda k \right\} \,\geq c_1 \sigma^2 \log N\right) \nonumber \\
		\leq & \exp\left(   -c_2\log (N) \right),
	\end{align}
	provided that $\lambda =  C\sigma^2 \log (N)$ with a sufficiently large $C > 0$.  The claim follows from a union bound argument by combining (\ref{eqn:sec}), (\ref{eqn:high_prob}),  the fact that  there are most $N^{2}$   subrectangles of  $L_{d,n}$, and choosing  $c_1$ large enough.
\end{proof}

\begin{lemma} \label{lem8}  
	Let $\tilde{\theta}$ be the DCART estimator.  If $\lambda  =  C\sigma^2\log (N) $ for a large enough $C$, then there exist  positive constants  $c_3$ and  $c_4$ such that the event 
	\begin{equation}
		\label{omega2}
		\Omega_2 \,=\,\left\{  k(\tilde{\theta} )  \,\leq \,    2 k_{\mathrm{dyad}}(\theta^*)      +  c_3 \right\}
	\end{equation}
	holds with probability at least  $1- N^{-c_4}$.
\end{lemma}
\begin{proof}
	First notice that by the basic inequality (\ref{eqn:basic}), it holds that
	\begin{equation}\label{eq-lem8-1}
		\lambda	 k(\tilde{\theta}) 
		\leq  2 \lambda k_{\mathrm{dyad}}(\theta^*) \,+\,  2  \left \{(y-\theta^*)^{\top} \frac{(\tilde{\theta}  -\theta^*)}{\|\tilde{\theta}  -\theta^*   \|}\right\}^2-  \lambda k(\tilde{\theta}).
	\end{equation}
	Therefore,  from Lemma \ref{lem7}, choosing  $\lambda     = C\sigma^2 \log (N)$ with large enough  $C$ implies that 
	with probability at least $1-  N^{-c_2}$ the event 
	\[
	\Omega \,=\,\left\{  2  \left \{(y-\theta^*)^{\top}  \frac{(\tilde{\theta}  -\theta^*)}{\|\tilde{\theta}  -\theta^*   \|}\right\}^2-  \lambda k(\tilde{\theta})   \geq c_1 \sigma^2 \log (N) \right\}.
	\]
	holds.
	Considering \eqref{eq-lem8-1} on the event $\Omega$, we have that
	\[
	k(\tilde{\theta}) \,\leq\,   2k_{\mathrm{dyad}}(\theta^*)   +    \frac{2c_1}{C}
	\]
	and the claim follows.
\end{proof}

\begin{lemma}
	\label{lem9}
	The event 
	\begin{equation} \label{omega3}
		\Omega_3 \,=\,\left\{ \underset{R \subset  L_{d,n},   \,\,R   \,\,\text{rectangle}   }{\max}\, \   \vert R \vert \,\vert  \bar{\theta}^*_R  -  \bar{y}_R  \vert^2       \,\leq\, c_1 \sigma^2\log N  \right\}
	\end{equation}
	holds with probability  at least  $1- N^{-c_2}$  for some postive constants  $c_1$ and $c_2$.
\end{lemma}

\begin{proof}
	This follows immediately from the fact  that there are at most  $N^2$ rectangles, the Gaussian tail inequality and a union bound argument. 
\end{proof}

\begin{lemma} \label{lem4}
	With the notation of Theorem \ref{thm1}, we define the set $\mathcal{Q}_4 \subset [k(\tilde{\theta})] \times [k(\tilde{\theta})]$ as
	\begin{equation}\label{eqnb:cond}
		\mathcal{Q}_4 = \left\{(i, j): \, \bar{\theta}^*_{S_i} \neq \bar{\theta}^*_{S_j}, \, \vert  R_i \vert  \leq   2 \vert  S_i  \vert ,\, \vert  R_j \vert  \leq   2 \vert  S_j  \vert \right\}.
	\end{equation} 
	Define the event
	\begin{equation}\label{eq-def-omega-4}
		\Omega_4 = \left\{\frac{\vert R_i \vert   \,\vert R_j \vert }{  \vert   R_i\vert  +  \vert R_j\vert   }\left(\bar{Y}_{R_i} -  \bar{Y}_{R_j} \right)^2 \,\geq \,  \frac{\min\{ \vert  R_i  \vert , \vert  R_j\vert  \}  }{4}\kappa^2 -  C k_{\mathrm{dyad}}(\theta^*) \sigma^2 \log(N), \, \forall (i, j) \in \mathcal{Q}_4\right\},
	\end{equation}
	where $C > 0$ is an absolute constant.  Then there exists an absolute constant $c > 0$ such that the event $\Omega_4$ holds with probability at least $1 - N^{-c}$.
\end{lemma}

\begin{proof}
	The proof is conducted assuming the high-probability event  $\mathcal{B}$ defined in (\ref{eqn:beta}).  Now, any for $(i, j) \in \mathcal{Q}_4$, we have that
	\begin{align*}
		& \frac{\vert R_i \vert \,\vert R_j \vert }{  \vert   R_i\vert  +  \vert R_j\vert   }\left(\bar{Y}_{R_i} -  \bar{Y}_{R_j} \right)^2 = \frac{\vert R_i \vert   \,\vert R_j \vert }{  \vert   R_i\vert  +  \vert R_j\vert   }\left\{ -   \bar{\theta}^*_{R_j} +  \bar{\theta}^*_{R_i}   +   (  \bar{Y}_{R_i} -  \bar{\theta}^*_{R_i}      +  \bar{\theta}^*_{R_j}-  \bar{Y}_{R_j}     )\right\}^2 \\    
		\geq & \frac{\vert R_i \vert   \,\vert R_j \vert }{2(\vert R_i\vert  +  \vert R_j\vert)   }\left(  -   \bar{\theta}^*_{R_j} +  \bar{\theta}^*_{R_i}  \right)^2 - \frac{\vert R_i \vert   \,\vert R_j \vert }{  \vert   R_i\vert  +  \vert R_j\vert   }\left(    \bar{Y}_{R_i} -  \bar{\theta}^*_{R_i}      +  \bar{\theta}^*_{R_j}-  \bar{Y}_{R_j}     \right)^2 \\
		\geq & \frac{\vert R_i \vert   \,\vert R_j \vert }{2(\vert R_i\vert  +  \vert R_j\vert)}\left(  -   \bar{\theta}^*_{R_j} +  \bar{\theta}^*_{R_i}  \right)^2 - C_{\mathcal{B}}^2 \sigma^2 \log(N)\\
		= & \frac{\vert R_i \vert   \,\vert R_j \vert }{2(\vert   R_i\vert  +  \vert R_j\vert)}\left\{  \bar{\theta}^*_{S_i} - \bar{\theta}^*_{S_j} + (\bar{\theta}^*_{S_j} - \bar{\theta}^*_{R_j} +  \bar{\theta}^*_{R_i} - \bar{\theta}^*_{S_i})   \right\}^2 - C_{\mathcal{B}}^2 \sigma^2 \log (N) \\
		\geq & \frac{\vert R_i \vert   \,\vert R_j \vert }{4(\vert   R_i\vert  +  \vert R_j\vert)}\left(  \bar{\theta}^*_{S_i} - \bar{\theta}^*_{S_j}    \right)^2 -  \frac{\vert R_i \vert   \,\vert R_j \vert }{2(\vert   R_i\vert  +  \vert R_j\vert)}\left(   \bar{\theta}^*_{S_j}    -   \bar{\theta}^*_{R_j} +  \bar{\theta}^*_{R_i}   - \bar{\theta}^*_{S_i}     \right)^2 \\
		& \hspace{1cm} - C_{\mathcal{B}}^2 \sigma^2 \log(N) \\
		\geq & \frac{\vert R_i \vert   \,\vert R_j \vert }{4(\vert   R_i\vert  +  \vert R_j\vert)}\left(  \bar{\theta}^*_{S_i} - \bar{\theta}^*_{S_j}    \right)^2 -  	\frac{\vert R_i \vert   \,\vert R_j \vert }{  \vert   R_i\vert  +  \vert R_j\vert   }\left(   \bar{\theta}^*_{S_j}    -   \bar{\theta}^*_{R_j}      \right)^2 - \frac{\vert R_i \vert   \,\vert R_j \vert }{  \vert   R_i\vert  +  \vert R_j\vert   }\left(   \bar{\theta}^*_{S_i}    -   \bar{\theta}^*_{R_i}      \right)^2 \\
		& \hspace{1cm} - C_{\mathcal{B}}^2 \sigma^2 \log (N) \\   
		\geq & 	\frac{\min\{  \vert  R_i\vert,  \vert  R_j\vert   \}}{4} \kappa^2 -    \vert R_j\vert \left(   \bar{\theta}^*_{S_j}    -   \bar{\theta}^*_{R_j}      \right)^2  -    \vert R_i\vert \left(   \bar{\theta}^*_{S_i}    -   \bar{\theta}^*_{R_i}      \right)^2 - C_{\mathcal{B}}^2 \sigma^2 \log (N)\\  
		\geq & 	\frac{\min\{  \vert  R_i\vert,  \vert  R_j\vert   \}}{4} \kappa^2 -    \vert  R_j\vert \left\{   \frac{1}{\vert  S_j\vert }  \sum_{l \in S_j} (\theta^*_l -\bar{\theta}^*_{R_j}      )^2    \right\} - \vert  R_i\vert \left\{   \frac{1}{\vert  S_i\vert }  \sum_{l \in S_i} (\theta^*_{l}  -\bar{\theta}^*_{R_i}      )^2    \right\} \\
		& \hspace{1cm} - C_{\mathcal{B}}^2 \sigma^2 \log(N), 
	\end{align*}
	where the  first and third inequalities  follow from the inequality  $(a+b)^2 \geq  a^2/2 -b^2$, the second by the definition of $\mathcal{B}$ in \eqref{eqn:beta}, the fourth by the inequality $(a+b)^2 \leq  2a^2 +  2b^2$ and the sixth by Jensen's inequality.  Then in the event $\Omega_1$ defined in \eqref{eq-omega-1-def}, it from Lemma \ref{lem7} that,
	\begin{align*}
		& \frac{\vert R_i \vert   \,\vert R_j \vert }{  \vert   R_i\vert  +  \vert R_j\vert   }\left(\bar{Y}_{R_i} -  \bar{Y}_{R_j} \right)^2 \\
		\geq & \frac{\min\{  \vert  R_i\vert,  \vert  R_j\vert   \}}{4} \kappa^2 -   2  \sum_{l \in S_j} (\theta^*_{l}  -\bar{\theta}^*_{R_j})^2  -  2 \sum_{l \in S_i} (\theta^*_{l}  -\bar{\theta}^*_{R_i}      )^2    - C_{\mathcal{B}}^2 \sigma^2 \log (N)  \\  
		\geq & 	\frac{\min\{  \vert  R_i\vert,  \vert  R_j\vert   \}}{4} \kappa^2 -   4  \sum_{l \in S_j} (\theta^*_{l}  -\bar{y}_{R_j}      )^2  - 4 \vert  S_j\vert ( \bar{y}_{R_j}  -   \bar{\theta^*}_{R_j}   )^2 \\ 
		& \hspace{1cm}  -  4 \sum_{l \in S_i} (\theta^*_{l}  -\bar{y}_{R_i}      )^2 - 4 \vert  S_i\vert ( \bar{y}_{R_i}  -   \bar{\theta}^*_{R_i}   )^2    - C_{\mathcal{B}}^2 \sigma^2 \log(N) \\ 
		\geq & 	\frac{\min\{  \vert  R_i\vert,  \vert  R_j\vert   \}}{4} \kappa^2  -  C\sigma^2 k_{\mathrm{dyad}}(\theta^*) \log(N),
	\end{align*}
	for some constant  $C>0$, where the second inequality follows from the inequality  $(a+b)^2 \leq 2a^2 + 2b^2$, and the last one by Lemmas  \ref{lem7} and \ref{lem9}.
	The claim then follows.
\end{proof}

\begin{lemma} \label{lem10}
	With the notation of Theorem \ref{thm1}, we define the set $\mathcal{Q}_5 \subset [k(\tilde{\theta})] \times [k(\tilde{\theta})]$ as
	\begin{equation}\label{eqn:cond2}
		\mathcal{Q}_5 = \left\{(i, j): \, \bar{\theta}^*_{S_i} = \bar{\theta}^*_{S_j }, \, \vert  R_i \vert  \leq   2 \vert  S_i  \vert, \, \vert  R_j \vert  \leq   2 \vert  S_j  \vert\right\}.
	\end{equation}	
	Define the event 
	\begin{equation}\label{eqn:cond3}
		\Omega_5 = \left\{\frac{\vert R_i \vert   \,\vert R_j \vert }{  \vert   R_i\vert  +  \vert R_j\vert   }\left(\bar{Y}_{R_i} -  \bar{Y}_{R_j} \right)^2 \,\leq \,  C k_{\mathrm{dyad}}(\theta^*) \sigma^2 \log (N), \, \forall (i, j) \in \mathcal{Q}_5\right\},
	\end{equation}
	where $C > 0$ is an absolute constant.  Then there exists an absolute constant $c > 0$ such that $\Omega_5$  holds	with probability at least $1-N^{-c}$.
\end{lemma}

\begin{proof}
	We assume through that  that the high-probability event  $\mathcal{B}$  defined in (\ref{eqn:beta}) holds.  Let  $(i, j) \in \mathcal{Q}_5$. Then,
	\[
	\begin{array}{lll}
		&&\displaystyle \frac{\vert R_i \vert   \,\vert R_j \vert }{  \vert   R_i\vert  +  \vert R_j\vert   }\left(\bar{Y}_{R_i} -  \bar{Y}_{R_j} \right)^2   \\
		& \leq & \displaystyle	\frac{2\vert R_i \vert   \,\vert R_j \vert }{  \vert   R_i\vert  +  \vert R_j\vert   }\left(\bar{Y}_{R_i} -   \bar{\theta}^*_{R_i} - \bar{Y}_{R_j} +  \bar{\theta}^*_{R_j} \right)^2 \,+\,	\frac{2\vert R_i \vert   \,\vert R_j \vert }{  \vert   R_i\vert  +  \vert R_j\vert   }\left(\bar{\theta}^*_{R_j}  -   \bar{\theta}^*_{R_i}  \right)^2\\
		&\leq  & \displaystyle  	2 C_{\mathcal{B}} \sigma^2\log(N)\,+\,\frac{4\vert R_i \vert   \,\vert R_j \vert }{  \vert   R_i\vert  +  \vert R_j\vert   }\left(\bar{\theta}^*_{R_j}  -   \bar{\theta}^*_{S_j}  \right)^2 + 	\frac{4\vert R_i \vert   \,\vert R_j \vert }{  \vert   R_i\vert  +  \vert R_j\vert   }\left(\bar{\theta}^*_{R_i}  -   \bar{\theta}^*_{S_i}  \right)^2\\
		&\leq  & \displaystyle  	2 C_{\mathcal{B}} \sigma^2\log(N) \,+\,	\frac{4\vert R_j\vert }{\vert S_j \vert }\sum_{l \in  S_j} ( \theta_l^* -  \bar{\theta}^*_{  R_j  }  )^2\,+\,	\frac{4\vert R_i\vert }{\vert S_i \vert }\sum_{l \in  S_i} ( \theta_l^* -  \bar{\theta}^*_{  R_i  }  )^2\\
		&\leq  & \displaystyle  	2 C_{\mathcal{B}} \sigma^2\log(N) \,+\,	8\sum_{l \in  S_j} ( \theta_l^* -  \bar{\theta}^*_{  R_j  }  )^2\,+\,	8\sum_{l \in  S_i} ( \theta_l^* -  \bar{\theta}^*_{  R_i  }  )^2\\
		&\leq  & \displaystyle  	2 C_{\mathcal{B}} \sigma^2\log (N) \,+\,	16\sum_{l \in  S_j} ( \theta_l^* -  \bar{y}_{  R_j  }  )^2\,+\,	16\sum_{l \in  S_i} ( \theta_l^* -  \bar{y}_{  R_i  }  )^2\\
		& &\displaystyle  \,+\,16\vert S_j \vert (  \bar{\theta}^*_{  R_j  } -  \bar{y}_{  R_j  }  )^2\,+\, 16\vert S_i \vert (  \bar{\theta}^*_{  R_i  } -  \bar{y}_{  R_i  }  )^2.
	\end{array}
	\]
	The first and second inequalities use the trivial fact that $(a+b)^2 \leq 2 a^2 + 2b^2$, the second inequality uses the event $\mathcal{B}$ and the third follows from Lemma \ref{lem1}.
	Combining the above inequality with Lemmas  \ref{lem2}, \ref{lem7} and \ref{lem9} completes the proof.
\end{proof}

\section{Experiments section details}

\subsection{Scenarios}
\label{sec:scenarios}

We detail all the signal patterns considered in the simulations in \Cref{sec-numerical}.  All these scenarios are depicted in \Cref{fig1}.  

\textbf{Scenario 1}.  For all $(a, b) \in L_{2, n}$, let
\[
\theta^*_{(a,b)}   \,=\, \begin{cases}
	1  & \text{if} \,\,  \frac{n}{4}<a< \frac{3n}{4}\,\,\text{and}\,\,  \frac{n}{4}<b< \frac{3n}{4},\\
	0  & \text{otherwise}.
\end{cases}
\]

\textbf{Scenario 2}. For all $(a, b) \in L_{2, n}$, let
\[
\theta^*_{(a,b)}   \,=\, \begin{cases}
	1  & \text{if} \,\,   (a- \frac{n}{4})^2 + (b- \frac{n}{4})^2  <   \left(\frac{n}{5}\right)^2,      \\
	1  & \text{if} \,\,   (a- \frac{3n}{4})^2 + (b- \frac{3n}{4})^2  <   \left(\frac{n}{5}\right)^2,      \\
	0  & \text{otherwise}.
\end{cases}
\]

\textbf{Scenario 3}.  For all $(a, b) \in L_{2, n}$, let
\[
\theta^*_{(a,b)}   \,=\, \begin{cases}
	1  & \text{if} \,\,     a \in (\frac{n}{4},\frac{3n}{4})\,\,\text{and}\,\, b \in (\frac{n}{4},\frac{3n}{8}    ),  \\
	1  & \text{if} \,\,   a \in (\frac{5n}{8},\frac{3n}{4})\,\,\text{and}\,\, b \in [\frac{3n}{8},\frac{3n}{4}    ),   \\
	-1 &\text{if }\,\,  a > \frac{3n}{4}\,\,\text{and}\,\, b > \frac{3n}{4},  \\
	0  & \text{otherwise}.
\end{cases}
\]

\textbf{Scenario 4}.  For all $(a, b) \in L_{2, n}$, let
\[
\theta^*_{(a,b)}   \,=\, \begin{cases}
	1  & \text{if} \,\,    a  < \frac{n}{5}    \,\,\text{and}\,\, b <\frac{n}{5},   \\
	2  & \text{if} \,\,    a  < \frac{n}{5}    \,\,\text{and}\,\, b >\frac{4n}{5},   \\
	3 & \text{if} \,\,    a  >\frac{4n}{5}    \,\,\text{and}\,\, b <\frac{4n}{5},  \\
	4 & \text{if} \,\,    a  >\frac{4n}{5}    \,\,\text{and}\,\, b >\frac{4n}{5},   \\
	5 & \text{if} \,\,    a  \in  (\frac{3n}{8},\frac{5n}{8})   \,\,\text{and}\,\, b \in  (\frac{3n}{8},\frac{5n}{8}),    \\
	0  & \text{otherwise}.
\end{cases}
\]

\subsection{Tuning parameters for naive two step-estimator}
\label{sec:tuning_parameters}

We first construct a sequence of DCART estimators $\tilde{\theta}(\lambda)$, $\lambda \in \mathcal{S}_{\lambda} = \{5 +  (30-5)l/14  , \, l = 0, \ldots, 14\}$.   Indexing the nodes in $L_{d, n}$ as $\{i_1,\ldots, i_{n^2}\}$, we calculate
\[
\hat{\sigma}^2 \,=\, (2n^2)^{-1} \sum_{ j \in [n^2-1]} (y_{i_j} -   y_{i_{j+1}} )^2.
\]
Based on this variance estimator, we choose 
\[
\lambda_1 = \argmin_{\lambda \in \mathcal{S}_{\lambda}}\left[ \sum_{ i \in L_{d,n}   } \{y_i -   \tilde{\theta}_i(\lambda)\}^2  +   \hat{\sigma}^2 k(\tilde{\theta}(\lambda)) \log (N)\right] \quad \mbox{and} \quad \tilde{\theta} = \tilde{\theta}(\lambda_1).   
\]
Once  $\tilde{\theta}$  is computed, in the second step, we  construct the final estimator denoted here as $\widehat{\Lambda}$ by setting  $\lambda_2 = \lambda_1$, $\gamma = 2^3$ and $\eta  = 2^3$ (see Section \ref{sec:two_step}).  The choice  $\lambda_2 = \lambda_1$ is consistent with the theory, since in all the scenarios  considered here $k_{\mathrm{dyad}}(\theta^*)$ is small.

\subsection{Implementation details of total variation based estimator} \label{sec:tv}

We now discuss the implementation details for the total variation based estimator  used in our experiments. Starting from  the $L_{d,n}$ lattice, we let $D$   be an incidence matrix corresponding to $L_{d,n}$,  see for instance  \cite{tibshirani2011solution}. We then compute, using the algorithm from \cite{tansey2015fast},  the estimators 
\[
\beta_{\lambda}\,=\,  \underset{\beta \in   \mathbb{R}^{L_{d,n} } }{\arg \min}\,\,\left\{  \frac{1}{2}\|\beta - y\|^2  \,+\,   \lambda \| D\beta\|_1 \right\}
\]  
for $\lambda \in       \{  10^{  3l/19 }  \,:\,   l = 0,1,\ldots, 19  \}$. Then letting   $\hat{\sigma}^2$ as in Section \ref{sec-numerical}, we  let 
\[
\lambda^*\,=\,   \underset{\lambda  \in  \{  10^{  3l/19 }  \,:\,   l = 0,1,\ldots, 19  \}  }{\arg \min}\,\left\{ \|  \beta_{\lambda}    -y\|^2   \,+\,     \hat{\sigma}^2 c(\beta_{\lambda}) \log (N)  \right\}
\]
where  $c(\beta_{\lambda})$ is the number of connected components  in $L_{d,n}$ induced by  $\beta_{\lambda}$.  In other words,  $c(\beta_{\lambda})$ is the estimated  degrees of freedom in the model associated with $\beta_{\lambda}$ in the language of  \cite{tibshirani2012degrees}. Then we set $\hat{\beta} $ equal to  $ \beta_{\lambda^*}$ after rounding each entry of  $ \beta_{\lambda^*}$ to three decimal digits.

Next, let $\{R_l\}_{l \in [q]}$  be the partition of $L_{d,n}$ induced by $\hat{\beta}$, $\eta = \gamma = 8$ and  $a =0.15$.  For each $(i, j) \in [q] \times [q]$, let $z_{(i, j)} = 1$ if 
\[
\mathrm{dist}(R_i,R_j) \leq  \gamma, \quad \min\{\vert  R_i \vert,  \vert   R_ j\vert  \} \geq \eta
\]
and  $\vert   \bar{Y}_{ R_i }   - \bar{Y}_{ R_j } \vert <a$; otherwise, let $z_{(i, j)} = 0$.  With this notation, let $E = \{ e \in [q] \times [q]:\, z_e = 1\}$ and let $\{\mathcal{C}_l\}_{l \in [\hat{L}]}$ be the collection of all the connected components of the undirected graph $([q] \backslash\mathcal{I}, E)$, where $\mathcal{I} = \{i   \in  [q]:\,  \vert  R_i \vert  \leq \eta  \}$. 
Our final estimator becomes
\begin{equation}\label{eq-tilde-lambda-supp2}
	\Lambda^{\prime }\, =\, \left\{  \cup_{j \in \mathcal{C}_1} R_j,\ldots,\cup_{j \in \mathcal{C}_{ \hat{L} } } R_j \right\}.
\end{equation}
Notice that   in (\ref{eq-tilde-lambda-supp2})  we do not include  the  sets $R_j$ with a small number of elements as we found that by using them the  performance of the estimator  becomes worst.

\subsection{Additional scenario}

In this subsection we consider an additional scenario, namely Scenario 5. For all $(a, b) \in L_{2, n}$,  we let 
\[
\theta^*_{(a,b)}   \,=\, \begin{cases}
	2  & \text{if} \,\,   a<\frac{n}{5}  \,\,\,\text{and}\,\,\,b>\frac{2n}{5},\\
	3  & \text{if} \,\,  a>\frac{4n}{5}  \,\,\,\text{and}\,\,\,b<\frac{3n}{5},\\
	4 & \text{if} \,\,   \vert  a - \frac{n}{2}\vert  <\frac{n}{4.5}  \,\,\,\text{and}\,\,\,b<\frac{n}{4.5},\\
	0  & \text{otherwise}.
\end{cases}
\]

\begin{figure}[h!]
	\begin{center}
		\includegraphics[width=1.52in,height=1.52in]{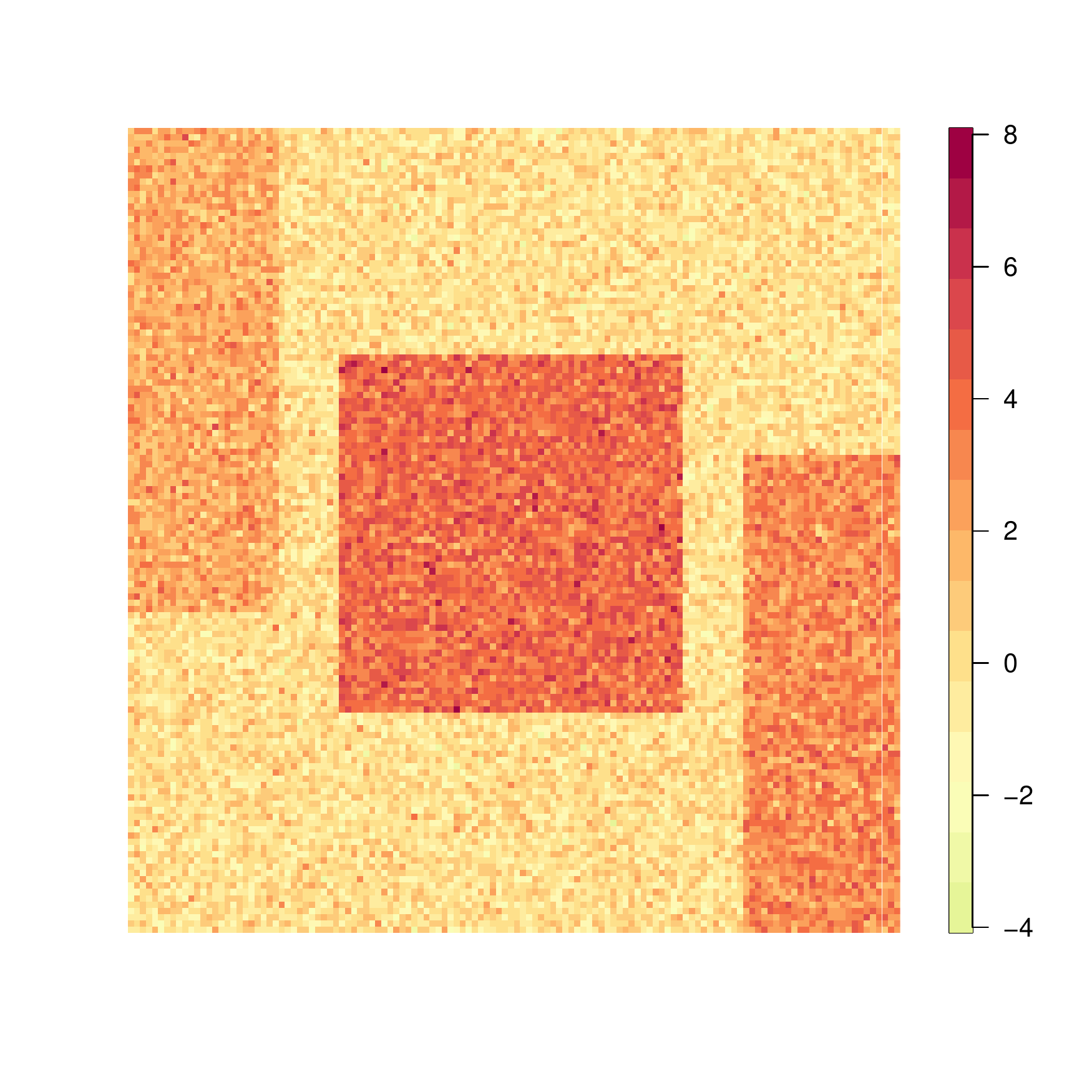} 
		\includegraphics[width=1.52in,height=1.52in]{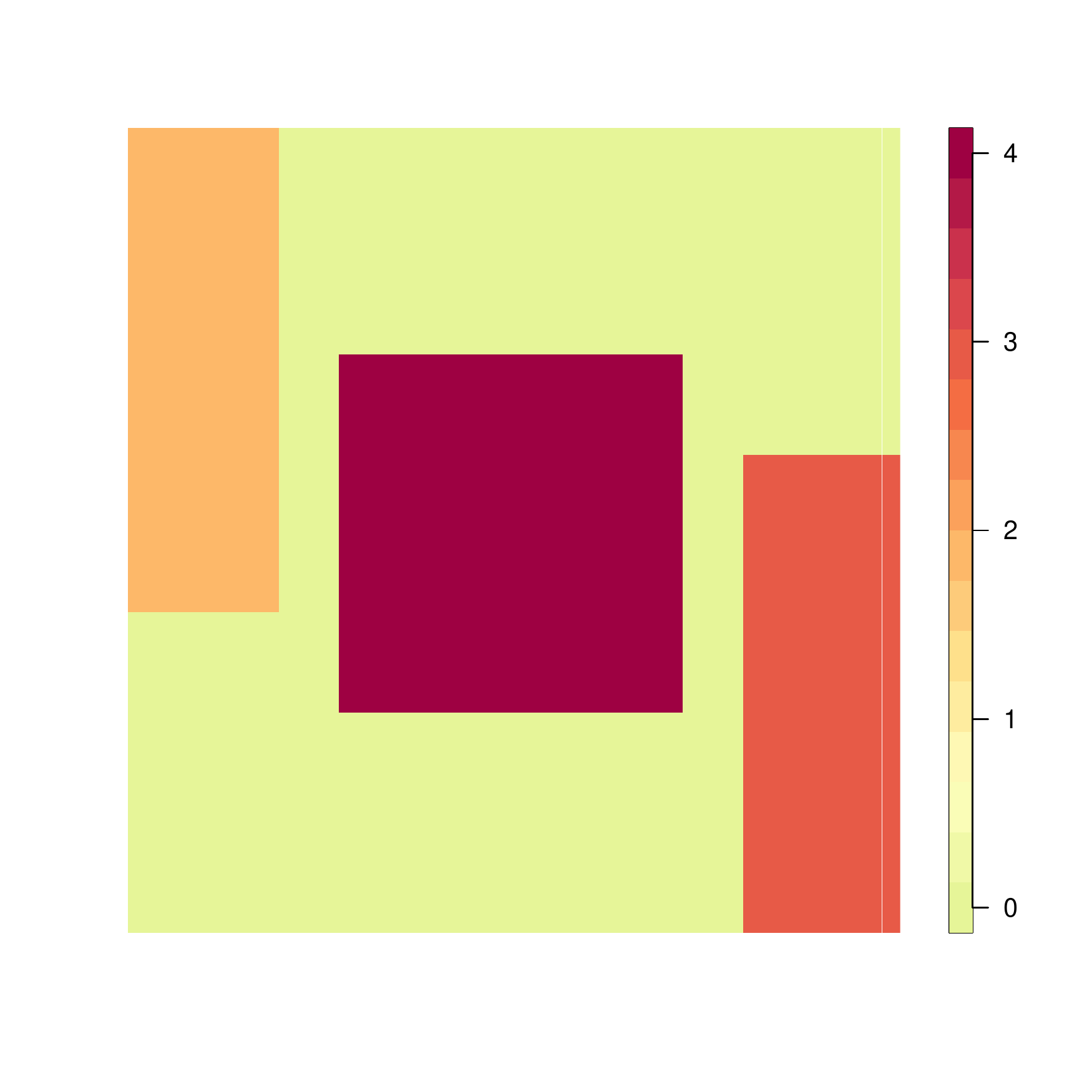} 
		\includegraphics[width=1.52in,height=1.52in]{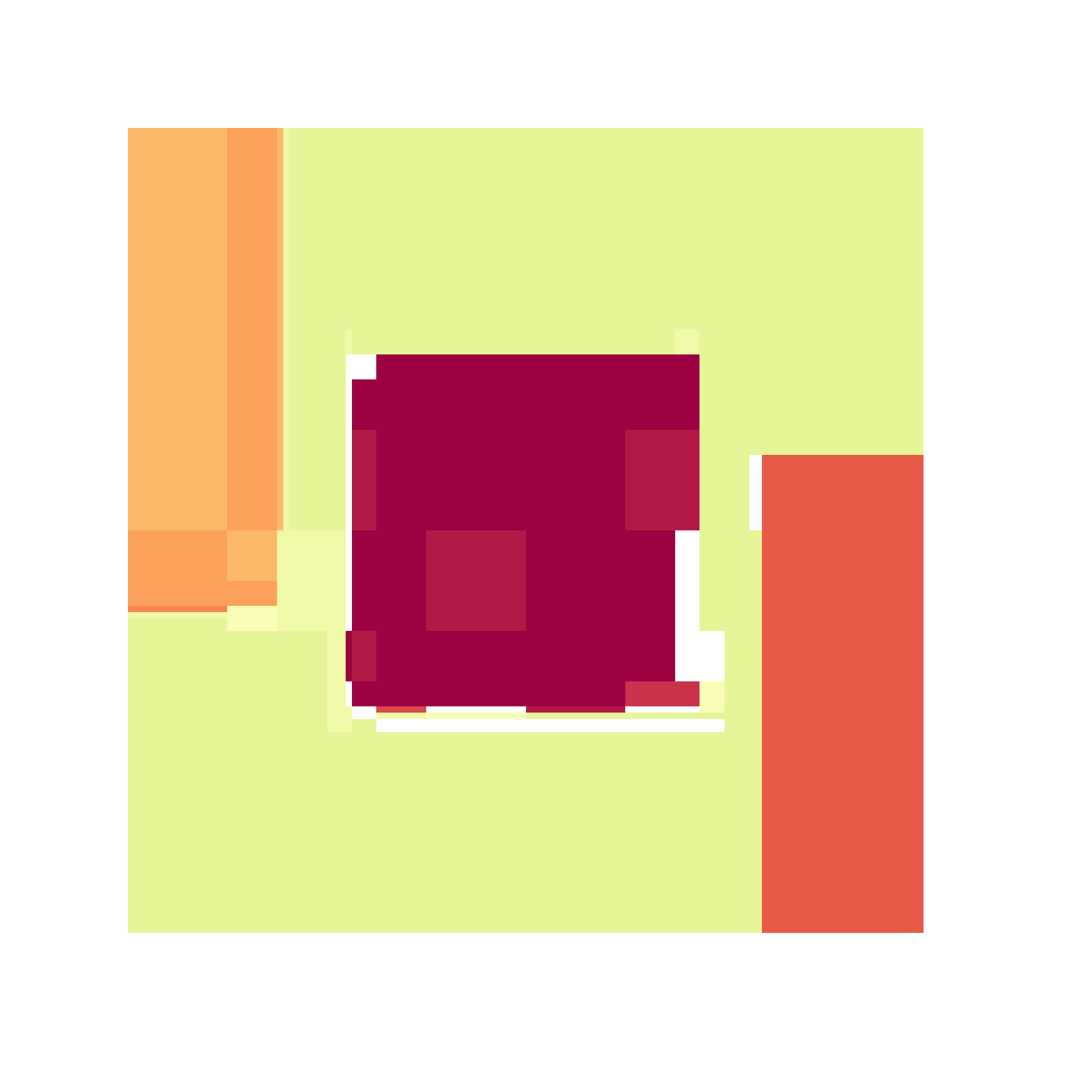} 
		\includegraphics[width=1.52in,height=1.52in]{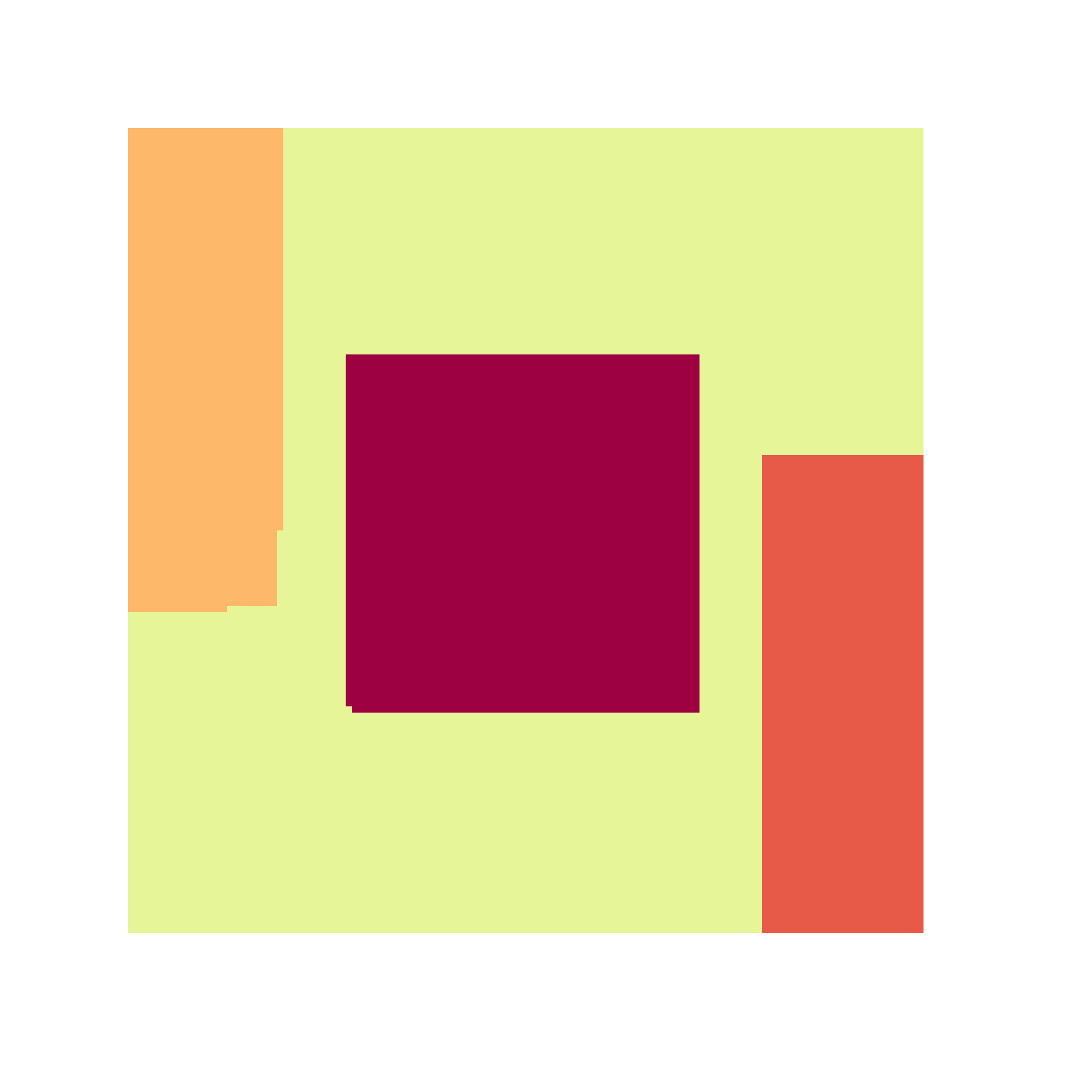}
		\caption{\label{fig1} Visualization of  Scenario 5.  From left to right: An instance of $y$, the signal $\theta^*$, DCART, and DCART after merging.  In this example the data are generated with $\sigma =1$.}
	\end{center}
\end{figure}

\begin{table}[t!]
	\centering
	\caption{\label{tab2} Performance  evaluations   over 50 repetitions under Scenario 5. The performance metrics $\text{dist}_1$ and $\text{dist}_2$ are defined in the text.   The numbers in parenthesis denote standard errors. }
	\medskip
	\setlength{\tabcolsep}{3.4pt}
	\begin{scriptsize}
		\begin{tabular}{|rr|rr|rr|rr|rr|rr|} 
			\hline
			\multicolumn{2}{|c|}{Setting} & 	\multicolumn{2}{c|}{$\text{dist}_1$}    & 	\multicolumn{2}{c|}{$\text{dist}_2$}\\
			&  $\sigma$            &  $\widehat{\Lambda}$                      &TV-based          &$\widehat{\Lambda}$   &     TV-based    \\
			\hline	   
			$5$            &       $0.5$ &              211.68(745.74)                                    &   \textbf{130.72(44.49)}   &    \textbf{ 0.0(0.27)}      &   0.12(0.33)    \\ 
			$5$            &       $1.0$ &        766.84(1281.38)           &   \textbf{398.92(362.39)}&   \textbf{0.0(0.57)}        &   1.32(0.9) \\
			$5$            &       $1.5$ &            1406.96(1589.79)                                   &\textbf{921.08(214.55)}  &      \textbf{0.0(0.65)}    &  2.68(1.31) \\
			\hline
			\hline  
		\end{tabular}
	\end{scriptsize}
\end{table}

Performance evaluations for Scenario 5 are given in Table \ref{tab2}. There, we can see that our proposed method provides the best estimation of the number of piecewise constant regions.

\begin{figure}[h!]
	\begin{center}
		\includegraphics[width=2.7in,height=2.2in]{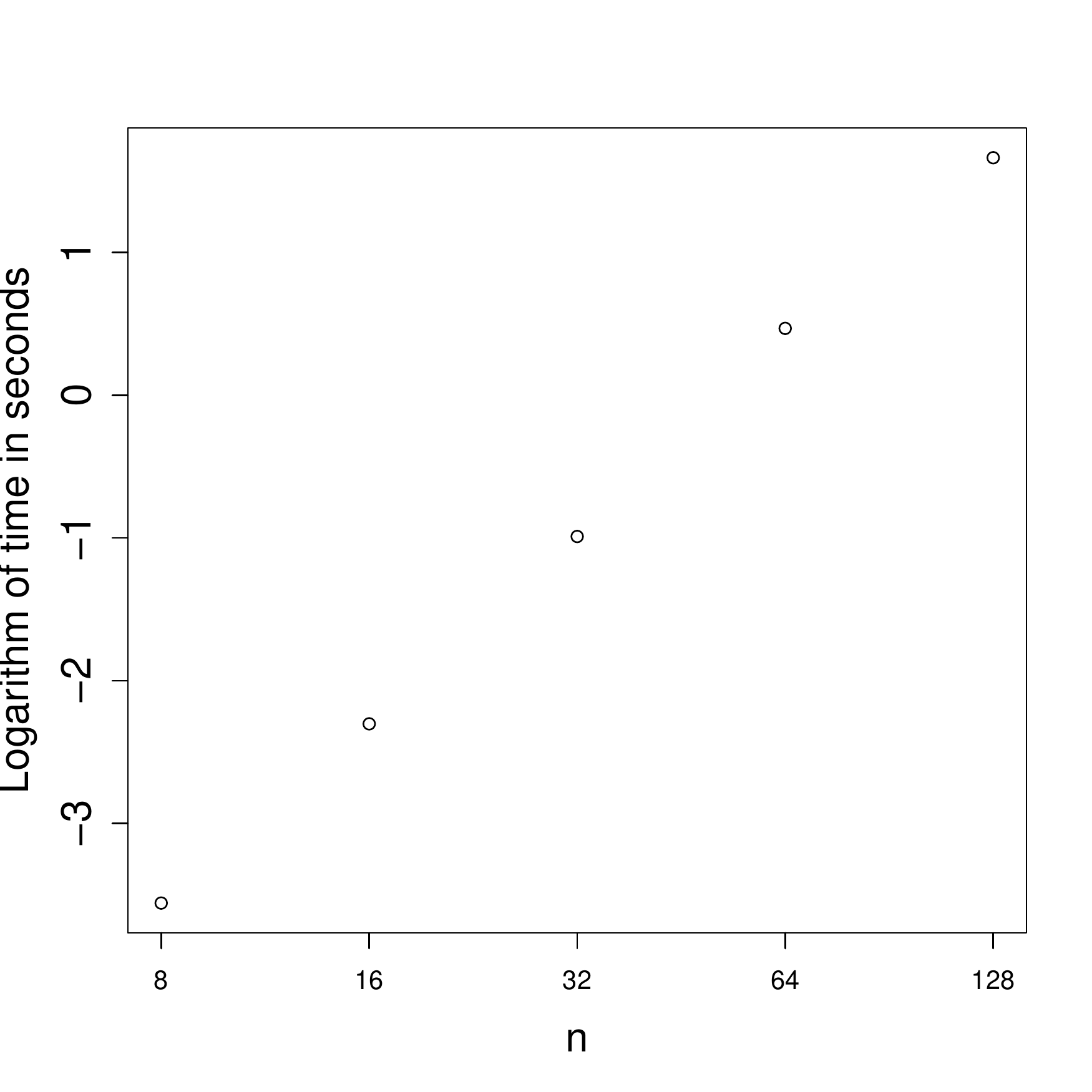} 
		\includegraphics[width=2.7in,height=2.2in]{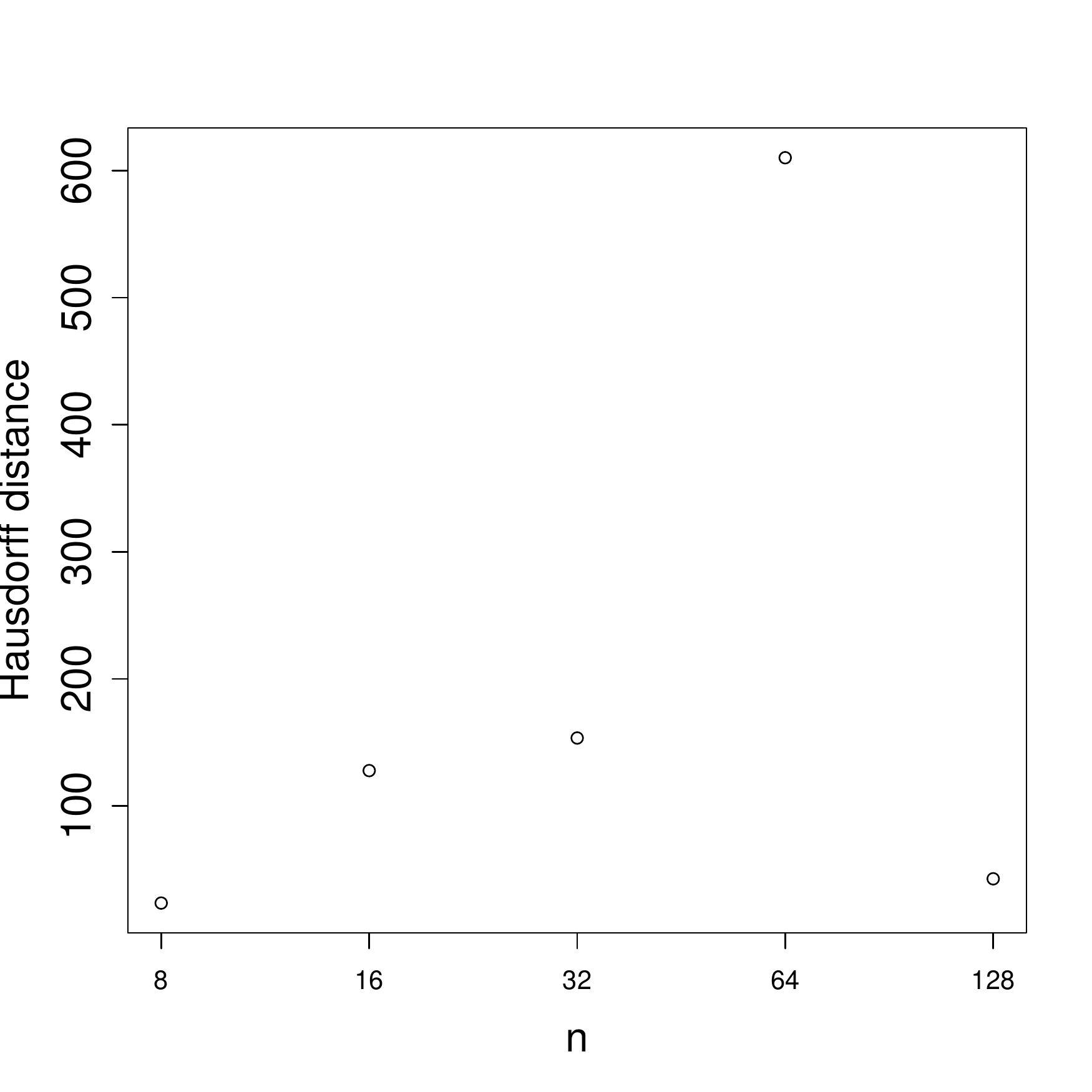} 
		\caption{\label{fig4} Time and Hausdorff distance evaluations, averaging over 50 Monte Carlo simulations, of $\widehat{\Lambda}$  for different values of $n$ for Scenario 5. Here $\sigma =1$.}
	\end{center}
\end{figure}

Finally, with the same implementation details as in Section \ref{sec-numerical} of the paper, we compute  the running time of $\widehat{\Lambda}$ for Scenario 5. The results are shown in Figure \ref{fig4} where we can clearly see a linear trend.

\section{ Non axis-aligned data}\label{sec-nonaxix}
We now briefly discuss how our method could be extended to non axis-aligned data.  Suppose that we are given measurements $\{(x_i,y_i)\}_{i=1}^N$  which are independent copies of a pair of random variables  $(X,Y) \in [0,1]^d\times \mathbb{R}$. Suppose that $n \asymp (N /\log N)^{1/d}$ with  $n \in \mathbb{N}$. Define  
\[
I_{i_1,\ldots\,i_d} \,:=\,  \left[ \frac{i_1-1}{n},  \frac{i_1}{n}\right]  \times \ldots \times \left[ \frac{i_d-1}{n},  \frac{i_d}{n}\right]  
\]
for $i_1,\ldots,i_d \in \{1,\ldots,n\}$.  Then define  $\tilde{y} \in \mathbb{R}^{L_{d,n}}$ as  
\[
\tilde{y}_{i_1,\ldots\,i_d} \,:=\,\frac{1}{ \vert  \{ j\,:\, y_j \in    I_{i_1,\ldots\,i_d} \}     \vert }\sum_{ j\,:\, y_j \in    I_{i_1,\ldots\,i_d}} y_j,
\]
if  $\vert  \{ j\,:\, y_j \in    I_{i_1,\ldots\,i_d} \}   \neq \emptyset$ and otherwise we  set  $  \tilde{y}_{i_1,\ldots\,i_d}  =      \tilde{y}_{i_1^{\prime},\ldots\,i_d^{\prime}} $ where  $I_{i_1^{\prime},\ldots\,i_d^{\prime}} $ is the closest rectangle to $I_{i_1,\ldots\,i_d} $
satisfying    $\vert  \{ j\,:\, y_j \in    I_{i_1^{\prime},\ldots\,i_d^{\prime}} \}   \neq \emptyset$. Both the choice of $n$ and the constrution of $\tilde{y}$ are inspired by ideas from \cite{madrid2020adaptive}.

After having constructed $\tilde{y} \in \mathbb{R}^{L_{d,n}}$ we can then run  DCART and our modified version.

\bibliographystyle{ims}
\bibliography{ref}

\end{document}